\documentclass[12pt]{amsart}

\usepackage{rotating}
\usepackage{tikz}
\usepackage{tikz-cd}
\usepackage{subfig}
\usepackage{adjustbox}
\usepackage[noadjust]{cite}
\usepackage{longtable}
\usepackage{multirow}
\usepackage[colorlinks, linkcolor=blue, allcolors=blue]{hyperref}
\hypersetup{nesting=true,debug=true,naturalnames=true}
\usepackage[margin=1in]{geometry}
\usepackage[percent]{overpic}
\usepackage{thmtools, thm-restate}
\usepackage{enumitem}
\usepackage{CJKutf8}
\tikzcdset{every label/.append style = {font = \small}}

\makeatletter
\def\@cline#1-#2\@nil{%
  \omit
  \@multicnt#1%
  \advance\@multispan\m@ne
  \ifnum\@multicnt=\@ne\@firstofone{&\omit}\fi
  \@multicnt#2%
  \advance\@multicnt-#1%
  \advance\@multispan\@ne
  \leaders\hrule\@height\arrayrulewidth\hfill
  \cr
  \noalign{\nobreak\vskip-\arrayrulewidth}}
\makeatother

\usetikzlibrary{decorations.markings}

\title{Equivariant 4-genera of strongly invertible and periodic knots} 

\author{Keegan Boyle}  
\author{Ahmad Issa}

\address{Department of Mathematics, University of British Columbia, Canada} 
\email{kboyle@math.ubc.ca}
\address{Department of Mathematics, University of British Columbia, Canada} 
\email{aissa@math.ubc.ca}

\newcommand{\Z}{{\mathbb{Z}}}

\newcommand{\bg}{{\widetilde{bg}}}

\newcommand{\BL}{L_b}
\newcommand{\QBL}{L_{qb}}
\newcommand{\lk}{{\widetilde{lk}}}

\newcommand{\hfootnote}[1]{%
  \makebox[0pt][l]{\ \framebox(6,10){\vspace{-4 pt}\footnotemark}}%
  \footnotetext{#1}%
}
\newcommand{\ufootnote}[1]{%
  \makebox[0pt][l]{\framebox(6,10){\vspace{-4 pt}\footnotemark}}%
  \footnotetext{#1}%
}

\newcommand\restr[2]{{
  \left.\kern-\nulldelimiterspace 
  #1 
  \vphantom{\big|} 
  \right|_{#2} 
  }}

\newtheorem{lemma}{Lemma}
\newtheorem{proposition}{Proposition}
\newtheorem{corollary}{Corollary}
\newtheorem{theorem}{Theorem}

\newtheorem*{thm:Montesinos}{Theorem \ref{thm:Montesinos}}
\newtheorem*{thm:siconcordancebound}{Theorem \ref{thm:siconcordancebound}}

\theoremstyle{definition}
\newtheorem{definition}{Definition}[section]
\newtheorem{remark}[definition]{Remark}
\newtheorem{example}[definition]{Example}
\newtheorem{question}[definition]{Question}

\begin{document}
\begin{abstract}
We study the equivariant genera of strongly invertible and periodic knots. Our techniques include some new strongly invertible concordance group invariants, Donaldson's theorem, and the g-signature. We find many new examples where the equivariant 4-genus is larger than the 4-genus. 
\end{abstract}
\maketitle
\vspace{-0.8 cm}
\section{Introduction}
Edmonds showed that periodic knots bound equivariant minimal genus Seifert surfaces \cite{Edmonds}. By contrast, the equivariant 4-genus of a periodic knot, or more generally any type of symmetric knot, is rather subtle, and may vary between different symmetries on the same knot. A \emph{symmetric knot} $(K, \rho)$ is a knot $K \subset S^3$ together with a finite order diffeomorphism $\rho \colon S^3 \rightarrow S^3$ such that $\rho(K) = K$. Given a symmetric knot $(K, \rho)$, we define the \emph{equivariant $4$-genus} $\widetilde{g}_4(K)$ of $K$ as the minimal genus of an orientable, smoothly properly embedded surface $S \subset B^4$ with boundary $K$ for which $\rho$ extends to a diffeomorphism $\widetilde{\rho} \colon (B^4,S) \rightarrow (B^4,S)$ with the same order as $\rho$. If $\rho$ is orientation preserving, the fixed point set of $\rho$ is either empty or the unknot. In the latter case the fixed point set is either disjoint from $K$, in which case we call $(K, \rho)$ \emph{periodic}, or intersects $K$ in exactly two points, in which we case we call $(K, \rho)$ \emph{strongly invertible}. This paper is mainly concerned with the equivariant $4$-genus of periodic and strongly invertible knots and the strongly invertible concordance group. 

As a motivating example, consider the knot $9_{40}$ with strong inversions $\tau_1, \tau_2$ and 2-periodic symmetry $\rho$ shown in Figure \ref{fig:9_40Intro}. The (non-equivariant) 4-genus of $9_{40}$ is 1. In contrast, we show that $\widetilde{g}_4(9_{40},\tau_1) = 2$ using Donaldson's theorem and that $\widetilde{g}_4(9_{40},\rho) = 3$ using an application of the Riemann-Hurwitz formula. It is unknown whether $\widetilde{g}_4(9_{40},\tau_2)$ is 1 or 2. See Appendix \ref{sec:Donaldson_table} for more examples where we bound $\widetilde{g}_4(K)$. 

\begin{figure}
  \begin{overpic}[width=160pt, grid=false]{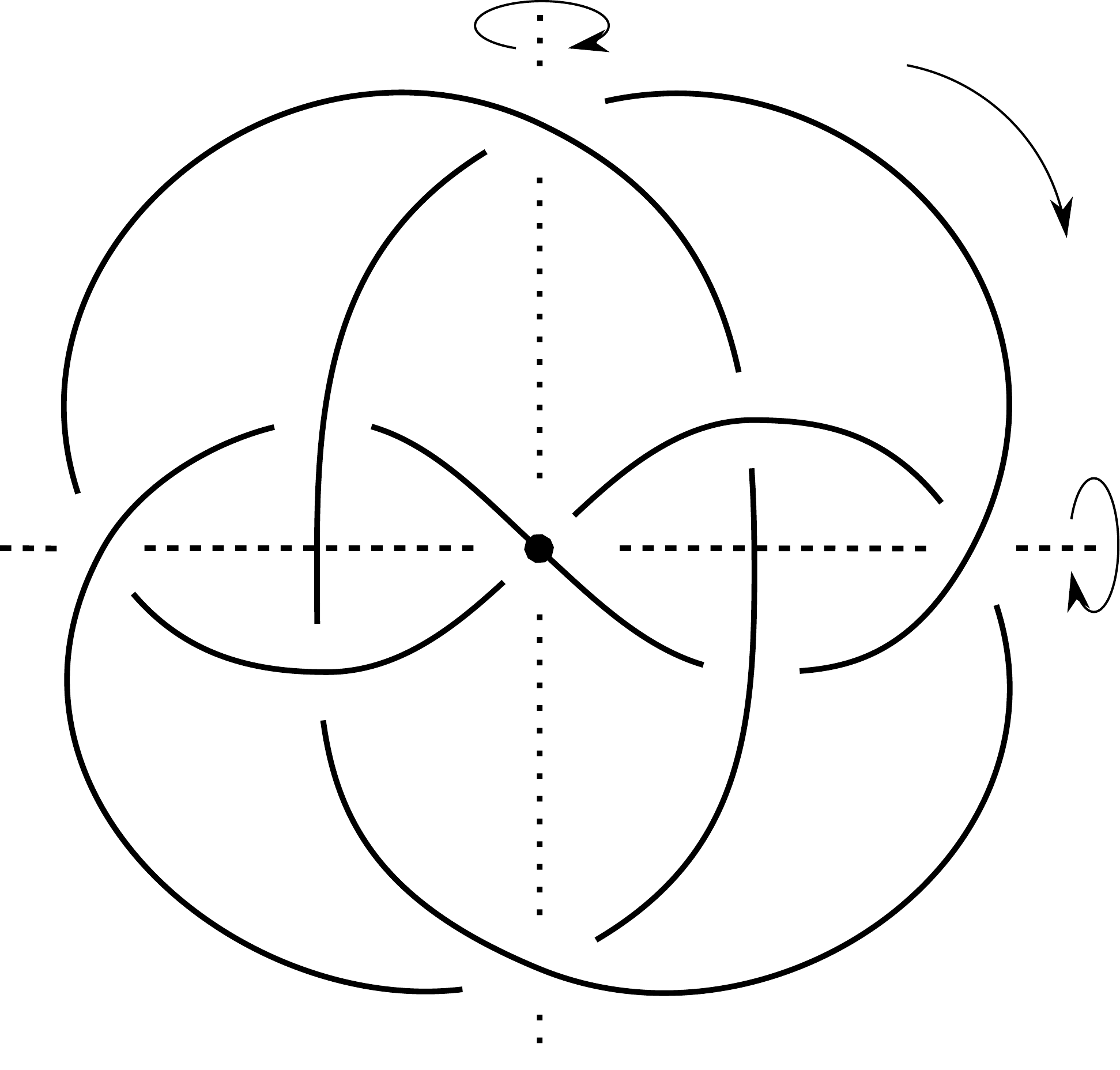}
    \put (56, 92) {$\rho$}
    \put (102, 47) {$\tau_2$}
    \put (94, 85) {$\tau_1$}
  \end{overpic}
\caption{The knot $9_{40}$ with strong inversions $\tau_1, \tau_2$ and 2-periodic symmetry $\rho$. The strong inversion $\tau_1$ is rotation around an axis perpendicular to the plane of the diagram. Here $\widetilde{g}_4(9_{40},\tau_1) = 2$, $\widetilde{g}_4(9_{40},\rho) = 3$, and $\widetilde{g}_4(9_{40},\tau_2) \in \{1,2\}$.}
\label{fig:9_40Intro}
\end{figure}

The above example contrasts with the existing literature which has focused primarily on obstructing slice knots from bounding equivariant slice disks. For a periodic knot $K$, such obstructions have been found using a variety of techniques. In \cite{Naik}, Naik showed that $K$ must have linking number 1 with the axis of symmetry, and obtained further obstructions using equivariant metabolizers and Casson-Gordon invariants. In \cite{Cha-Ko}, Cha and Ko showed that if $K$ is equivariantly slice then certain twists of the quotient knot are slice in a homology 4-ball, and then obtained obstructions by applying Casson-Gordon invariants. In \cite{Davis-Naik}, Davis and Naik used Reidemeister torsion to obtain restrictions on an equivariant version of the Alexander polynomial when $K$ is equivariantly slice.

For strongly invertible knots, equivariant slice disk obstructions have been obtained by Sakuma \cite{Sakuma} using Kojima and Yamasaki's $\eta$-polynomial \cite{KY} and by Dai, Hedden, and Mallick \cite{CorksInvolutions} as a consequence of their cork obstructions coming from Heegaard Floer homology.

In this paper we more generally compare $g_4(K)$ and $\widetilde{g}_4(K)$ for both periodic and strongly invertible knots. Our main results fall into three categories. First, we use Donaldson's theorem to give various examples where $\widetilde{g}_4(K) > g_4(K)$. Second, we use $g$-signatures to define an equivariant version of the signature giving lower bounds on equivariant 4-genera. Third, we define several strongly invertible concordance invariants based on some new topological constructions. 

It has been known since shortly after Donaldson's theorem was proved in the 1980s \cite{Donaldson} that the theorem can often be used to obstruct the existence of a slice disk (see for example \cite{MR2302495}), and more generally can be used to obstruct $g_4(K) = |\sigma(K)|/2$ (see \cite[Proposition 8]{MR3556284}). In the equivariant setting, we show that the symmetry of an equivariant surface in $B^4$ lifts to the double branched cover of $B^4$ over the surface. Donaldson's theorem can often be used to show that $\widetilde{g}_4(K) > |\sigma(K)|/2$ by obstructing the double branched cover of the knot from equivariantly bounding a definite 4-manifold of a certain rank. This obstruction takes the form of the non-existence of a certain equivariant lattice embedding which can be checked combinatorially. Identifying the intersection form of the double branched cover of $B^4$ over a spanning surface with the Gordon-Litherland form, we prove the following theorem.

\begin{restatable}{theorem}{eqembedding}
\label{thm:eq_embedding} Let $K \subset S^3$ be a knot with a periodic or strongly invertible symmetry $\rho \colon S^3 \rightarrow S^3$. Suppose that $F \subset S^3$ is a spanning surface for $K$ for which the Gordon-Litherland pairing $\mathcal{G}_F$ is positive definite and $\rho(F) = F$. If $\widetilde{g}_4(K) = -\sigma(K)/2$ then there is an embedding of lattices $\iota \colon (H_1(F), \mathcal{G}_F) \rightarrow (\Z^k, \mbox{Id})$ such that $\delta \circ \iota = \iota \circ \rho_*$, where $\delta$ is an automorphism of $(\Z^k, \mbox{Id})$ with order$(\delta) = $ order$(\rho)$ and $k = -\sigma(K)/2 + b_1(F)$. In particular, the following diagram commutes.
\end{restatable}
\begin{center}\begin{tikzcd}
  {(H_1(F),\mathcal{G}_F)} & {(\mathbb{Z}^k,\mbox{Id})} \\
  {(H_1(F),\mathcal{G}_F)} & {(\mathbb{Z}^k,\mbox{Id})}
  \arrow["{\rho_*}"', from=1-1, to=2-1]
  \arrow["{\iota}", from=1-1, to=1-2]
  \arrow["{\delta}", from=1-2, to=2-2]
  \arrow["{\iota}"', from=2-1, to=2-2]
\end{tikzcd}
\end{center}

Theorem \ref{thm:eq_embedding} can be applied, for example, when $K$ is a periodic or strongly invertible knot with an alternating diagram in which the symmetry is visible\footnote{An alternating knot with a periodic symmetry of order $> 2$ always has an alternating diagram in which the symmetry is visible; see \cite{PeriodicProjections} or \cite{OddOrder}. For symmetries of order $2$, the situation is more complicated \cite{Involutions}.}. In this case, $F$ can be chosen as the positive definite checkerboard surface. Using a computer, we enumerated all alternating symmetric diagrams through 11 crossings and checked the obstruction from Theorem \ref{thm:eq_embedding}. In all examples we found with $\widetilde{g}_4(K) > |\sigma(K)|/2$ we also have that $g_4(K) = -\sigma(K)/2$ so that $\widetilde{g}_4(K) > g_4(K)$. These periodic and strongly invertible examples are listed in Appendix \ref{sec:Donaldson_table}. 

The best lower bound we can obtain with Theorem \ref{thm:eq_embedding} is $\widetilde{g}_4(K) \geq g_4(K) + 1$. However, we prove the following theorem which in particular shows that $\widetilde{g}_4(K)$ can be much larger than $g_4(K)$.

\begin{restatable}{theorem}{eqgsigineq} \label{thm:eq4g_sig_ineq}
Let $K$ be an $n$-periodic knot with quotient knot $\overline{K}$. Then
\[
\widetilde{g}_4(K) \geq \frac{|n \cdot \sigma(\overline{K}) - \sigma(K)|}{2(n-1)}.
\]
\end{restatable}

We prove this theorem using an equivariant concordance invariant $\widetilde{\sigma}(K)$ which we call the $g$-signature of $K$. The $g$-signature of $K$ is a natural equivariant generalization of the knot signature defined in terms of $g$-signatures of the double cover of $B^4$ branched over an equivariant surface with boundary $K$. For periodic knots, we prove the lower bound $\widetilde{g}_4(K) \geq |\widetilde{\sigma}(K)|/2$. Moreover, if $K$ is periodic with quotient $\overline{K}$ we are able to express $\widetilde{\sigma}(K)$ purely in terms of $\sigma(K)$ and $\sigma(\overline{K})$ (see Theorem \ref{thm:periodicgsig}), from which we obtain Theorem \ref{thm:eq4g_sig_ineq}.

The following theorem gives a family of 2-periodic Montesinos knots $K_n$ for which Theorem \ref{thm:eq4g_sig_ineq} shows that there is an arbitrarily large gap between $\widetilde{g}_4(K_n)$ and $g_4(K_n)$. 

\begin{theorem} \label{thm:Montesinos}
Let $\{K_n\}$ be the family of 2-periodic Montesinos knots\footnote{We follow the convention for Montesinos knots notation from \cite{MR3825858}.} with
\[
K_n = M\left(1;-n, 2n+2,-n\right)
\] 
shown in Figure \ref{fig:UnboundedExampleIntro}, where $n$ is odd and positive. The difference $\widetilde{g}_4(K_n) - g_4(K_n)$ is unbounded. In fact, $\widetilde{g}_4(K_n) = 2n$ and $g_4(K_n) = 1$.
\end{theorem}

\begin{figure}[!htbp]
\begin{overpic}[width=350pt, grid=false]{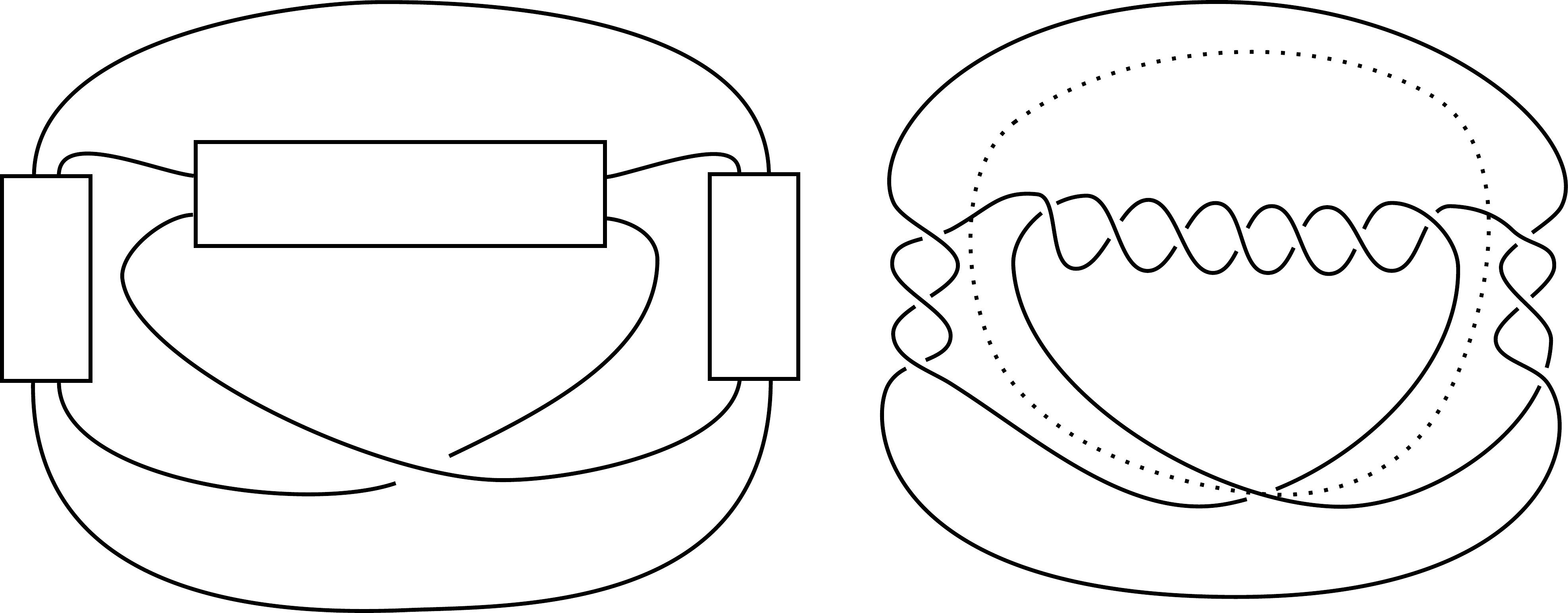}
\put (20.5, 26) {$2n+1$}
\put (1.1, 21) {-$n$}
\put (46.3, 21) {-$n$}
\end{overpic}
\caption{A family of 2-periodic knots $K_n$ (left) for $n$ odd, with $n=3$ shown on the right. The boxes are twist regions with the labeled number of half-twists. The period can be seen by performing a flype on the central crossing region (enclosed by a dotted loop in the right diagram), then rotating the entire diagram by $\pi$ within the plane of the diagram.}
\label{fig:UnboundedExampleIntro}
\end{figure}

Theorem \ref{thm:eq4g_sig_ineq} shows that $\widetilde{g}_4(K_n) \geq n$, but an argument involving the Riemann-Hurwitz formula shows that in fact $\widetilde{g}_4(K_n) \ge 2n$. We also provide an example where Theorem \ref{thm:eq4g_sig_ineq} gives a stronger lower bound than the Riemann-Hurwitz formula argument (see Example \ref{ex:periodic2}).

For a strongly invertible knot $K$, one can attempt to define $\widetilde{\sigma}(K)$ analogously as the $g$-signature of the double cover of $B^4$ branched over an equivariant surface $S$. Unfortunately, this turns out to depend on the choice of $S$. However, if we require that $S$ is a \emph{butterfly surface}, that is, the pointwise fixed arc on $S$ is separating\footnote{The fixed arc plays the role of the butterfly's thorax, separating the surface into symmetric wings.}, then $\widetilde{\sigma}(K)$ no longer depends on the choice of $S$. In fact, we find that $\widetilde{\sigma}(K)$ is an equivariant concordance invariant. This naturally leads us to define the \emph{butterfly 4-genus} $\bg_4(K)$ of $K$ as the minimal genus of a butterfly surface in $B^4$ with boundary $K$. The $g$-signature gives the following lower bound on the butterfly 4-genus.
\begin{theorem} \label{thm:siconcordancebound}
If $K$ is a directed\,\footnote{A direction is a choice of oriented half-axis which we use to define $\widetilde{\sigma}(K)$; see Definition \ref{def:direction} and Figure \ref{fig:ecs}.} strongly invertible knot which bounds a butterfly surface in $B^4$, then $\bg_4(K) \geq |\widetilde{\sigma}(K)|/2$.
\end{theorem}
In Example \ref{ex:invertible}, we give a family of strongly invertible knots $K_n$ for which Theorem \ref{thm:siconcordancebound} shows that $\bg_4(K_n) - g_4(K_n)$ is unbounded. However, we are unable to answer the analogous question about the equivariant 4-genus.
\begin{question}
Is there a family of strongly invertible knots $K_n$ for which $\widetilde{g}_4(K_n) - g_4(K_n)$ is unbounded?
\end{question}

Since an equivariant slice disc is always a butterfly surface, Theorem \ref{thm:siconcordancebound} can be used to obstruct $\widetilde{g}_4(K) = 0$. We obtain further obstructions to $\widetilde{g}_4(K) = 0$ by defining some new invariants of the strongly invertible concordance group $\widetilde{\mathcal{C}}$. This group, first defined by Sakuma \cite{Sakuma}, consists of equivariant concordance classes of \emph{directed} strongly invertible knots, that is, strongly invertible knots along with an orientation on the axis of symmetry and a choice of half-axis. The group operation is given by equivariant connect sum. In \cite{Sakuma}, Sakuma defined a strongly invertible concordance invariant as the Kojima-Yamasaki $\eta$-polynomial \cite{KY} of the two component link consisting of the axis of symmetry and the quotient of two parallel push-offs of a given strongly invertible knot. Similarly, we define new strongly invertible concordance invariants by associating two new 2-component links $\BL(K)$ and $\QBL(K)$ with a directed strongly invertible knot $K$. 

\begin{figure}[!htbp]
\scalebox{.6}{\includegraphics{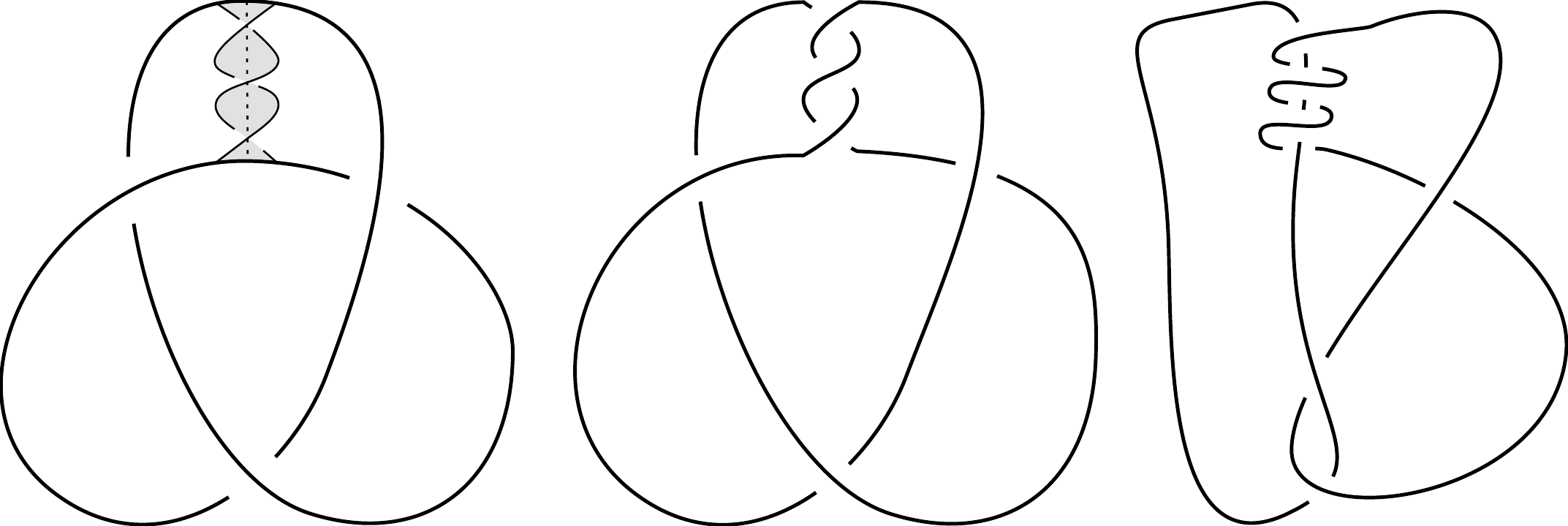}}
\caption{The butterfly link $\BL(3_1^+)$ (center) and quotient butterfly link $\QBL(3_1^+)$ (right) for the directed strong inversion $3_1^+$ (left) on the trefoil. The surgery band is shaded in gray.}
\label{fig:butterflylink}
\end{figure}

First, the \emph{butterfly link} $\BL(K)$ is the 2-periodic link with linking number 0 constructed by performing a band move on $K$ along a band containing the chosen half-axis. Second, the \emph{quotient butterfly link} $\QBL(K)$ is the link consisting of the quotient of $\BL(K)$ by its periodic symmetry, and the axis. See Figure \ref{fig:butterflylink} for an example of $\BL(K)$ and $\QBL(K)$. The terminology ``butterfly link'' refers to the fact that cutting a butterfly Seifert surface for $K$, if it exists, along its pointwise fixed arc gives a surface with boundary $\BL(K)$ (see Figure \ref{fig:8_21}). 

In Section \ref{sec:concordance_invariants}, we use $\BL$ and $\QBL$ to produce several group homomorphisms from the strongly invertible concordance group $\widetilde{\mathcal{C}}$. First, note that if $K$ is equivariantly slice, then $\BL(K)$ and $\QBL(K)$ are (strongly) slice links so that their $\eta$-polynomials vanish. Applying the $\eta$-polynomial to $\BL$ then gives a group homomorphism $\eta \circ \BL\colon \widetilde{\mathcal{C}} \to \mathbb{Z}[t,t^{-1}]$. Second, the linking number of $\QBL(K)$ is an even integer $2\cdot \lk(K)$, and this defines a group homomorphism $\lk:\widetilde{\mathcal{C}} \to \mathbb{Z}$. Finally, we have the following surjective homomorphisms to the smooth concordance group $\mathcal{C}$ and the topological concordance group $\mathcal{C}^{top}$.

\begin{restatable}{theorem}{bandqb}
\label{thm:bandqb}
The maps $\mathfrak{b}:\widetilde{\mathcal{C}} \to \mathcal{C}$ and $\mathfrak{qb}:\widetilde{\mathcal{C}} \to \mathcal{C}^{top}$ defined as follows are surjective group homomorphisms. 
\begin{enumerate}[label=(\roman*)]
\item $\mathfrak{b}(K)$ is the smooth concordance class of one component of $\BL(K)$.
\item $\mathfrak{qb}(K)$ is the topological concordance class of the non-axis component of $\QBL(K)$.
\end{enumerate}
\end{restatable}

Unlike Sakuma's polynomial, our concordance invariants are sensitive to the choice of direction on strong inversions. For a specific example, $\mathfrak{qb}(3_1^-) = m3_1$ and $\mathfrak{qb}(3_1^+) = 0_1$, where $3_1^+$ and $3_1^-$ are the same strong inversion with oppositely chosen half-axes (see Example \ref{ex:trefoil}). As another application, our concordance invariants show that neither strong inversion on $8_9$ is equivariantly slice, even though their Sakuma polynomials vanish and $8_9$ is slice. See Appendix \ref{sec:table} for the invariants of $8_9$ and some other low-crossing examples.


As a final note we show that strongly invertible knots do not always bound butterfly surfaces using the following theorem, which follows immediately from Proposition \ref{prop:arf} and Proposition \ref{prop:linkobs}.

\begin{theorem} \label{thm:butterfly_obstruction}
Let $K$ be a strongly invertible knot. 
\begin{enumerate}[label=(\roman*)]
\item \label{item:arf}If $K$ bounds a butterfly Seifert surface in $S^3$, then the Arf invariant of $K$ is $0$.
\item \label{item:linking}If $K$ bounds a butterfly surface in $B^4$, then $\lk(K) = 0$.
\end{enumerate}
\end{theorem}

\subsection{Organization}
In Section \ref{sec:concordance_group} we give some background on the strongly invertible concordance group. In Section \ref{sec:equivariant_genera} we define the equivariant genera of strongly invertible and periodic knots, and prove Theorem \ref{thm:butterfly_obstruction}\ref{item:arf}. In Section \ref{sec:concordance_invariants} we define the butterfly link $\BL(K)$, the quotient butterfly link $\QBL(K)$, the axis-linking number $\lk(K)$, and we prove Theorem \ref{thm:bandqb} and Theorem \ref{thm:butterfly_obstruction}\ref{item:linking}. In Section \ref{sec:Donaldson} we discuss lifting group actions to the double branched cover, then prove Theorem \ref{thm:eq_embedding}. In Section \ref{sec:gsig} we discuss the $g$-signature $\widetilde{\sigma}(K)$, proving Theorem \ref{thm:eq4g_sig_ineq} and Theorem \ref{thm:siconcordancebound}. In Appendix \ref{sec:table} we provide a table of some low-crossing directed strongly invertible knots and their invariants. In Appendix \ref{sec:Donaldson_table} we provide a table of examples where Theorem \ref{thm:eq_embedding} shows that $\widetilde{g}_4(K) > g_4(K)$.

\subsection{Acknowledgments}
We would like to thank Liam Watson for his encouragement and interest in this project, several helpful conversations, and his comments on an earlier draft. We would also like to thank Makoto Sakuma and Danny Ruberman for some helpful comments.

\section{The strongly invertible concordance group} \label{sec:concordance_group}
The strongly invertible concordance group $\widetilde{\mathcal{C}}$ was first defined by Sakuma \cite{Sakuma}; we recall the basic definitions here. This group is an equivariant version of the usual smooth concordance group $\mathcal{C}$. We also recall the definitions of the forgetful homomorphism $\mathfrak{f}\colon \widetilde{\mathcal{C}} \to \mathcal{C}$ and the doubling homomorphism $\mathfrak{r}\colon \mathcal{C} \to \widetilde{\mathcal{C}}$.
\begin{definition} \label{def:direction}
A strongly invertible knot separates the axis of symmetry into two components, each of which we call a \emph{half-axis}. A \emph{direction} on a strongly invertible knot is a choice of half-axis, and a choice of orientation on the axis. We call a strongly invertible knot along with a choice of direction, a \emph{directed strongly invertible knot} $K$. We do not require a choice of orientation on the knot itself (since it is invertible). The \emph{axis-reverse} of $K$ is the same strongly invertible knot with the opposite choice of orientation on the axis, the \emph{mirror} $mK$ of $K$ is (as usual) given by reversing the orientation on the ambient $S^3$, and the \emph{antipode} $K^-$ of $K$ is given by choosing the other choice of half-axis with the same orientation.
\end{definition}
Note that the two fixed points of a directed strongly invertible knot have a natural ordering given by the choice of oriented half-axis. The point at the beginning of the half-axis is the first fixed point, and the point at the end of the half-axis is the second fixed point. 

\begin{definition}
Two strongly invertible knots $(K,\tau)$ and $(K',\tau')$ are \emph{equivariantly isotopic} or \emph{equivalent} if there is an orientation-preserving homeomorphism $\phi\colon (S^3,K) \to (S^3,K')$ such that $\phi \circ \tau = \tau' \circ \phi$. Furthermore, if $K$ and $K'$ are directed and $\phi$ preserves the chosen oriented half-axis then we say that $(K,\tau)$ and $(K',\tau')$ are \emph{equivalent as directed strongly invertible knots}. 
\end{definition}

\begin{definition} \label{def:equi_connect_sum}
Let $K$ and $K'$ be directed strongly invertible knots. The \emph{equivariant connect sum}, $K\widetilde{\#}K'$ is the directed strongly invertible knot obtained by cutting $K$ at its second fixed point, and $K'$ at its first fixed point, then gluing the two knots and axes in the way that is compatible with the orientations on the axes and choosing the half-axis for $K\widetilde{\#}K'$ as the union of the half-axes for $K$ and $K'$ (see Figure \ref{fig:ecs}). Note that there is no twisting ambiguity along the axis since the knots are strongly invertible. 
\end{definition}
\begin{figure}[!htbp]
\scalebox{.5}{\includegraphics{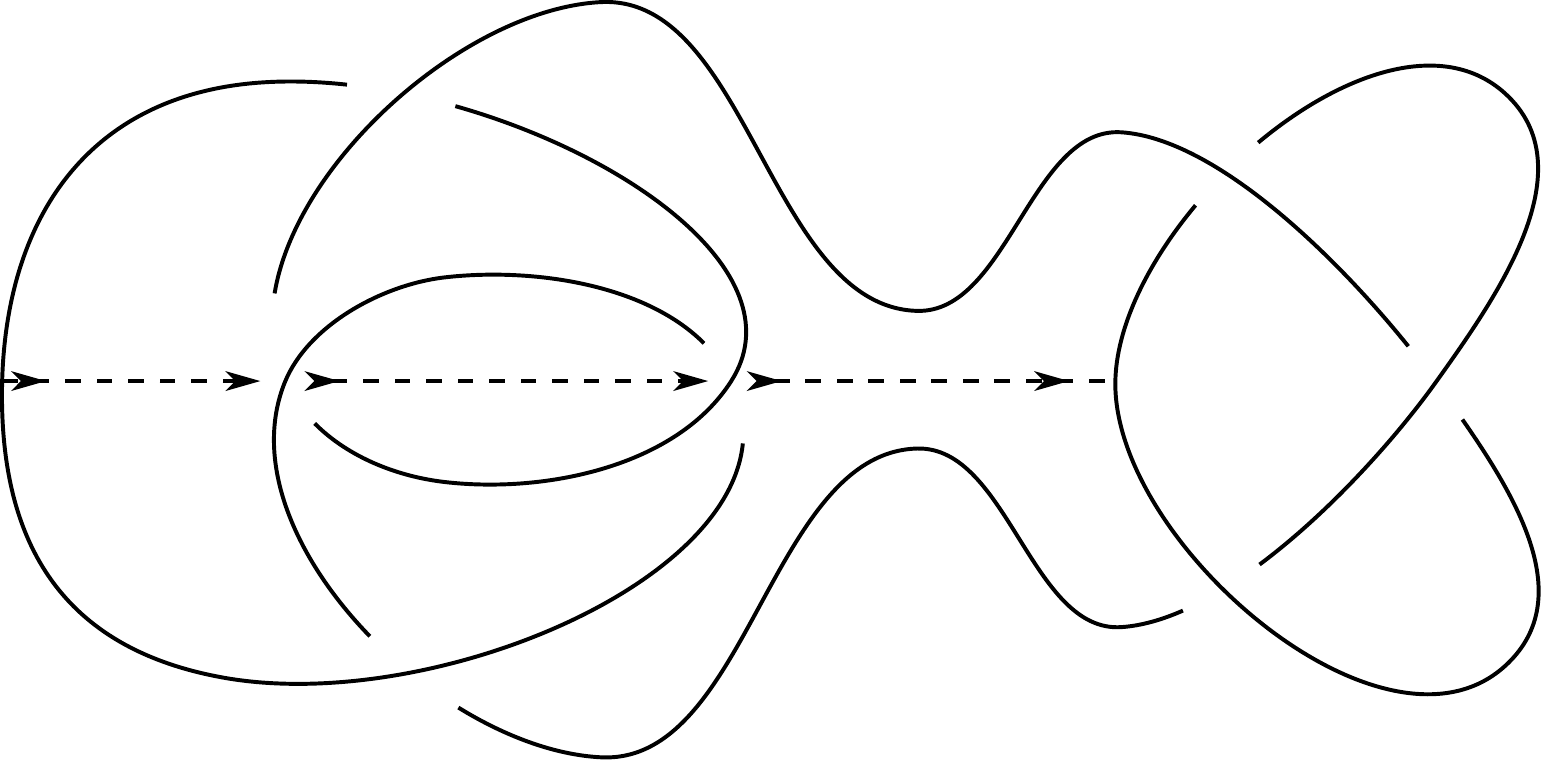}}
\caption{A directed strongly invertible knot constructed by an equivariant connect sum of a directed strong inversion on the figure eight knot and a directed strong inversion on the trefoil.}
 \label{fig:ecs}
\end{figure}

\begin{definition}
Two directed strongly invertible knots $K_1$ and $K_2$ are \emph{(smoothly) equivariantly concordant} if there is a smooth proper embedding $c\colon S^1 \times I \hookrightarrow S^3 \times I$ equivariant with respect to some smooth involution $\tau$ on $S^3 \times I$ such that the following hold.
\begin{enumerate}[label=(\roman*)]
\item $\restr{\tau}{S^3 \times \{i\}}$ is the strong inversion on $K_i$ for $i \in \{1,2\}$.
\item The orientations on the axes of $K_1$ and $K_2$ induce the same orientation on the fixed-point annulus $F$ of $\tau$, and the half-axes of $K_1$ and $K_2$ are contained in the same component of $F \backslash c(S^1 \times I)$.
\end{enumerate}
\end{definition}

We denote by $\widetilde{\mathcal{C}}$ the strongly invertible concordance group, that is, the group with elements given by equivalence classes of directed strongly invertible knots under equivariant concordance, with group operation given by equivariant connect sum. See \cite{Sakuma} for a proof that this group is well-defined. The inverse of a directed strongly invertible knot $K \in \widetilde{\mathcal{C}}$ is the axis-reverse of the mirror of $K$. We use the notation $\mathcal{C}$ for the usual smooth concordance group, and $\mathcal{C}^{top}$ for the topological concordance group.

There is an obvious group homomorphism $\mathfrak{f}\colon \widetilde{\mathcal{C}} \to \mathcal{C}$ given by forgetting about the strong inversion on $K$, which we will refer to as the \emph{forgetful map}. Note that $K$ is isotopic to its reverse (since it is strongly invertible), so we do not need to specify an orientation on $K$ to define $\mathfrak{f}$. There is also a group homomorphism $\mathfrak{r}:\mathcal{C} \to \widetilde{\mathcal{C}}$ induced by $K \mapsto K\#rK$, with the strong inversion on $K \# rK$ exchanging the two summands. Note that this map does not depend on where the connect sum band is attached to $K$. The direction on $\mathfrak{r}(K) = K \# rK$ is defined as follows. Take the edge of the connect summing band contained in $K$ and translate it along the band to coincide with the half-axis contained in the connect summing band (see Figure \ref{fig:k_sum_rk}). This gives an oriented choice of half-axis. See Figure \ref{fig:8_21} for an example of the strong inversion on $\mathfrak{r}(K) = K \# rK$. 

\begin{figure}[!htbp]
  \begin{overpic}[width=230pt, grid=false]{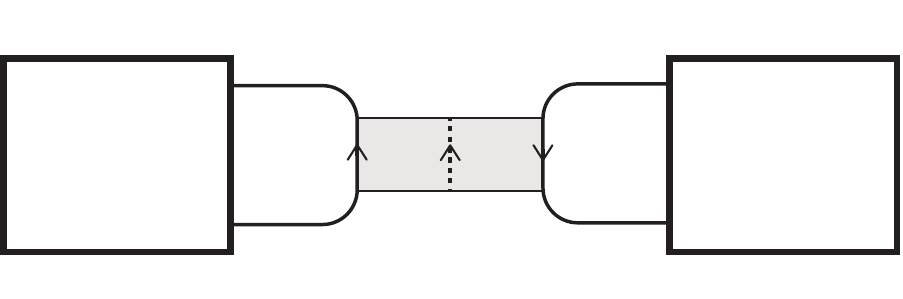}
    \put (10.5, 14) {$K$}
    \put (83, 14) {$rK$}
  \end{overpic}
\caption{A schematic for the chosen directed strong inversion on $\mathfrak{r}(K) = K \# rK$. The shaded band is the connect-summing band for $K \# rK$.}
\label{fig:k_sum_rk}
\end{figure}

\section{Equivariant Genera} \label{sec:equivariant_genera}
In this section, we give a brief introduction to the equivariant 4-genus $\widetilde{g}_4(K)$ and define the butterfly 4-genus $\bg_4(K)$ and butterfly 3-genus $\bg_3(K)$. In Proposition \ref{prop:RH} we give an elementary lower bound on $\widetilde{g}_4(K)$ for periodic knots using the Riemann-Hurwitz formula. In Proposition \ref{prop:arf} we prove that the Arf invariant obstructs strongly invertible knots from bounding a butterfly Seifert surface in $S^3$.

\begin{definition} \label{def:extension}
Let $K$ be a periodic or strongly invertible knot in $S^3$ with $\rho\colon (S^3,K) \to (S^3,K)$ the periodic symmetry or strong inversion. A surface $F \subset B^4$ with $\partial F = K \subset \partial B^4$ is an \emph{equivariant} surface for $(K,\rho)$ if $F$ is connected, smoothly properly embedded in $B^4$, and there exists a diffeomorphism $\overline{\rho}\colon (B^4,F) \to (B^4,F)$ restricting to $\rho$ on $\partial B^4$ with order$(\overline{\rho}) = $ order$(\rho)$. We call $\overline{\rho}$ an \emph{extension} of $\rho$.
\end{definition}
In Definition \ref{def:extension}, $\rho$ can always be extended to $B^4$ by taking the cone of $\rho$, although exotic extensions with a knotted fixed point disk are also possible (see for example \cite{exotic_extension}). The difficulty is in finding an extension which respects a given surface.
\begin{definition} \label{def:equivariant_4genus}
The \emph{equivariant 4-genus} $\widetilde{g}_4(K,\rho)$ of a periodic or strongly invertible knot $(K,\rho)$ is the minimal genus of an orientable equivariant surface for $K$. If $\rho$ is clear from context, we simply write $\widetilde{g}_4(K)$.
\end{definition}

In order to study the equivariant genus, it will be helpful to look at symmetric diagrams which we define precisely here; see Appendix \ref{sec:Donaldson_table} for some examples.
\begin{definition} \label{def:symmetric_diagrams}
Let $(K,\rho)$ be a periodic or strongly invertible knot. A knot diagram for $K$ is 
\begin{enumerate}[label=(\roman*)]
\item \emph{transvergent} if the order of $\rho$ is 2, and $\rho$ acts as rotation around an axis contained within the plane of the diagram, and
\item \emph{intravergent} if $\rho$ acts as rotation around an axis perpendicular to the plane of the diagram. 
\end{enumerate}
In either case, the diagram is called \emph{symmetric}.
\end{definition}

Hiura proved that strongly invertible knots bound equivariant Seifert surfaces by constructing such a surface starting with a transvergent diagram \cite{Ryota}. Another way to see this is to first observe that every periodic or strongly invertible knot admits an intravergent diagram. Applying Seifert's algorithm to such a diagram then produces an equivariant Seifert surface, which shows the following proposition.

\begin{proposition}
Every periodic or strongly invertible knot bounds an equivariant Seifert surface. 
\end{proposition}
This proposition shows that the equivariant 4-genus always exists by equivariantly pushing the Seifert surface into $B^4$. Furthermore, for an alternating diagram, Seifert's algorithm constructs a minimal genus surface \cite{MR99665,MR99664}. Thus for knots with an alternating intravergent diagram, we obtain that $\widetilde{g}_4(K) \leq g_3(K)$. Note however, Hiura gave examples of strongly invertible alternating knots for which there is no equivariant minimal-genus Seifert surface \cite{Ryota}. The following proposition gives another useful way to obtain upper bounds on $\widetilde{g}_4(K)$ from a symmetric diagram.
\begin{proposition} \label{prop:crossing_change_genus}
Let $(K,\rho)$ and $(K',\rho')$ be strongly invertible or periodic knots. If there are symmetric diagrams for $K$ and $K'$ which are related by $n$ equivariant crossing changes, then
\[
|\widetilde{g}_4(K) - \widetilde{g}_4(K')| \leq n.
\]
In particular, if $K'$ is the unknot, then $\widetilde{g}_4(K) \leq n$.
\end{proposition}
\begin{proof}
The $n$ equivariant crossing changes give an equivariant genus $n$ cobordism from $K$ to $K'$, obtained by attaching a pair of equivariant bands for each crossing. Gluing an equivariant surface for $K$ (or $K'$) to this cobordism gives the stated inequality. Finally, the equivariant 4-genus of the unknot is always $0$ (in the strongly invertible case, this follows from \cite[Proposition 2]{MR474265}).
\end{proof}

\subsection{The equivariant 4-genera of a strongly invertible knot}
Given a strongly invertible knot bounding an orientable equivariant surface in $B^4$, there is a fixed arc connecting the fixed points on the knot. If this fixed arc is separating, we call the surface a \emph{butterfly surface}. (See Figure \ref{fig:8_21} for justification of the term butterfly surface.) To be precise, we give the following definition.
\begin{definition}
A \emph{butterfly surface} is an orientable compact connected surface $S$ along with an involution $\rho$ on $S$ such that the pointwise fixed set of $\rho$ contains an arc disconnecting $S$. Let $(K,\tau)$ be a strongly invertible knot. An equivariant surface $S \subset B^4$ with $\partial S = K$ is a \emph{butterfly surface} for $K$ if $(S,\restr{\overline{\tau}}{S})$ is a butterfly surface, where $\overline{\tau}$ is an extension of $\tau$ to $B^4$. If $K$ is directed, an equivariant Seifert surface $S$ for $K$ is a \emph{butterfly Seifert surface} for $K$ if $(S,\restr{\tau}{S})$ is a butterfly surface and the pointwise fixed arc on $S$ coincides with the chosen half-axis.
\end{definition}
We have a few immediate comments about butterfly surfaces. First, butterfly surfaces never contain pointwise fixed circles since the fixed arc separates the surface into two components which are exchanged by the involution. Second, the quotient of a butterfly surface is orientable. Third, the symmetry on a butterfly surface guarantees that its genus is always even. Fourth, the genus determines the equivariant homeomorphism type of a butterfly surface (see for example \cite{Dugger}). Finally, any disk with an orientation-reversing involution is necessarily a butterfly surface since any fixed arc disconnects a disk.
\begin{figure}
\scalebox{1}{\includegraphics{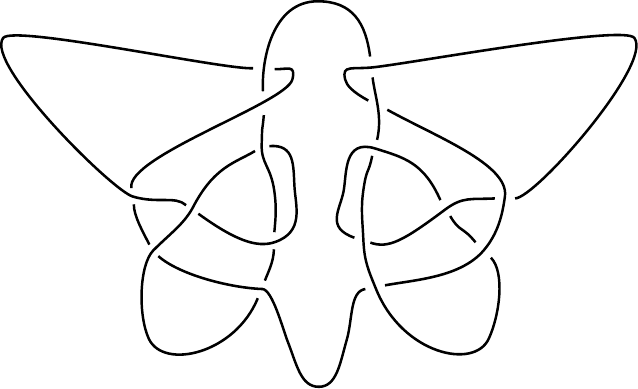}}
\caption{The bounded checkerboard surface for this diagram is a butterfly surface with boundary $8_{21} \# r8_{21}$. The strong inversion is rotation around a vertical axis.}
\label{fig:8_21}
\end{figure}

\begin{definition} \label{def:butterfly_4genus}
The \emph{butterfly 4-genus} $\bg_4(K)$ of a strongly invertible knot $K$ is the minimal genus of a butterfly surface in $B^4$ for $K$. 
\end{definition}
Note that $\widetilde{g}_4(K) \leq \bg_4(K)$ and that $\widetilde{g}_4(K)$ and $\bg_4(K)$ are both strongly invertible concordance invariants. 

\subsection{The equivariant 4-genus of a periodic knot} \label{subsec:periodicgenera}
Unlike in the strongly invertible case, an equivariant surface in $B^4$ with boundary a periodic knot has a 0-dimensional pointwise fixed set. Thus the quotient is a surface with boundary the quotient of the periodic knot. This leads to an inequality on the equivariant 4-genus coming from the Riemann-Hurwitz formula. To begin, we need the following well-known proposition.
\begin{proposition}
\label{prop:quotient_top_b4}
If $\rho:B^4 \to B^4$ is a finite order diffeomorphism with fixed-point set a disk then the quotient of $B^4$ by $\rho$ is homeomorphic to $B^4$.
\end{proposition}
\begin{proof}
By \cite[Corollary II.6.3 and Theorem III.5.4]{Bredon} the quotient of $B^4$ is a simply connected homology 4-ball. Moreover, since the fixed-point set is a disk, the quotient is a topological manifold and hence homeomorphic to $B^4$ by work of Freedman \cite{Freedman}.
\end{proof}
By Edmonds' theorem \cite{Edmonds}, every periodic knot has a minimal genus Seifert surface which is equivariant. However $\widetilde{g}_4(K)$ is not equal to $g_4(K)$ in general. One way to see this is the following well-known consequence of Proposition \ref{prop:quotient_top_b4} and the Riemann-Hurwitz formula; see the proof of \cite[Corollary 3.6]{Naik}. For an example application, see the periodic symmetry on $11_{161}$ in Appendix \ref{sec:Donaldson_table} where $g_4(11_{161}) = 1$ and $\widetilde{g}_4(11_{161}) = 3$.

\begin{proposition}
\label{prop:RH}
Consider an $n$-periodic knot $K$ with quotient $\overline{K}$, and linking number $\lambda$ between $K$ and the axis of symmetry (arbitrarily oriented). Then 
\[
\widetilde{g}_4(K) \geq n\cdot g^{top}_4(\overline{K}) + \dfrac{(n-1)(|\lambda|-1)}{2}.
\]
\end{proposition}
\begin{proof}
Take a minimal genus orientable equivariant surface $S$ in $B^4$ with $\partial S = K$, equivariant with respect to an extension $\rho:B^4 \to B^4$ of the periodic symmetry. By Proposition \ref{prop:quotient_top_b4}, the quotient of $B^4$ by $\rho$ is homeomorphic to $B^4$. Furthermore, since the periodic symmetry preserves the orientation on $K$, it preserves the orientation on $S$. Hence the fixed-point set of $S$ is a finite set of points and the quotient is therefore a surface $\overline{S} \subset B^4$ with $\partial \overline{S} = \overline{K} \subset S^3$.

Now $g(\overline{S}) \geq g^{top}_4(\overline{K})$, and the axis of symmetry intersects $\overline{S}$ in at least $|\lambda|$ points so that $S$ is a branched cover of $\overline{S}$ over at least $|\lambda|$ points. The Riemann-Hurwitz formula gives that
\[
\chi(S) = n\chi(\overline{S}) - (n-1)b,
\]
where $b$ is the number of branch points on $\overline{S}$. Using that $|\lambda| \leq b$ and solving for $g_4(S) = \widetilde{g}_4(K)$ gives the stated inequality.

\end{proof}
\subsection{The equivariant 3-genera of a strongly invertible knot} \label{subsec:3genera}
The main focus of this paper is on the 4-genus, but in this section we take a brief detour to discuss some interesting properties of the equivariant 3-genus and butterfly 3-genus of strongly invertible knots.
\begin{definition}
The \emph{butterfly 3-genus} $\bg_3(K)$ of a directed strongly invertible knot $K$ is the minimal genus of a butterfly Seifert surface for $K$. If no such surface exists, then we write $\bg_3(K) = \infty$. See Figure \ref{fig:52butterfly} for some examples.
\end{definition}

\begin{figure}[!htbp]
\scalebox{.65}{\includegraphics{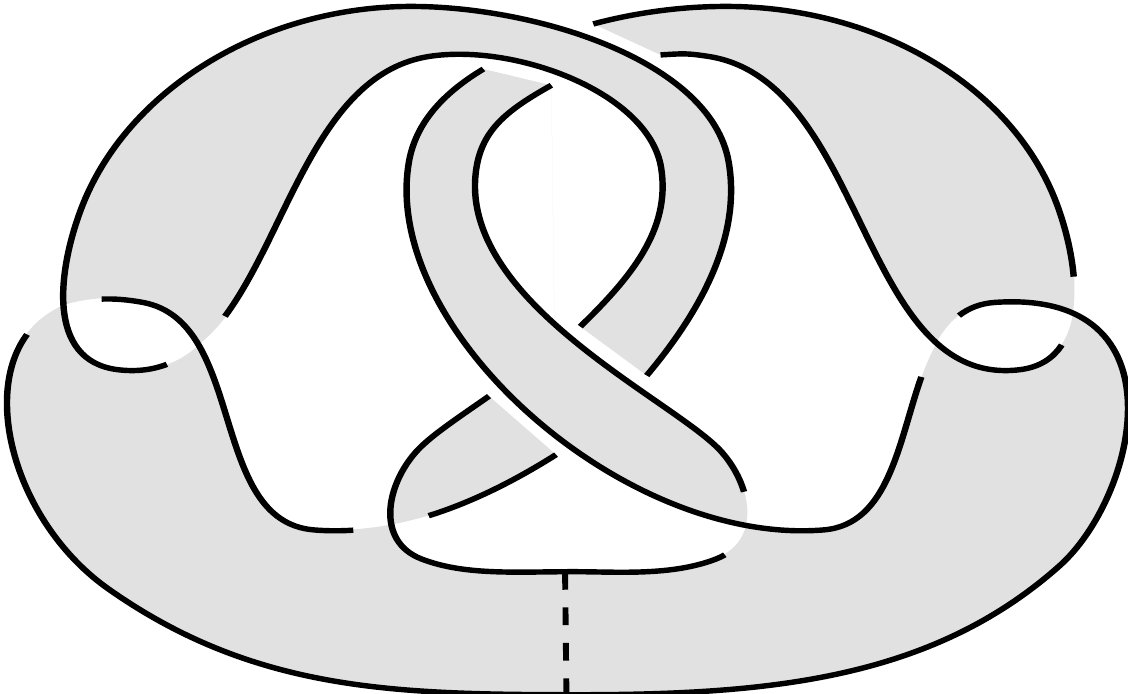}}\ \ \ \scalebox{.6}{\includegraphics{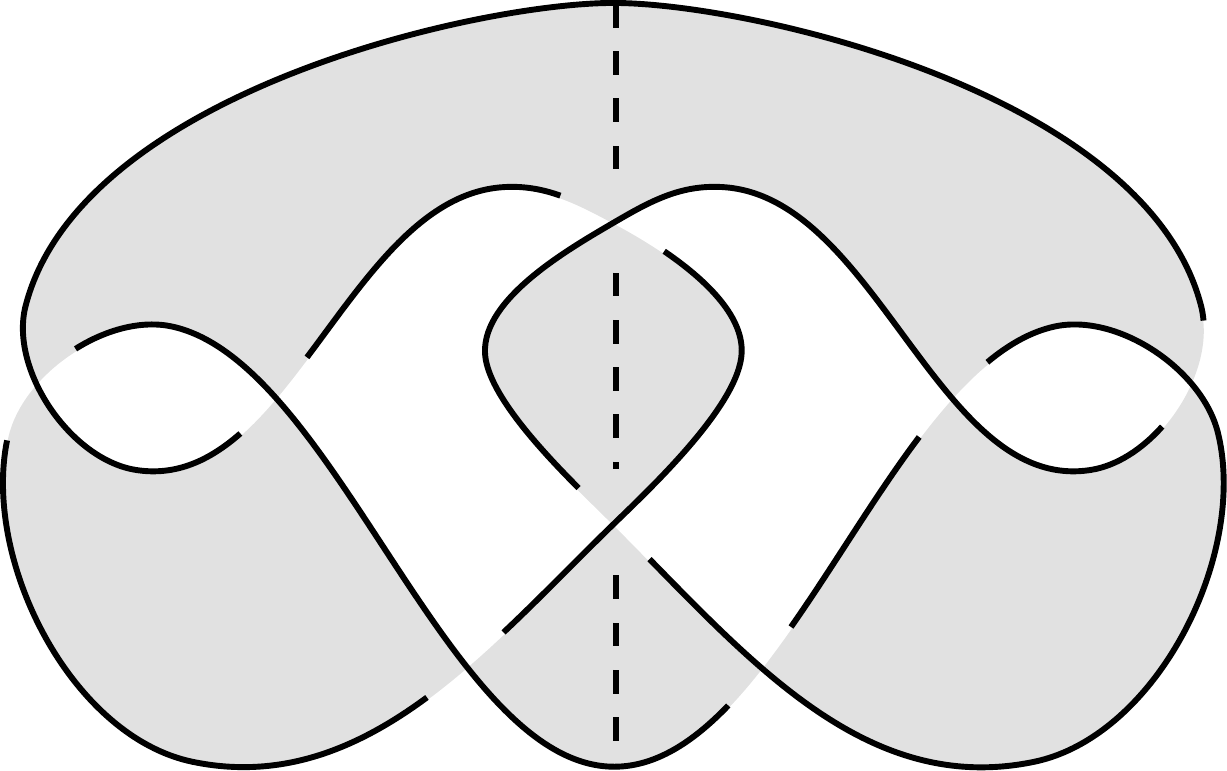}}
\caption{A butterfly Seifert surface for the directed strong inversion $5_2b^-$ (left) on $5_2$, showing that $\bg_3(5_2b^-) \leq 2$, and a butterfly Seifert surface for the directed strong inversion $6_1b^-$ (right) on $6_1$, showing that $\bg_3(6_1b^-) \leq 2$. The fixed arcs are indicated as dotted lines.}
 \label{fig:52butterfly}
\end{figure}

Note that unlike the butterfly 4-genus, the definition of the butterfly 3-genus depends on the direction on the strongly invertible knot. The following proposition shows that butterfly Seifert surfaces do not always exist.

\begin{proposition} \label{prop:arf}
Let $K$ be a directed strongly invertible knot. If the Arf invariant $a(K) = 1$, then $\bg_3(K) = \infty$.
\end{proposition}
\begin{proof}
  Suppose that $\bg_3(K) < \infty$ so that there exists a butterfly surface $S \subset S^3$ with $\partial S = K$. By definition the fixed arc $\alpha$ separates $S$ into two surfaces $S_0$ and $S_1$. Let $K_0 = \partial S_0$ and $K_1 = \partial S_1$. By thinking of each of $S_0$ and $S_1$ as a disk with bands attached we see that away from $\alpha$, we can perform pass moves between the bands of $S_0$ and $S_1$ until they are unlinked (see \cite{Kauffman} for a definition of pass moves). Denote the isotoped unlinked surfaces $S_0'$ and $S_1'$ respectively. A boundary connect sum of $S_0'$ and $S_1'$ near $\alpha$ gives a surface $S'$ which is pass move equivalent to $S$. Hence $a(K) = a(\partial S) = a(\partial S')$ since the Arf invariant is preserved by pass moves \cite{Kauffman}. However, the additivity of the Arf invariant under connect sum implies that $a(\partial S') = a(K_0) + a(K_1) = 0$ where we use that $K_0$ and $K_1$ are isotopic since the strong inversion exchanges them. 
\end{proof}
There are additional obstructions to the existence of a butterfly Seifert surface (see for example Proposition \ref{prop:linkobs}). We also consider the equivariant 3-genus. 

\begin{definition}
The \emph{equivariant 3-genus} $\widetilde{g}_3(K)$ is the minimal genus of an orientable surface $S \subset S^3$ which has boundary $K$, and is equivariant with respect to the strong inversion.
\end{definition}
Edmonds' theorem \cite{Edmonds} states that the equivariant 3-genus of a periodic knot is equal to its (non-equivariant) 3-genus. If an analogous result were to hold for strongly invertible knots then the equivariant 3-genus would be additive under equivariant connect sum. However, Hiura showed that for strongly invertible knots the equivariant 3-genus is sometimes larger than the 3-genus \cite{Ryota}. Nonetheless we are able to show this additivity directly. We can also show the additivity of the butterfly 3-genus.


\begin{proposition}
\label{prop:add}
Given directed strongly invertible knots $K_1$ and $K_2$ and an arbitrary knot $K$, the following hold.
\begin{enumerate}[label=(\roman*)]
  \item \label{item:bg} $\bg_3(K_1)+\bg_3(K_2) = \bg_3(K_1\widetilde{\#}K_2)$.
  \item \label{item:g} $\widetilde{g}_3(K_1) + \widetilde{g}_3(K_2) = \widetilde{g}_3(K_1\widetilde{\#}K_2)$.
  \item $\widetilde{g}_3(K \# rK) = \bg_3(K \# rK) = 2g_3(K)$. 
\end{enumerate}
Here we use the natural directed strong inversion on $K \# rK$; see the end of Section \ref{sec:concordance_group}. Recall that $\widetilde{\#}$ refers to the equivariant connect sum; see Definition \ref{def:equi_connect_sum}.

\end{proposition}
\begin{proof}
First, given a pair of butterfly surfaces for $K_1$ and $K_2$ in $S^3$, the equivariant boundary connect sum of these surfaces is a butterfly surface for $K_1 \widetilde{\#} K_2$ so that $\bg_3(K_1)+\bg_3(K_2) \geq \bg_3(K_1\widetilde{\#}K_2)$. Next, given a minimal genus butterfly surface $\widetilde{S}$ for $K_1 \widetilde{\#} K_2$, consider the equivariant sphere $\Sigma \subset S^3$ decomposing $K_1 \widetilde{\#} K_2$ as a connected sum. After performing an equivariant isotopy we may assume that $\Sigma$ and $\widetilde{S}$ intersect transversely. Then the same type of argument as in the non-equivariant setting (see for example \cite[Chapter 2]{Lickorish}) decomposes $\widetilde{S}$ as the boundary sum of an equivariant Seifert surface $S_1$ for $K_1$ and an equivariant Seifert surface $S_2$ for $K_2$. Since the fixed arc in $\widetilde{S}$ is separating, so are the fixed arcs in $S_1$ and $S_2$. Hence $\bg_3(K_1)+\bg_3(K_2) \leq \bg_3(K_1\widetilde{\#}K_2)$, proving \ref{item:bg}. A similar argument proves \ref{item:g}.

Finally, a boundary sum of any minimal genus Seifert surface for $K$ with itself gives a butterfly Seifert surface of genus $2g_3(K)$ for $K\# rK$.
\end{proof}

\section{Some new strongly invertible concordance invariants} \label{sec:concordance_invariants}
In this section we associate to $K$ a new 2-component link $\BL(K)$ with linking number 0 which we call the butterfly link, and another 2-component link $\QBL(K)$ with even linking number $2 \cdot \lk(K)$ which we call the quotient butterfly link. The integer $\lk(K)$ is a strongly invertible concordance invariant, and we apply Kojima and Yamasaki's eta polynomial (see Definition \ref{def:KY}) to each of $\BL(K)$ and (when $\lk(K) = 0$) $\QBL(K)$ to get strongly invertible concordance invariants. Unlike Sakuma's invariant \cite{Sakuma}, our strongly invertible concordance invariants are able to distinguish some knots from their antipode and to obstruct the strong inversions on $8_9$ from being equivariantly slice. We also define the group homomorphisms $\mathfrak{b}$ and $\mathfrak{qb}$. 

\begin{definition} \label{def:butterfly_link}
Let $K$ be a directed strongly invertible knot. Consider a band which attaches to $K$ at the two fixed points of the strong inversion, and which runs parallel to the chosen half-axis. Performing a band move on $K$ along this band produces a 2-component link, with linking number depending on the number of twists in the band. The \emph{butterfly link} $\BL(K)$ of $K$ is the 2-component 2-periodic link with linking number 0 obtained by this band move on $K$. Both components of $\BL(K)$ are oriented to agree with the orientation on the axis on the strands of $\BL(K)$ parallel to the half-axis. See Figure \ref{fig:butterflylink} in the introduction for an example.
\end{definition}

The following proposition allows us to relate a butterfly surface for a knot $K$ to an equivariant surface for $\BL(K)$.
\begin{proposition}
\label{prop:BL_bounds_surface}
Let $S \subset B^4$ be a butterfly surface for a directed strongly invertible knot $(K,\tau)$. Then the butterfly link $\BL(K)$ bounds a two component orientable surface $S' = S_1 \cup S_2$ in $B^4$ with $g(S_1) = g(S_2) = \frac{1}{2}g(S)$, and the components are exchanged by the extension of $\tau$ to $B^4$ corresponding to $S$. Moreover, there exists an equivariant $(I \times D^2) \subset B^4$ such that $(\partial I) \times D^2 \subset S'$, $I \times (\partial D^2) \subset (S \cup B)$ and $S' = [S \backslash (I \times \partial D^2)] \cup [(\partial I) \times D^2]$, where $B\subset S^3$ is a band with the corresponding band move on $K$ giving $\BL(K)$.
\end{proposition}

\begin{proof}
Let $\overline{\tau}\colon (B^4,S) \to (B^4,S)$ be an extension of $\tau$, with fixed set a disk $F \subset B^4$. Then $S \cap F$ is an arc $\alpha$ separating $F$ into two components. Since $(K,\tau)$ is directed, there is a preferred choice of half-axis in $S^3$ which is contained in a component $D$ of $F - \alpha$. We can further choose an equivariant tubular neighborhood $N(D)$ of $D$ and identify $\partial N(D)$ with the unit tangent bundle over $D$. Then $\partial N(D)\cap S$ is an $S^0$-subbundle of $\partial N(D)|_{\alpha}$. Consider an equivariant band $B$ in $S^3$ containing the chosen half-axis and intersecting $K$ in a neighborhood of each fixed point. We can choose $B$ so that $\partial B - K$ is contained in $\partial N(D)$ and hence $B \cup S$ intersects $\partial N(D)$ in an $S^0$-subbundle of $\partial N(D)|_{\partial D}$. Note that there are furthermore many choices of $B$ corresponding to twists around the axis. Choose $B$ so that the $S^0$-subbundle extends across all of $D$ to an $S^0$-subbundle $E$ of $\partial N(D)$, which naturally has two sections $D_1$ and $D_2$. We can now perform a band move on $K$ along $B$ to obtain a link $L$ which bounds the (disconnected) surface 
\[
[S - N(D)] \cup D_1 \cup D_2.
\]
Then $L$ is a 2-component link with linking number 0 since the components bound disjoint surfaces in $B^4$. Hence $L = \BL(K)$. Finally, we can extend $E$ to an equivariant $I$-subbundle of $N(D)$ to get the desired $(I \times D^2) \subset S^4$ with the properties stated in the proposition.
\end{proof}

\begin{remark}
In Proposition \ref{prop:BL_bounds_surface} we moreover have that if the 2-periodic link $\BL(K)$ bounds an equivariant 2-component surface in $B^4$ (which is reconnected by the surgery band), then reattaching the surgery band gives an equivariant surface for the strongly invertible knot. In particular, if $\BL(K)$ is an equivariantly slice link, then $K$ is an equivariantly slice strongly invertible knot.
\end{remark}
Now that we have defined the butterfly link $\BL(K)$, we can obtain a strongly invertible concordance invariant by applying the Kojima-Yamasaki $\eta$-polynomial, which we now recall. 
\begin{definition}[\cite{KY}] \label{def:KY}
Let $L = K_1 \cup K_2$ be a 2-component link with lk$(K_1,K_2) = 0$ and let $X = S^3 \backslash K_1$. Let $\widetilde{X}$ be the infinite cyclic cover of $X$ so that $H_1(\widetilde{X})$ is naturally a $\mathbb{Z}[t,t^{-1}]$-module where $t$ generates the group of covering transformations. Let $\ell$ be a homological longitude of $K_2$ and let $\widetilde{\ell}$ and $\widetilde{K}_2$ be nearby lifts of $\ell$ and $K_2$ to $\widetilde{X}$. Let $f(t) \in \mathbb{Z}[t,t^{-1}]$ so that $f(t)[\widetilde{\ell}] = 0 \in H_1(\widetilde{X})$. Finally, let $\zeta$ be a 2-chain in $\widetilde{X}$ with $\partial \zeta = f(t)\cdot\widetilde{\ell}$ and $\zeta$ transverse to $t^i\widetilde{K}_2$ for all $i$. Then the \emph{Kojima-Yamasaki $\eta$-polynomial} is
\[
\eta(L) = \frac{1}{f(t)}\sum_{i=-\infty}^{\infty} \mbox{Int}(\zeta,t^i\widetilde{K}_2)t^i \in \mathbb{Z}[t,t^{-1}], 
\] 
where Int$(\zeta,t^i\widetilde{K}_2)$ is the signed count of intersection points between $\zeta$ and $t^i\widetilde{K}_2$ in $\widetilde{X}$. 
\end{definition}
\begin{remark}
The coefficient of $t^i$ in $\eta(L)$ can be thought of as the linking number between $t^i\widetilde{K}_2$ and $\widetilde{K}_2$ in $\widetilde{X}$. In general, $\eta(L)$ depends on the order of $K_1$ and $K_2$. 
\end{remark}
We will use the following properties of the $\eta$-polynomial.
\begin{theorem}[\cite{KY}]
\label{thm:eta_properties}
Let $L = K_1 \cup K_2$ be a 2-component link with lk$(K_1,K_2) = 0$. Then $\eta(L)$ has the following properties.
\begin{enumerate}[label=(\roman*)]
\item $\eta(L)(t) = \eta(L)(t^{-1})$
\item $\eta(L)(1) = 0$
\item $\eta(L)$ does not depend on the orientation of $L$.
\item $\eta(L)$ is a topological link concordance invariant. More precisely, if $L \subset S^3 \times \{0\}$ and $L' \subset S^3 \times \{1\}$ bound a pair of disjoint topologically locally flatly embedded cylinders in $S^3 \times I$, then $\eta(L) = \eta(L')$. Moreover if $L$ is the unlink then $\eta(L) = 0$.
\end{enumerate}
\end{theorem}
\begin{definition}
The \emph{butterfly polynomial} $\eta(\BL(K))$ of a strongly invertible knot $K$ is Kojima-Yamasaki's eta polynomial applied to the butterfly link of $K$. Note that this does not depend on an ordering of the components of the link since $\BL(K)$ is symmetric.
\end{definition}
The following lemma shows that the $\eta$-polynomial is additive in certain circumstances.
\begin{lemma}
\label{lemma:eta_link_additivity}
For $i \in \{1,2\}$, let $L_i = J_i \cup K_i$ be a 2-component link in $S^3$. Let $L = J \cup K$ be a component-wise connect sum of $L_1$ and $L_2$ for which there is a connect-summing sphere $\Sigma \subset S^3$ separating $L_1$ and $L_2$ which simultaneously decomposes $J = J_1 \# J_2$ and $K = K_1 \# K_2$. If lk$(J_i,K_i) = 0$ for $i \in \{1,2\}$, then
\[
\eta(L) = \eta(L_1) + \eta(L_2).
\]
\end{lemma}
\begin{proof}
Let $\widetilde{X}_J$ be the infinite cyclic cover of $S^3 \backslash J$. In this infinite cyclic cover, $\Sigma \backslash J$ lifts to an infinite strip $\widetilde{\Sigma} \cong \mathbb{R} \times (0,1)$. Cutting $\widetilde{X}_J$ along $\widetilde{\Sigma}$ splits $\widetilde{X}_J$ into two components: $\widetilde{X}_{J_1}$ and $\widetilde{X}_{J_2}$, which are the infinite cyclic covers of $S^3 \backslash J_1$ and $S^3 \backslash J_2$ respectively. Let $\widetilde{K}, \widetilde{\ell}_K \subset \widetilde{X}_J$ be nearby lifts of $K$ and the homological longitude of $K$ respectively. Then $\widetilde{K} = \widetilde{K}_1 \# \widetilde{K}_2$ and $\widetilde{\ell}_K = \widetilde{\ell}_{K_1} \# \widetilde{\ell}_{K_2}$, where $\widetilde{K_i}$ and $\widetilde{\ell}_{K_i}$ are nearby lifts of $K_i$ and a homological longitude of $K_i$ for $i \in \{1,2\}$. Let $f(t) \in \mathbb{Z}[t,t^{-1}]$ such that $f(t)[\widetilde{\ell}_{K_1}] = 0 \in H_1(\widetilde{X}_{J_1})$ and $f(t)[\widetilde{\ell}_{K_2}] = 0 \in H_1(\widetilde{X}_{J_2})$. Then for $i \in \{1,2\}$, let $\zeta_i$ be a 2-chain in $\widetilde{X}_{J_i} \subset \widetilde{X}_J$ with $\partial \zeta_i = f(t) \cdot \widetilde{\ell}_{K_i}$ transverse to all translates of $\widetilde{K}_i$. Let $\zeta$ be the 2-chain in $\widetilde{X}_J$ with $\partial \zeta = f(t) \cdot \widetilde{K}$ which is obtained by gluing $\zeta_1$ and $\zeta_2$ together along bands realizing translates of the connect sum $\widetilde{\ell}_K = \widetilde{\ell}_{K_1} \# \widetilde{\ell}_{K_2}$. Then the points of intersection between $\zeta$ and $t^i\widetilde{K}$ are precisely the points of intersection between $\zeta_1$ and $t^i\widetilde{K}_1$ and between $\zeta_2$ and $t^i\widetilde{K}_2$. Thus for all $i$ we have that
\[
\mbox{Int}(\zeta,t^i\widetilde{K}) = \mbox{Int}(\zeta_1,t^i\widetilde{K}_1) + \mbox{Int}(\zeta_2,t^i\widetilde{K}_2).
\]
Hence $\eta(L) = \eta(L_1) + \eta(L_2)$.
\end{proof}
As a consequence of this lemma, we obtain the following proposition. 
\begin{proposition}
\label{prop:BL}
The butterfly polynomial $\eta(\BL(-))\colon \widetilde{\mathcal{C}} \to \mathbb{Z}[t,t^{-1}]$ is a group homomorphism from the strongly invertible concordance group to the additive group of (symmetric) Laurent polynomials.
\end{proposition}
\begin{proof}
Let $K$ and $K'$ be directed strongly invertible knots. It is enough to check that if $K$ is equivariantly slice, then $\eta(\BL(K)) = 0$, and that 
\[
\eta(\BL(K\widetilde{\#}K')) = \eta(\BL(K)) + \eta(\BL(K')).
\]
If $K$ is slice, then the 2-component link $\BL(K)$ is also slice by Proposition \ref{prop:BL_bounds_surface}. Hence by Theorem \ref{thm:eta_properties}, $\eta(\BL(K)) = 0$. Next, the equivariant connect-summing sphere for $K\widetilde{\#}K'$ realizes $\BL(K\widetilde{\#}K')$ as a component-wise connect sum of $\BL(K)$ and $\BL(K')$ as in Lemma \ref{lemma:eta_link_additivity}. Hence $\eta(\BL(K\widetilde{\#}K')) = \eta(\BL(K)) + \eta(\BL(K')).$
\end{proof}
Using the butterfly link, we define some other interesting invariants.
\begin{definition} \label{def:axis_linking}
The \emph{axis-linking} number $\lk(K)$ of $K$ is the linking number between one component of $\BL(K)$ and the axis of symmetry. 
\end{definition}

\begin{definition} \label{def:quotient_butterfly_link}
The \emph{quotient butterfly link} $\QBL(K)$ of $K$ is the 2-component link consisting of the quotient of $\BL(K)$ by the 2-periodic action and the axis in the quotient.
\end{definition}
Note that the linking number between the components of $\QBL(K)$ is $2 \cdot \lk(K)$. The following proposition shows that $\lk(K)$ provides an obstruction to the existence of a butterfly surface in $B^4$ with boundary $K$.
\begin{proposition}
\label{prop:linkobs}
Let $K$ be a directed strongly invertible knot. If $K$ bounds a butterfly surface of genus $g$ in $B^4$, then $\QBL(K)$ bounds an orientable two component surface topologically locally flatly embedded in $B^4$ consisting of a disk with boundary the axis and a genus $g/2$ surface bounding the other component of $\QBL(K)$. In particular, $\lk(K) = 0$.
\end{proposition}
\begin{proof}
Suppose $K$ bounds a genus $g$ butterfly surface in $B^4$, invariant with respect to an extension $\tau\colon B^4 \to B^4$ of the strong inversion and let $F \subset B^4$ be the disk of fixed points. Then by Proposition \ref{prop:BL_bounds_surface}, the 2-component link $\BL(K)$ bounds an orientable two component surface $S = S_1 \cup S_2$, with the components exchanged by $\tau$. This surface is disjoint from $F$. Taking the quotient of $(B^4,S,F)$ by $\tau$ gives a pair of orientable surfaces with boundary $\QBL(K)$ in a topological 4-ball by Proposition \ref{prop:quotient_top_b4}. The quotient of $F$ is a disk, and the quotient of $S$ is a genus $g/2$ surface. Hence the components of $\QBL(K)$ have linking number 0 so that $\lk(K) = 0$.
\end{proof}

\begin{proposition}
\label{prop:QBL_lk_additive}
The axis-linking number $\lk(-):\widetilde{\mathcal{C}} \to \mathbb{Z}$ and the eta polynomial of the quotient butterfly link $\eta(\QBL(-))\colon \mbox{ker}(\lk) \to \mathbb{Z}[t,t^{-1}]$ are group homomorphisms.
\end{proposition}
\begin{proof}
Let $K$ and $K'$ be directed strongly invertible knots. It is enough to check the following. 
\begin{enumerate}
\item If $K$ is equivariantly slice then $\lk(K) = 0$ and $\eta(\QBL(K)) = 0$. 
\item $\lk(K\widetilde{\#}K') = \lk(K) + \lk(K')$. 
\item If $\lk(K) = \lk(K') = 0$, then $\eta(\QBL(K\widetilde{\#}K')) = \eta(\QBL(K)) + \eta(\QBL(K'))$.
\end{enumerate}
To show (1), suppose $K$ is equivariantly slice. By Proposition \ref{prop:linkobs}, $\lk(K) = 0$ and $\QBL(K)$ bounds a pair of disjoint disks in $B^4$. Hence by Theorem \ref{thm:eta_properties}, $\eta(\QBL(K)) = 0$. To show (2) and (3), consider the quotient of the equivariant connect-summing sphere for $K\widetilde{\#}K'$ by $\tau$. This quotient sphere decomposes $\QBL(K\widetilde{\#}K')$ as a component-wise connect sum of $\QBL(K)$ and $\QBL(K')$ as in Lemma \ref{lemma:eta_link_additivity}. This decomposition implies $\lk(K) + \lk(K') = \lk(K\widetilde{\#}K')$, and by Lemma \ref{lemma:eta_link_additivity} that $\eta(\QBL(K\widetilde{\#}K')) = \eta(\QBL(K)) + \eta(\QBL(K'))$ provided $\lk(K) = \lk(K') = 0$.
\end{proof}

\begin{example}
Consider the knot $K$ = K13n1496 with the strong inversion shown in Figure \ref{fig:13n1496}. A straightforward computation shows that $K$ has trivial Alexander polynomial, which by work of Freedman \cite[Theorem 1.13]{Freedman} implies that $K$ is topologically slice. On the other hand, Rasmussen's $s$-invariant is 2 so that $K$ is not smoothly slice. With respect to the strong inversion, $\eta(\QBL(K)) = 2t^{-2}-4+2t^2 \neq 0$. Hence by a similar argument as in Proposition \ref{prop:linkobs}, a topological slice disk can not be made equivariant with respect to any diffeomorphism $B^4 \to B^4$ extending the strong inversion on $K$. In particular, Freedman's construction of a topological slice disk can not be made equivariantly with respect to a diffeomorphism extending this strong inversion. 
\end{example}
\begin{figure}[!htbp]
\scalebox{.6}{\includegraphics{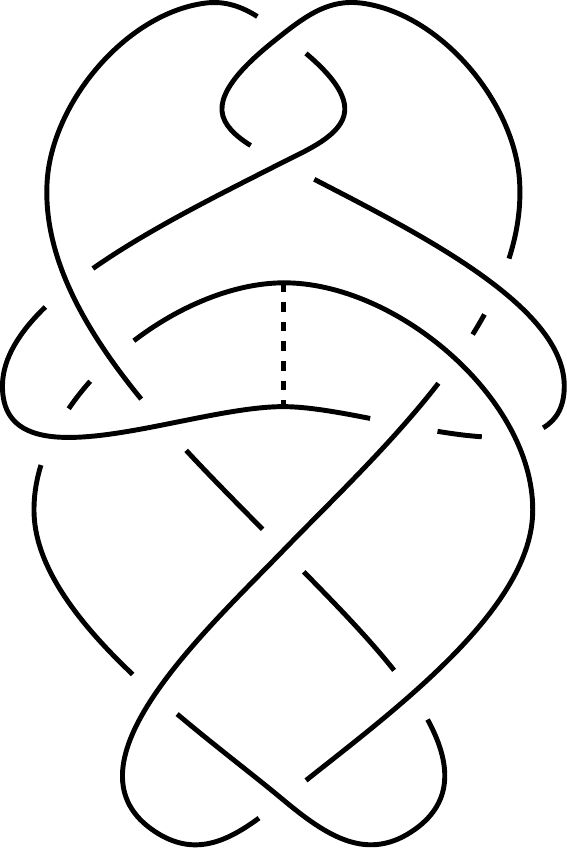}}
\caption{A directed strong inversion K13n1496$^+$.}
\label{fig:13n1496}
\end{figure}

\bandqb*

\begin{proof}
It suffices to show that if $K$ is equivariantly smoothly slice then $\mathfrak{b}(K)$ is smoothly slice and $\mathfrak{qb}(K)$ is topologically slice, that $\mathfrak{b}(K_1\widetilde{\#}K_2)$ is isotopic (and hence smoothly concordant) to $\mathfrak{b}(K_1)\#\mathfrak{b}(K_2)$, and that $\mathfrak{qb}(K_1 \widetilde{\#} K_2)$ is isotopic (and hence topologically concordant) to $\mathfrak{qb}(K_1) \# \mathfrak{qb}(K_2)$.

Suppose $K$ is equivariantly slice. Then as we have seen in the proof of Proposition \ref{prop:BL}, each component of $\BL(K)$ bounds a smooth disk in $B^4$. In particular $\mathfrak{b}(K)$ is smoothly slice. Furthermore, the two components bound a disjoint equivariant pair of disks, so that the image of one disk in the quotient has no self-intersection points. Then since the quotient is a topological $B^4$ by Proposition \ref{prop:quotient_top_b4}, $\mathfrak{qb}(K)$ is topologically slice. 

Now consider $K_1 \widetilde{\#} K_2$ and an equivariant sphere $S$ separating $K_1$ and $K_2$. Then the same $S$ realizes $\mathfrak{b}(K_1 \widetilde{\#} K_2)$ as $\mathfrak{b}(K_1) \# \mathfrak{b}(K_2)$. Furthermore, since $S$ is equivariant the quotient realizes $\mathfrak{qb}(K_1 \widetilde{\#} K_2)$ as $\mathfrak{qb}(K_1) \# \mathfrak{qb}(K_2)$. 
\end{proof}
\begin{remark}
Note that $\mathfrak{b}(K)$ can be easily computed directly as the union of the chosen half-axis and (either) half of the strongly invertible knot. For example in Figure \ref{fig:52butterfly}, $\mathfrak{b}(5_2b^-) = 3_1$ and $\mathfrak{b}(6_1b^-) = 4_1$ are both immediately visible from the diagrams for the butterfly surfaces.
\end{remark}
Recall the group homomorphism $\mathfrak{r}\colon \mathcal{C} \to \widetilde{\mathcal{C}}$ given by $\mathfrak{r}(K) = K \# rK$ defined at the end of Section \ref{sec:concordance_group}. Note that $\mathfrak{b}(\mathfrak{r}(K)) = K$ and $\mathfrak{qb}(\mathfrak{r}(K)) = K$. This implies that $\mathfrak{r}$ is injective, and $\mathfrak{b}$ and $\mathfrak{qb}$ may be thought of as retracts from $\widetilde{\mathcal{C}} \to \mathcal{C}$ and $\widetilde{\mathcal{C}} \to \mathcal{C}^{top}$ respectively. 

The invariants $\eta(\BL(K)), \eta(\QBL(K)),$ and $\lk(K)$ behave nicely under taking the axis-reverse (reversing the orientation on the axis) or mirror of $K$ (see Proposition \ref{prop:reverse_mirror_invariants}), but seem to have subtle behavior when taking the antipode (taking the other choice of half-axis). This is an advantage over Sakuma's polynomial however, since these invariants sometimes distinguish antipodes in the strongly invertible concordance group as seen in the following example.

\begin{example} \label{ex:trefoil}
Consider the two elements of the strongly invertible concordance group $3_1^+$ and $3_1^-$ given by the two choices of half-axis for the unique strong inversion on the right-handed trefoil (see the table in Appendix \ref{sec:table}). We have that $\mathfrak{qb}(3_1^+)$ is the unknot, but $\mathfrak{qb}(3_1^-)$ is the left-handed trefoil. Since the unknot and trefoil are not concordant in the topological concordance group, $3_1^+$ and $3_1^-$ are not strongly invertibly concordant by Theorem \ref{thm:bandqb}. In particular $3_1^+ \widetilde{\#} $inv$(3_1^-)$, where inv$(3_1^-)$ is the inverse of $3_1^-$ in $\widetilde{\mathcal{C}}$, is slice but not equivariantly slice. 
\end{example}

\begin{proposition} 
\label{prop:reverse_mirror_invariants}
Let $K$ be a directed strongly invertible knot with mirror $mK$ and axis-reverse $aK$. (The axis-reverse is $K$ with the orientation on the axis reversed.) Then
\begin{enumerate}[label=(\roman*)]
\item $\eta(\BL(K)) = \eta(\BL(aK)) = -\eta(\BL(mK))$, 
\item $\eta(\QBL(K)) = \eta(\QBL(aK)) = -\eta(\QBL(mK))$ provided $\eta(\QBL(K))$ is defined, and
\item $\lk(K) = \lk(aK) = -\lk(mK)$.
\end{enumerate}
\end{proposition}
\begin{proof}
Theorem \ref{thm:eta_properties} implies that $\eta(\BL(K)) = \eta(\BL(aK))$ and $\eta(\QBL(K)) = \eta(\QBL(aK))$. Since reversing the orientation of both components of the link does not change the linking number, we also have that $\lk(K) = \lk(aK)$. Now since $\eta\circ\BL$ is additive by Proposition \ref{prop:BL}, we know that $\eta(\BL(K)) = -\eta(\BL(maK)) = -\eta(\BL(mK))$ (since $maK$ is the inverse of $K$ in $\widetilde{\mathcal{C}}$), and similarly for $\lk$ and $\eta\circ\QBL$ by Proposition \ref{prop:QBL_lk_additive}.
\end{proof}

\section{Equivariant Donaldson's theorem obstruction} \label{sec:Donaldson}
In this section we prove Theorem \ref{thm:eq_embedding} and use it to give examples of strongly invertible and periodic knots where the equivariant $4$-genus is different from the (non-equivariant) $4$-genus (see Appendix \ref{sec:Donaldson_table} for a table of examples). To do so we consider the question of when a $3$-manifold with a symmetry bounds an equivariant positive-definite $4$-manifold and use Donaldson's theorem to obtain an obstruction in some cases. We start by showing that certain symmetries of $B^4$ can be lifted to $\Sigma(B^4,F)$, the double branched cover of $B^4$ over an equivariant surface $F$. 

\begin{proposition}\label{prop:symmetry_lifts} Let $F \subset S^4$ be a closed connected (not necessarily orientable) surface and let $\rho \colon S^4 \rightarrow S^4$ be an orientation-preserving finite order diffeomorphism such that $\rho(F) = F$. Suppose that $\rho$ has a fixed point in $S^4 \backslash F$. Let $q \colon \Sigma(S^4,F) \rightarrow S^4$ be the double branched covering projection. Then $\rho$ lifts to a smooth map $\widetilde{\rho} \colon \Sigma(S^4,F) \rightarrow \Sigma(S^4,F)$, that is, $q \circ \widetilde{\rho} = \rho \circ q$ and order$(\widetilde{\rho}) =$ order$(\rho)$. Furthermore there are exactly two such lifts, and if $A$ is the fixed-point set of $\rho$ and $C$ is a component of $A \backslash F$, then exactly one of these two lifts fixes $q^{-1}(C)$ pointwise.
\end{proposition}

\begin{proof} By the Equivariant Tubular Neighborhood theorem there is a $\rho$-equivariant tubular neighborhood of $F$ which we denote by $N(F) \subset S^4$. Let $X = S^4 \backslash \mbox{int}(N(F))$ be the exterior of $F$ in $S^4$. Denote by $\widetilde{X}$ the two-fold cyclic cover of $X$ corresponding to the kernel of $\pi_1(X) \to H_1(X;\mathbb{Z}/2\mathbb{Z}) \cong \mathbb{Z}/2\mathbb{Z}$. Let $s \in X$ be a fixed point under $\rho$ and let $\widetilde{s} \in q^{-1}(s) \subset \widetilde{X}$ be a lift of $s$ to $\widetilde{X}$. We implicitly use $s$ (resp. $\widetilde{s}$) as the basepoint of $\pi_1(X)$ (resp. $\pi_1(\widetilde{X})$).

Let $G \leq \pi_1(X)$ denote the image of $\pi_1(q) \colon \pi_1(\widetilde{X}) \rightarrow \pi_1(X)$. Since $G$ is the unique index $2$ (and hence normal) subgroup of $\pi_1(X)$, it is a characteristic subgroup. Hence, the image of $\pi_1(\rho \circ q) \colon \pi_1(\widetilde{X}) \rightarrow \pi_1(X)$ is also $G$. By the covering space lifting property, since $\mbox{Im}\ \pi_1(\rho \circ q) \subseteq \mbox{Im}\ \pi_1(q)$, there exists a unique map $\widetilde{\rho} \colon (\widetilde{X}, \widetilde{s}) \rightarrow (\widetilde{X}, \widetilde{s})$ such that $q \circ \widetilde{\rho} = \rho \circ q$.

The tubular neighborhood $N(F) \subset S^4$ has the structure of a $D^2$-bundle over $F$. The branched cover decomposes as $\Sigma(S^4,F) = \widetilde{X} \cup_{\partial \widetilde{X}} W$, where $W = q^{-1}(N(F))$ is naturally a $D^2$-bundle over $F \cong q^{-1}(F)$ by lifting the fibers of $N(F)$ to $\Sigma(S^4,F)$. It suffices to show that $\restr{\widetilde{\rho}}{\partial \widetilde{X}}$ extends to a map on $W$ such that $q \circ \widetilde{\rho} = \rho \circ q$ holds on all of $\Sigma(S^4,F)$. By abuse of notation we also denote the zero section of $W$ by $F$. Denote by $\restr{W}{x}$ the $D^2$ fiber of $W$ over $x \in F \subset W$. By the equivariant tubular neighborhood theorem, the action of $\rho$ on $N(F) \subset S^4$ is fiber preserving. Hence, the map $\widetilde{\rho}$ restricts to $\partial\restr{W}{x} \rightarrow \partial\restr{W}{\rho(x)}$ and this can be extended over the disk $\restr{W}{x}$ to a map $\restr{W}{x} \rightarrow \restr{W}{\rho(x)}$. This can be done smoothly over all $x \in F \subset W$ so that $q \circ \widetilde{\rho} = \rho \circ q$.

To see that order$(\widetilde{\rho}) = $ order$(\rho)$, note by induction that $q \circ \widetilde{\rho}^{\, i} = \rho^i \circ q$ for $i \geq 1$. In particular this implies that $\widetilde{\rho}^{\, i} \neq $ Id for $1 \leq i <$ order$(\rho)$ and $q \circ \widetilde{\rho}^{\, n} = q$ for $n = $ order$(\rho)$. Hence $\widetilde{\rho}^{\, n}$ is a branched covering transformation fixing the point $\widetilde{s}$ which is disjoint from the branch set $q^{-1}(F)$. Thus $\widetilde{\rho}^{\, n} = $ Id so order$(\widetilde{\rho}) = n = $ order$(\rho)$.

The above construction produces the unique lift $\widetilde{\rho}$ fixing $\widetilde{s}$. Composing with the (branched) deck transformation gives another lift exchanging the two points in $q^{-1}(s)$, and these are the only two lifts of $\rho$. Let $C$ be a component of $A \backslash F$ where $A$ is the fixed-point set of $\rho$, and let $\widetilde{C} = q^{-1}(C)$. We argue that there is a unique lift $\widetilde{\rho}:\Sigma(S^4,F) \to \Sigma(S^4,F)$ of $\rho$ which fixes $\widetilde{C}$ pointwise. To see the existence of such a lift, first choose $s \in C$ and let $\widetilde{s}$ be a lift of $s$ to $\widetilde{C}$ in the construction above. Note that $\widetilde{\rho}|_{\widetilde{C}}$ is a deck transformation of the two-fold cover $q|_{\widetilde{C}}$. Since $\widetilde{\rho}$ has a fixed point $\widetilde{s}$ and $C$ is connected, $\widetilde{\rho}$ must restrict to the identity on $\widetilde{C}$. Thus the lift $\widetilde{\rho}|_{\widetilde{X}}$ is independent of the choice of $\widetilde{s}$. The uniqueness of $\widetilde{\rho}$ follows from the uniqueness of the above construction by noting that $\widetilde{\rho}$ is independent of the particular choice of $\widetilde{s} \in \widetilde{C}$ since all of $\widetilde{C}$ is fixed pointwise.
\end{proof}
As a corollary we have the following version of Proposition \ref{prop:symmetry_lifts} for surfaces in $B^4$.
\begin{corollary}
\label{cor:B4_symmetry_lifts} Let $K\subset S^3$ be a knot and $F \subset B^4$ a properly embedded, connected (not necessarily orientable) surface with $\partial F = K$. Let $\rho \colon B^4 \rightarrow B^4$ be an orientation-preserving finite order diffeomorphism such that $\rho(F) = F$. Suppose that $\rho$ has a fixed point in $B^4 \backslash F$. Let $q \colon \Sigma(B^4,F) \rightarrow B^4$ be the double branched covering projection. Then $\rho$ lifts to a smooth map $\widetilde{\rho} \colon \Sigma(B^4,F) \rightarrow \Sigma(B^4,F)$, that is, $q \circ \widetilde{\rho} = \rho \circ q$ and order$(\widetilde{\rho}) =$ order$(\rho)$. Furthermore there are exactly two such lifts, and if $A$ is the fixed-point set of $\rho$ and $C$ is a component of $A \backslash F$, then exactly one of these two lifts fixes $q^{-1}(C)$ pointwise.

\end{corollary}
\begin{proof}
Double $(B^4,F)$ to obtain a closed connected surface $\overline{F}$ in $S^4$. Apply Proposition \ref{prop:symmetry_lifts} to get a lift $\widetilde{\rho}\colon \Sigma(S^4,\overline{F}) \to \Sigma(S^4,\overline{F})$. Restricting $\widetilde{\rho}$ to $\Sigma(B^4,F)$ gives the desired lift. 
\end{proof}

In order to have well-defined invariants, we now pin down one of the two lifts from Corollary \ref{cor:B4_symmetry_lifts}.
\begin{definition} \label{def:periodic_unique_lift}
  Let $K \subset S^3$ be a periodic knot bounding an equivariant surface $(F,\rho)$ in $B^4$. Let $A$ be the fixed point set of $\rho$ and let $C$ be the component of $A \backslash F$ containing the axis in $S^3$. The \emph{distinguished lift} of $\rho$ is the unique lift $\widetilde{\rho} \colon \Sigma(B^4, F) \rightarrow \Sigma(B^4,F)$ pointwise fixing $q^{-1}(C)$, where $q \colon \Sigma(B^4,F) \rightarrow B^4$ is the branched covering projection.
\end{definition}

\begin{definition} \label{def:invertible_unique_lift}
  Let $K \subset S^3$ be a directed strongly invertible knot bounding an equivariant surface $(F,\rho)$ in $B^4$. The axis of symmetry in $S^3$ is made up of two half-axes: the half-axis distinguished by the direction on $K$, and the other half-axis $\eta$. Let $C$ be the component of $A \backslash F$ containing $\eta$. The \emph{distinguished lift} of $\rho$ is the unique lift $\widetilde{\rho} \colon \Sigma(B^4, F) \rightarrow \Sigma(B^4,F)$ pointwise fixing $q^{-1}(C)$, where $q \colon \Sigma(B^4,F) \rightarrow B^4$ is the branched covering projection.
\end{definition}

\begin{remark}
In the periodic case the distinguished lift in fact pointwise fixes all of $q^{-1}(A)$, however we will not make use of this. In the strongly invertible case the non-distinguished lift fixes $q^{-1}(C')$ where $C'$ is the component of $A \backslash F$ containing the half-axis distinguished by the direction on $K$. 
\end{remark}

Before proving Theorem \ref{thm:eq_embedding}, we show that for an equivariant spanning surface the action of the distinguished lift on the intersection form can be simply described in terms of the Gordon-Litherland form \cite{Gordon-Litherland}. To begin, we briefly recall the Gordon-Litherland form. Let $K$ be a knot in $S^3$ and let $F$ be a compact, connected (not necessarily orientable) surface embedded in $S^3$ with $\partial F = K$. Gordon and Litherland define a bilinear form $\mathcal{G}_F \colon H_1(F) \times H_1(F) \rightarrow \Z$ as follows. Suppose $\alpha, \beta \in H_1(F)$ are represented by embedded oriented multicurves $a$, $b$ in $F$. We can push off $2b$ into $S^3\backslash F$, obtaining $\overline{b}$. (Locally, one copy of $b$ is pushed off each side of $F$.) Define $\mathcal{G}_F(\alpha, \beta) := \mbox{Lk}(a, \overline{b})$ to be the linking number of $a$ and $\overline{b}$. Gordon and Litherland show that the form $\mathcal{G}_F$ is a well-defined symmetric bilinear form. If $F$ is orientable, $\mathcal{G}_F$ is the symmetrized Seifert form.

Let $\widehat{F}$ be the surface obtained from $F$ by pushing $\mbox{int}(F)$ into $\mbox{int}(B^4)$. Let $\Sigma(B^4,\widehat{F})$ denote the double branched cover of $B^4$ over $\widehat{F}$. Gordon and Litherland showed that $(H_1(F), \mathcal{G}_F) \cong (H_2(\Sigma(B^4,\widehat{F})), Q)$, where $Q$ is the intersection pairing on $H_2(\Sigma(B^4,\widehat{F}))$.

\begin{proposition} \label{prop:GL_lattice_lift}
Let $K$ be a knot in $S^3$ and suppose it has a strongly invertible or periodic symmetry $\rho \colon (S^3, K) \rightarrow (S^3, K)$. Let $F$ be a compact, connected embedded surface in $S^3$ with $\partial F = K$ and $\rho(F) = F$. Let $\widetilde{\rho} \colon \Sigma(B^4,\widehat{F}) \rightarrow \Sigma(B^4,\widehat{F})$ be a lift of $\rho$ (as in Corollary \ref{cor:B4_symmetry_lifts}). Then under the identification of the intersection form  $(H_2(\Sigma(B^4,\widehat{F})), Q) \cong (H_1(F), \mathcal{G}_F)$, the induced map of lattices $\widetilde{\rho}_* \colon (H_2(\Sigma(B^4,\widehat{F})), Q) \rightarrow (H_2(\Sigma(B^4,\widehat{F})), Q)$ is equivalent to $\pm (\restr{\rho}{F})_* \colon (H_1(F), \mathcal{G}_F) \rightarrow (H_1(F), \mathcal{G}_F)$. Furthermore:
\begin{enumerate}[label=(\roman*)]
\item If $K$ is periodic and $\widetilde{\rho}$ is the distinguished lift of $\rho$ (as in Definition \ref{def:periodic_unique_lift}), then the induced map on lattices is equivalent to $+(\rho|_F)_*$.
\item Suppose $K$ is strongly invertible and directed, and $\widetilde{\rho}$ is the distinguished lift of $\rho$ (as in Definition \ref{def:invertible_unique_lift}). The direction distinguishes a half-axis $\eta^+$; call the other half-axis $\eta^-$. The induced map on lattices is equivalent to $+(\rho|_F)_*$ if $F$ contains $\eta^+$, and is equivalent to $-(\rho|_F)_*$ if $F$ contains $\eta^-$.
\end{enumerate}
\end{proposition}

\begin{proof}
The symmetry $\rho \colon S^3 \rightarrow S^3$ extends to a symmetry $\rho\colon B^4 \to B^4$ given by the cone of $\rho$. Following \cite[proof of Theorem 3]{Gordon-Litherland}, $\Sigma(B^4,\widehat{F})$ can be constructed as follows. Let $D_1$ denote the manifold obtained by cutting open $B^4$ along the trace of an equivariant isotopy which pushes $\mbox{int} (F)$ into $\mbox{int} (B^4)$. The manifold $D_1$ is homeomorphic to $B^4$ and the part exposed by the cut is given by an equivariant tubular neighborhood $N$ of $F$ in $S^3$. Let $\iota \colon N \rightarrow N$ be the involution given by reflecting each fiber. Let $D_2$ be another copy of $D_1$. Then
\[
\Sigma(B^4,\widehat{F}) = (D_1 \cup D_2) / (x \in N \subset D_1 \sim \iota(x) \in N \subset D_2).
\]
The isomorphism $\phi \colon (H_1(F), \mathcal{G}_F) \rightarrow (H_2(\Sigma(B^4,\widehat{F})), Q)$ is given as follows. Let $a$ be a $1$-cycle in $F$, then
$$\phi([a]) = [\mbox{(cone on }a\mbox{ in }D_1\mbox{)} - \mbox{(cone on }a\mbox{ in }D_2\mbox{)}].$$
Let $\widetilde{\rho} \colon \Sigma(B^4,\widehat{F}) \rightarrow \Sigma(B^4,\widehat{F})$ be a lift of $\rho\colon B^4 \to B^4$ as in Corollary \ref{cor:B4_symmetry_lifts}.
Note that there are two cases: $\widetilde{\rho}$ maps the interior of $D_1$ to the interior of $D_1$, or $\widetilde{\rho}$ maps the interior of $D_1$ to the interior of $D_2$. We first assume that it is mapped into the interior of $D_1$. The map $\widetilde{\rho}$ restricted to the branched set $F$ is equal to $\restr{\rho}{F}$. Thus, $\widetilde{\rho}(a) = \rho(a)$. For $i\in\{1,2\}$, the cone of $a$ in $D_i$ is mapped to a disk in $D_i$ with boundary $\rho(a)$. Hence, 
\begin{align*}
   \widetilde{\rho}_*(\phi([a])) &= [\mbox{(cone on }\rho(a)\mbox{ in }D_1\mbox{)} - \mbox{(cone on }\rho(a)\mbox{ in }D_2\mbox{)}] \\
					   &= \phi(\rho([a])).
\end{align*}
If instead $\widetilde{\rho}$ maps the interior of $D_1$ into the interior of $D_2$, we similarly get
\begin{align*}
  \widetilde{\rho}_*(\phi([a])) &= [\mbox{(cone on }\rho(a)\mbox{ in }D_2\mbox{)} - \mbox{(cone on }\rho(a)\mbox{ in }D_1\mbox{)}] \\
					   &= -\phi(\rho([a])).
\end{align*}

Now let $q\colon \Sigma(B^4,\widehat{F}) \to B^4$ be the branched covering map. If $K$ is periodic with axis $\gamma$, and $\widetilde{\rho}$ is the distinguished lift (see Definition \ref{def:periodic_unique_lift}), then $q^{-1}(\gamma) \subset \Sigma(B^4,\widehat{F})$ is pointwise fixed by $\widetilde{\rho}$. Since $\partial D_1 \backslash N$ contains points of $q^{-1}(\gamma)$, $\widetilde{\rho}(\mbox{int}(D_1)) = \mbox{int}(D_1)$. Hence the induced map of lattices is given by $+(\rho|_F)_*$.

Alternatively, suppose $K$ is strongly invertible and directed. The fixed-point axis in $S^3$ consists of a half-axis $\eta^+$ distinguished by the direction on $K$, and another half-axis $\eta^-$. Suppose furthermore that $\eta^+ \subset F$ and that $\widetilde{\rho}$ is the distinguished lift (see Definition \ref{def:invertible_unique_lift}). Then $\widetilde{\rho}$ fixes $q^{-1}(\eta^-)$ pointwise and $q^{-1}(\eta^-) \not\subset N$. Hence $\widetilde{\rho}$ has fixed points in $\partial D_1 \backslash N$ and so $\widetilde{\rho}(\mbox{int}(D_1)) = \mbox{int}(D_1)$. Thus the induced map of lattices is given by $+(\rho|_F)_*$. 

If $\widetilde{\rho}$ pointwise fixes $q^{-1}(\eta^-)$ then $\widetilde{\rho}$ pointwise fixes $q^{-1}(\eta^+) \cup q^{-1}(\eta^-)$, which is topologically a graph with two vertices and four parallel edges. This is impossible since the fixed-point set of $\widetilde{\rho}|_{\partial \Sigma(B^4,\widehat{F})}$ must be a 1-manifold. Hence $\widetilde{\rho}$ interchanges the two components of $q^{-1}(\mbox{Int}(\eta^+))$. If $F$ contains $\eta^-$ then $q^{-1}(\mbox{Int}(\eta^+))$ has a component in $D_1 \backslash N$ and a component in $D_2 \backslash N$. Hence $\widetilde{\rho}$ must be the lift interchanging int$(D_1)$ and int$(D_2)$. Thus the induced map of lattices is given by $-(\rho|_F)_*$. 
\end{proof}
We now have the tools needed to prove the main theorem of this section, which can sometimes be used to show that $\widetilde{g}_4(K) > |\sigma(K)|/2$.
\eqembedding*
\begin{center}\begin{tikzcd}
  {(H_1(F),\mathcal{G}_F)} & {(\mathbb{Z}^k,\mbox{Id})} \\
  {(H_1(F),\mathcal{G}_F)} & {(\mathbb{Z}^k,\mbox{Id})}
  \arrow["{\rho_*}"', from=1-1, to=2-1]
  \arrow["{\iota}", from=1-1, to=1-2]
  \arrow["{\delta}", from=1-2, to=2-2]
  \arrow["{\iota}"', from=2-1, to=2-2]
\end{tikzcd}
\end{center}

\begin{proof}
Let $\rho_F:B^4 \to B^4$ be the cone on $\rho:S^3 \to S^3$ and let $\hat{F} \subset B^4$ be a properly embedded surface equivariantly isotopic to $F$. Let $S \subset B^4$ be an orientable surface which is equivariant with respect to some extension $\rho_S\colon B^4 \to B^4$ of $\rho$ such that $\partial S = K$ and $g(S) = \widetilde{g}_4(K) = -\sigma(K)/2$. If $K$ is strongly invertible, we choose a direction on $K$ so that the distinguished half-axis is contained in $F$. Now let $\widetilde{\rho}_{\hat{F}}:\Sigma(B^4,\hat{F}) \to \Sigma(B^4,\hat{F})$ be the distinguished lift of $\rho_{\hat{F}}$ and $\widetilde{\rho}_S\colon \Sigma(B^4,S) \to \Sigma(B^4,S)$ be the distinguished lift of $\rho_S$ as in Definition \ref{def:periodic_unique_lift} or \ref{def:invertible_unique_lift}. Gluing these together, we obtain 
\[
(X,\widetilde{\rho}) = (\Sigma(B^4,\hat{F}), \widetilde{\rho}_{\hat{F}}) \cup (-\Sigma(B^4,S), \widetilde{\rho}_S).
\]
Since rank$(H_2(\Sigma(B^4,S))) =$ rank$(2g(S)) = -\sigma(K)$ (see for example \cite[Lemma 1]{GilmerLivingston}) and $\sigma(-\Sigma(B^4,S)) = -\sigma(K)$ (see \cite{KauffmanTaylor}), we have that $-\Sigma(B^4,S)$ is positive definite. The intersection form on $\Sigma(B^4,\hat{F})$ is isomorphic to $\mathcal{G}_F$ by \cite[Theorem 3]{Gordon-Litherland}. Hence $\Sigma(B^4,\hat{F})$ is positive definite by hypothesis and therefore $X$ is positive definite (see for example \cite[Proposition 7]{IssaMcCoy}). 

Since $X$ is smooth and positive definite, Donaldson's theorem \cite{Donaldson} implies that the intersection form $(H_2(X),Q_X)$ is isomorphic to $(\mathbb{Z}^k,\mbox{Id})$, where 
\[
k = \mbox{rank}\, H_2(X) = \mbox{rank}\, H_2(\Sigma(B^4,S)) + \mbox{rank}\, H_2(\Sigma(B^4,\hat{F})) = -\sigma(K) + b_1(F).
\]
Let $\iota:\Sigma(B^4,\hat{F}) \to X$ be the inclusion map. We then have the commutative diagram on the left of Figure \ref{fig:lattice_commutative} where all maps preserve intersection pairings. The commutative diagram on the right of Figure \ref{fig:lattice_commutative} is then obtained by identifying $(H_2(\Sigma(B^4,\hat{F})),Q_{\Sigma(B^4,\hat{F})}) \cong (H_1(F),\mathcal{G}_F)$ and $(H_2(X),Q_X) \cong (\mathbb{Z}^k,\mbox{Id})$, and noting that $\left(\restr{\widetilde{\rho}}{\Sigma(B^4,F)}\right)_* = \left(\restr{\rho}{F}\right)_*$ by Proposition \ref{prop:GL_lattice_lift}. Finally, $\widetilde{\rho}_*:\mathbb{Z}^k \to \mathbb{Z}^k$ is an automorphism since $\widetilde{\rho}$ is a diffeomorphism, and order$(\widetilde{\rho}_*) = $ order$(\widetilde{\rho}) = $ order$(\rho)$.

\begin{figure}
\begin{tikzcd}
  {H_2(\Sigma(B^4,\hat{F}))} && {H_2(X)} & {(H_1(F),\mathcal{G}_F)} && {(\mathbb{Z}^k,Id)} \\
  \\
  {H_2(\Sigma(B^4,\hat{F}))} && {H_2(X)} & {(H_1(F),\mathcal{G}_F)} && {(\mathbb{Z}^k,Id)}
  \arrow["{\widetilde{\rho}_*}"{name=0}, from=1-3, to=3-3]
  \arrow["{\iota_*}"', from=3-1, to=3-3]
  \arrow["{\iota_*}", from=1-1, to=1-3]
  \arrow["{\left(\restr{\widetilde{\rho}}{\Sigma(B^4,F)}\right)_*}"', from=1-1, to=3-1]
  \arrow["{\widetilde{\rho}_*}", from=1-6, to=3-6]
  \arrow["{\iota_*}", from=1-4, to=1-6]
  \arrow["{\iota_*}"', from=3-4, to=3-6]
  \arrow["{\left(\restr{\rho}{F}\right)_*}"{name=1, swap}, from=1-4, to=3-4]
  \arrow[Rightarrow, "{=}", from=0, to=1, shorten <=7pt, shorten >=7pt, phantom, no head]
\end{tikzcd}
\caption{A commutative diagram for the equivariant lattice embedding used in the proof of Theorem \ref{thm:eq_embedding}.}
\label{fig:lattice_commutative}
\end{figure}

\end{proof}

To apply this theorem, it will be convenient to recall that given a knot diagram $K$, the Gordon-Litherland form of a checkerboard surface is determined by the corresponding checkerboard graph which we briefly describe (see \cite{Gordon-Litherland} for details). Choosing a black and white checkerboard coloring of the knot diagram determines a black checkerboard surface $F$ spanning $K$. We denote by $G(F)$ the planar multigraph with vertices corresponding to white regions of the checkerboard coloring and edges corresponding to crossings between pairs of distinct white regions. Additionally, we weight the edges of $G(F)$ with $\pm 1$ depending on the sign of the corresponding crossing; $1$ for a right half-twist between white regions, and $-1$ for a left half-twist. We weight the vertices of $G(F)$ with negative the sum of the weights of incident edges and denote by $w(e)$ and $w(v)$ the weight of an edge $e$ and a vertex $v$ respectively. See Figure \ref{fig:9_40lattice} for an example. We call $G(F)$ the \emph{checkerboard graph} of $F$.

Now $H_1(F) \cong \mathbb{Z}\langle v_1, v_2, \dots, v_n\rangle/(v_1 + v_2 + \dots + v_n)$ where $v_1, v_2, \dots, v_n$ are the vertices of $G(F)$. This isomorphism is given by mapping $v_i$ to the class of the loop in $F$ going counterclockwise once around the white region corresponding to $v_i$. Under this isomorphism, the Gordon-Litherland form $\mathcal{G}_F$ is determined by 
\[
\langle v_i,v_j \rangle =
\begin{cases} 
  w(v_i), & i = j \\
  \sum_{e \in E(v_i,v_j)} w(e), & i \neq j, \\
\end{cases}
\]
where $E(v_i,v_j)$ is the set of edges between $v_i$ and $v_j$.

Suppose $(K,\rho)$ is a periodic or strongly invertible knot which has an alternating intravergent or transvergent diagram (see Definition \ref{def:symmetric_diagrams}). This diagram then has an equivariant checkerboard surface $F$ with positive definite Gordon-Litherland form and the symmetry of the diagram induces a planar symmetry of the checkerboard graph. The map $\rho_*\colon H_1(F) \to H_1(F)$ can then be determined directly from the planar symmetry of the checkerboard graph.

Theorem \ref{thm:eq_embedding} may be used to potentially obstruct $\widetilde{g}_4(K) = -\sigma(K)/2$ by showing that there is no embedding of lattices $\iota:(H_1(F),\mathcal{G}_F) \to (\mathbb{Z}^k, \mbox{Id})$ making the diagram in Theorem \ref{thm:eq_embedding} commute. More precisely, first enumerate all lattice embeddings $(H_1(F),\mathcal{G}_F) \to (\mathbb{Z}^k, \mbox{Id})$ (there are finitely many) which could potentially serve as such an $\iota$. For each embedding, then show that there does not exist a map $\delta\colon (\mathbb{Z}\langle e_1, \dots,e_k\rangle, \mbox{Id}) \to (\mathbb{Z}\langle e_1, \dots,e_k\rangle, \mbox{Id})$ making the diagram commute. Note that $\delta$ maps each $e_i$ to $\pm e_j$ for some $j$ and hence there are only finitely many possibilities for $\delta$. In some cases this $\delta$ may exist so that Theorem \ref{thm:eq_embedding} does not provide an obstruction.

In the case that $g_4(K) = -\sigma(K)/2$, this strategy can be used to produce examples with $\widetilde{g}_4(K) > g_4(K)$; see Example \ref{ex:9_40_lattice} and more generally Appendix \ref{sec:Donaldson_table}. Finally, we note that the signature is easy to compute in this setting. Choosing an arbitrary orientation on $K$, the Gordon-Litherland formula for the signature gives $\sigma(K) = \sigma(\mathcal{G}_F) - n$, where $n$ is the number of positive crossings in the alternating diagram.

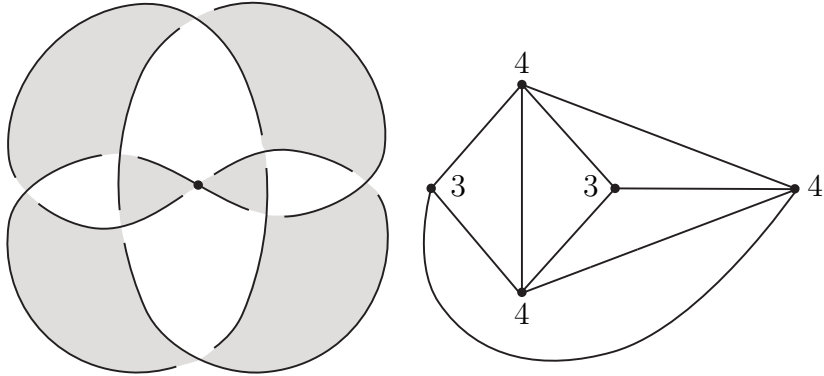
\begin{figure}[!htbp]
  \begin{overpic}[width=300pt, grid=false]{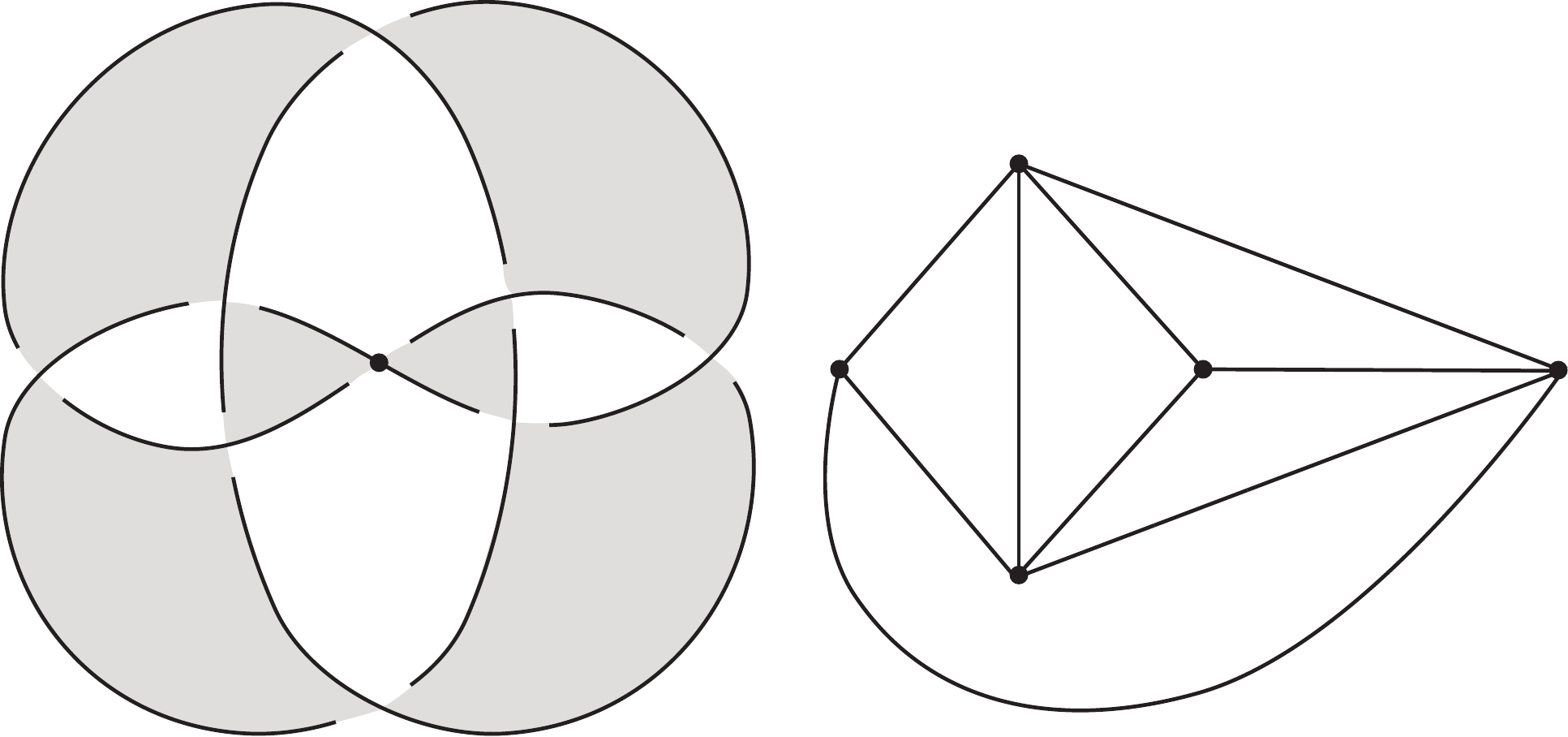}
    \put (101, 22.5) {$4$}
    \put (64, 38) {$4$}
    \put (72.7, 22.5) {$3$}
    \put (64, 6.3) {$4$}
    \put (56, 22.5) {$3$}
  \end{overpic}
\caption{The knot $9_{40}$ with a strong inversion given by rotation around an axis perpendicular to the page (left) and the checkerboard graph (right) for the shaded checkerboard surface. Every edge has weight $-1$, and the vertex weights are labeled.}
\label{fig:9_40lattice}
\end{figure}

\begin{example} \label{ex:9_40_lattice}
Consider the knot $K = 9_{40}$ with shaded checkerboard surface $F$ as shown in Figure \ref{fig:9_40lattice}. We see that $K$ has an alternating diagram in which a strong inversion is given by rotation around an axis perpendicular to the plane of the diagram and $F$ is equivariant with respect to this strong inversion. The Gordon-Litherland pairing is described by the corresponding checkerboard graph as shown in the figure. One can check that $\sigma(K) = -2$ and $g_4(K) = 1$ (\cite{MurakamiSugishita}; see also \cite{knotinfo}). By Theorem \ref{thm:eq_embedding}, if $\widetilde{g}_4(K) = -\sigma(K)/2 = 1$ then there exists an equivariant lattice embedding $\iota \colon (H_1(F), \mathcal{G}_F) \rightarrow (\Z^{6}, \mbox{Id})$. We used a computer program to check that there are exactly two distinct lattice embeddings $\iota_1, \iota_2\colon (H_1(F), \mathcal{G}_F) \rightarrow (\Z^{6}, \mbox{Id})$ up to automorphisms of $(\Z^{6}, \mbox{Id})$, and it so happens that $\rho_* \circ \iota_i = \iota_j$ for $\{i,j\} = \{1,2\}$ (see Figure \ref{fig:9_40_embedding}). If there were an automorphism $\delta:(\Z^{6}, \mbox{Id}) \to (\Z^{6}, \mbox{Id})$ such that $\iota_i \circ \rho_* = \delta \circ \iota_i$, then $\iota_j = \delta \circ \iota_i$, contradicting that $\iota_i$ and $\iota_j$ are distinct up to automorphisms of $(\Z^{6}, \mbox{Id})$ for $\{i,j\} = \{1,2\}$. Thus $\widetilde{g}_4(K) > 1 = g_4(K)$. In fact, one can show that $\widetilde{g}_4(K) = 2$. Indeed, performing an equivariant pair of crossing changes on a pair of opposite crossings closest to the axis in Figure \ref{fig:9_40lattice} gives the unknot, and thus by Proposition \ref{prop:crossing_change_genus}, $\widetilde{g}_4(K) \leq 2$.
\end{example}

\begin{figure}[!htbp]
\begin{overpic}[width=350pt, grid=false]{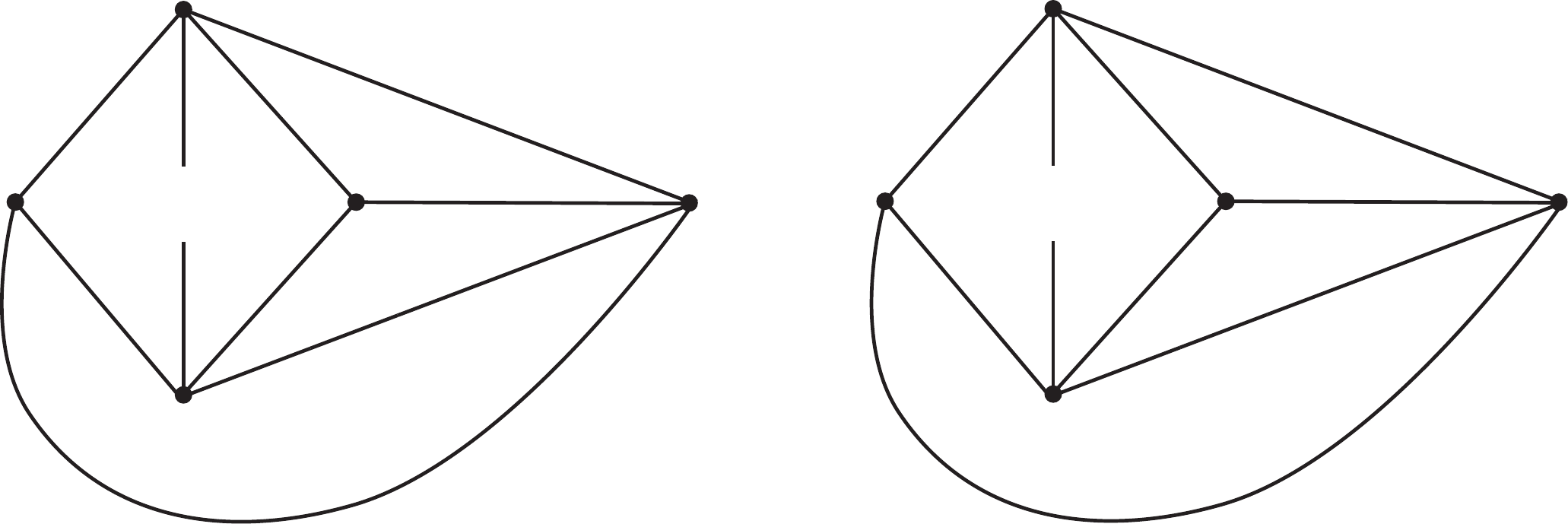}
    \put (13, 33) {\footnotesize $e_1$--$e_2$--$e_3$+$e_4$}
    \put (3, 20) {\footnotesize $e_1$+$e_2$+$e_3$}
    \put (5, 6) {\footnotesize --$e_1$+$e_2$--$e_3$--$e_5$}
    \put (22, 21.6) {\footnotesize --$e_4$+$e_5$+$e_6$}
    \put (36, 24) {\footnotesize --$e_1$--$e_2$+$e_3$--$e_6$}

    \put (58, 20) {\footnotesize --$e_4$+$e_5$+$e_6$}
    \put (68.5, 33) {\footnotesize --$e_1$+$e_2$--$e_3$--$e_5$}
    \put (62.5, 6) {\footnotesize $e_1$--$e_2$--$e_3$+$e_4$}
    \put (78, 21.6) {\footnotesize $e_1$+$e_2$+$e_3$}
    \put (92, 24) {\footnotesize --$e_1$--$e_2$+$e_3$--$e_6$}
  \end{overpic}
\caption{The unique pair of lattice embeddings (up to automorphism) of $H_1$ of a checkerboard surface for the knot $9_{40}$ into $\Z \langle e_1,e_2,e_3,e_4,e_5,e_6 \rangle$.}
\label{fig:9_40_embedding}
\end{figure}
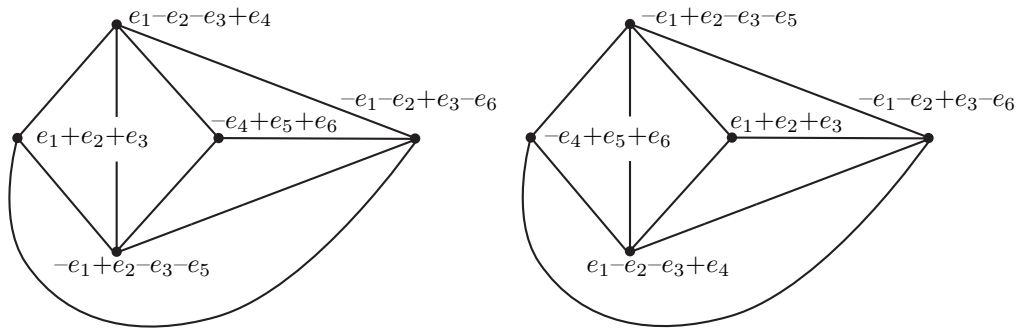

\begin{remark}
In Example \ref{ex:9_40_lattice}, Theorem \ref{thm:eq_embedding} was used to show that the smooth equivariant 4-genus of $K$ is $2$. In fact, we can show that the topological equivariant 4-genus of $K$ is also $2$ by essentially the same argument. In Theorem \ref{thm:eq_embedding} the smoothness is only needed in order to apply Donaldson's theorem to conclude that a certain closed positive definite 4-manifold $X$ has diagonalizable intersection form. However, in this example $b_2(X) = 6$ and in fact for algebraic reasons there is only one positive definite unimodular lattice of each rank $\leq 7$, namely the diagonal lattice (\cite{MR90606}, see also \cite[Section 6]{MR0506372}). 
\end{remark}

\section{The \texorpdfstring{$g$}\ -signature of knots} \label{sec:gsig}
In this section, we use the $g$-signature for 4-manifolds (see for example \cite{Gordongsig}) to generalize the knot signature to the equivariant setting. We use this generalization to prove a lower bound on the equivariant 4-genus of a periodic knot and on the butterfly 4-genus of a strongly invertible knot. We begin by recalling the definition of the $g$-signature for 4-manifolds.

\begin{definition} \cite[Section 1]{Gordongsig}
Let $X$ be a compact oriented 4-manifold with a finite order orientation-preserving diffeomorphism $\rho:X \to X$. The intersection form on $H_2(X;\mathbb{Z})$ induces a hermitian form $\varphi\colon H \times H \to \mathbb{C}$ where $H = H_2(X;\mathbb{C}) = H_2(X;\mathbb{Z}) \otimes \mathbb{C}$, by
\[
\varphi(x \otimes a, y\otimes b) = a \overline{b}(x \cdot y).
\]
Then $\varphi(x \otimes a, y \otimes b) = \varphi(\rho_* x \otimes a,\ \rho_* y \otimes b)$ for all $(x \otimes a),(y \otimes b)\in H$. We may choose a $\rho$-invariant direct sum decomposition $H = H^+ \oplus H^- \oplus H^0$ which is orthogonal with respect to $\varphi$ and where $\varphi$ is positive definite, negative definite, and zero on $H^+,H^-,$ and $H^0$ respectively. Then the $g$-signature is 
\[
\widetilde{\sigma}(X,\rho) = \mbox{trace}(\rho_*|_{H^+}) - \mbox{trace}(\rho_*|_{H^-}).
\]
\end{definition}
It is usually easier to compute the $g$-signature using the following well-known proposition.
\begin{proposition}
\label{prop:compute_gsig}
Let $X$ be a compact oriented 4-manifold with a finite order $n$ orientation-preserving diffeomorphism $\rho:X \to X$, and let $\varphi:H \times H \to \mathbb{C}$ be the intersection pairing where $H = H_2(X;\mathbb{C})$. Then 
\[
\widetilde{\sigma}(X,\rho) = \sum_{j = 0}^{n-1} \omega^j \sigma(\restr{\varphi}{H(\omega^j)}),
\]
where $\omega = e^{2\pi i/n}$ is an $n^{th}$ root of unity and $H(\omega^j)$ is the $\omega^j$-eigenspace of $\rho_*:H_2(X;\mathbb{C}) \to H_2(X; \mathbb{C})$. 
\end{proposition}
\begin{proof}
Let $\lambda_1,\dots, \lambda_k$ be the eigenvalues of $\restr{\rho_*}{H^+}$ and $\lambda_{k+1}, \dots, \lambda_l$ be the eigenvalues of $\restr{\rho_*}{H^-}$. Since $\rho^n = $ Id, the $\lambda_j$ are $n^{th}$ roots of unity. Then 
\begin{align*}
\widetilde{\sigma}(X,\rho) &= \mbox{trace}(\rho_*|_{H^+}) - \mbox{trace}(\rho_*|_{H^-}) \\
&= \sum_{j=1}^k \lambda_j - \sum_{j=k+1}^l \lambda_j = \sum_{j = 0}^{n-1} \omega^j (p_j - n_j)\\
&= \sum_{j = 0}^{n-1} \omega^j \sigma(\restr{\varphi}{H(\omega^j)}),
\end{align*}
where $p_j = |\{i \leq k : \lambda_i = \omega^j\}|$ and $n_j = |\{i > k : \lambda_i = \omega^j\}|$.
\end{proof}

With this in hand, we define the $g$-signature for certain group actions on knots.
\subsection{Periodic knots and the g-signature} \label{subsec:periodic_gsig}
\begin{definition}
\label{def:periodic_gsig}
Let $K$ be an $n$-periodic knot in $S^3$ with $\rho\colon (S^3,K) \to (S^3,K)$ a generator of the periodic symmetry. Let $F \subset B^4$ be an orientable equivariant surface for $(K,\rho)$, and $\overline{\rho}\colon (B^4, F) \to (B^4, F)$ be an extension of $\rho$. Let $\widetilde{\rho}:\Sigma(B^4,F) \to \Sigma(B^4,F)$ be the distinguished lift of $\overline{\rho}$ to the branched double cover (see Definition \ref{def:periodic_unique_lift}). The \emph{g-signature} $\widetilde{\sigma}(K,\rho)$ of $K$ is the average of the $g$-signatures $\widetilde{\sigma}(\Sigma(B^4,F),\widetilde{\rho}^{\,i})$, that is,
\[
\widetilde{\sigma}(K,\rho) = \frac{1}{n-1}\sum_{i=1}^{n-1}\widetilde{\sigma}(\Sigma(B^4,F),\widetilde{\rho}^{\,i}).
\]
When the periodic symmetry is clear we write $\widetilde{\sigma}(K)$ for $\widetilde{\sigma}(K,\rho)$.
\end{definition}
Note that in Definition \ref{def:periodic_gsig}, since we take the average of the $g$-signatures, $\widetilde{\sigma}(K)$ is independent of the choice of generator $\rho$ for the $n$-periodic symmetry. Additionally, the lift $\widetilde{\rho}$ is independent of the orientation on $F$ and hence $\widetilde{\sigma}(K)$ does not depend on the orientation of $K$. We now show that the $g$-signature is well-defined.
\begin{theorem}
The $g$-signature of an $n$-periodic knot $(K,\rho)$ is independent of the choice of orientable equivariant surface and extension of the periodic symmetry to $B^4$. 
\end{theorem}
\begin{proof}
For $i \in \{1,2\}$, let $F_i$ be an orientable equivariant surface for $K$ in $B^4$ and let $\overline{\rho}_i\colon (B^4,F_i) \to (B^4,F_i)$ be an extension of $\rho$ to $B^4$. Then let $(S^4,F) = (B^4,F_1) \cup (-B^4, -F_2)$ with $\overline{\rho}:(S^4,F) \to (S^4,F)$ the symmetry which restricts to $\overline{\rho}_i$ on $(B^4,F_i)$ for $i \in \{1,2\}$. By the same argument as in Proposition 4.1 of \cite{Naik}, the surface $F$ bounds an equivariant 3-manifold $M \subset S^4$. 

Thinking of $S^4 = \partial B^5$, we can extend $\overline{\rho}$ to an $n$-periodic symmetry (which we again refer to as $\overline{\rho}$) of $B^5$ by taking the cone of $\overline{\rho}$. Under this extension, we can perform an equivariant isotopy pushing $M$ into the interior of $B^5$ fixing $\partial M$ to get a 3-manifold $\widehat{M}$ properly embedded in $B^5$. The symmetry $\overline{\rho}$ on $B^5$ lifts to a symmetry $\widetilde{\rho}$ on the double branched cover $\Sigma(B^5,\widehat{M})$ of $B^5$ over $\widehat{M}$ by a similar argument to that of Proposition \ref{prop:symmetry_lifts}. This lift can be chosen to restrict to the lift of $\overline{\rho}_i$ on each $\Sigma(B^4,F_i) \subset \Sigma(S^4,F)$ used in the definition of $\widetilde{\sigma}(K)$. For $0 < j < n$, since $\partial (\Sigma(B^5,\widehat{M}),\widetilde{\rho}^{\, j}) = (\Sigma(S^4,F), \widetilde{\rho}^{\, j})$, we have that $\widetilde{\sigma}(\Sigma(S^4,F), \widetilde{\rho}^{\, j}) = 0$ by \cite[Section 1]{Gordongsig}. Hence by Novikov additivity (again, see \cite[Section 1]{Gordongsig}),
\[
0 = \widetilde{\sigma}(\Sigma(S^4,F), \widetilde{\rho}^{\, j}) = \widetilde{\sigma}(\Sigma(B^4,F_1), \widetilde{\rho}^{\, j}) - \widetilde{\sigma}(\Sigma(B^4,F_2), \widetilde{\rho}^{\, j}),
\]
for $0< j < n$. Averaging these gives the desired result.
\end{proof}
\begin{proposition} \label{prop:gsig_concordance}
The $g$-signature of a periodic knot is an equivariant concordance invariant.
\end{proposition}
\begin{proof}
Let $K$ and $K'$ be equivariantly concordant. Then the union of an equivariant surface $S$ in $B^4$ with boundary $K$ and the equivariant concordance cylinder is an equivariant surface $S'$ in $B^4$ with boundary $K'$. The map $H_2(\Sigma(B^4,S);\mathbb{C}) \to H_2(\Sigma(B^4,S');\mathbb{C})$ induced by inclusion is an equivariant isomorphism. Hence $\widetilde{\sigma}(K) = \widetilde{\sigma}(K')$.
\end{proof}
\begin{theorem}
\label{thm:gsigbound}
If $K$ is a periodic knot, then $|\widetilde{\sigma}(K)| \leq 2\widetilde{g}_4(K)$.
\end{theorem}
\begin{proof}
Let $S \subset B^4$ be a minimal genus orientable equivariant surface with $\partial S = K$. Then it is well known (see for example \cite[Lemma 1]{GilmerLivingston}) that
\[
\mbox{rank}(H_2(\Sigma(B^4,S);\mathbb{C})) = \mbox{rank}(H_1(S;\mathbb{C})) = 2g(S) = 2\widetilde{g}_4(K).
\]
Fix $i \in \{1,2,\dots,n-1\}$. Applying Proposition \ref{prop:compute_gsig} to $(\Sigma(B^4,S),\widetilde{\rho}^{\,i})$, where $\widetilde{\rho}$ is the lift of $\rho$ from the definition of $\widetilde{\sigma}(K)$, we have
\begin{align*}
|\widetilde{\sigma}(\Sigma(B^4,S),\widetilde{\rho}^{\,i})| = \bigg{|}\sum_{j = 0}^{n-1} \omega^j \sigma(\restr{\varphi}{H(\omega^j)})\bigg{|} \leq \sum_{j = 0}^{n-1} |\sigma(\restr{\varphi}{H(\omega^j)})| \leq \mbox{rank}(H_2(\Sigma(B^4,S);\mathbb{C})) = 2\widetilde{g}_4(K).
\end{align*}
Hence by the triangle inequality, $|\widetilde{\sigma}(K)| \leq 2\widetilde{g}_4(K)$.
\end{proof}
The following theorem allows us to express the $g$-signature purely in terms of the signature of $K$ and the signature of the quotient $\overline{K}$. 

\begin{theorem}
\label{thm:periodicgsig}
Let $K$ be n-periodic with quotient knot $\overline{K}$. Then 
\[
\widetilde{\sigma}(K) = \frac{n \cdot \sigma(\overline{K}) - \sigma(K)}{n-1}.
\]
\end{theorem}

\begin{proof}
Let $S \subset B^4$ be an orientable equivariant surface with $\partial S = K$, equivariant with respect to $\rho:B^4 \to B^4$ extending the periodic action on $K$. Let $\widetilde{\rho}\colon \Sigma(B^4,S) \to \Sigma(B^4,S)$ be the lift of $\rho$ as in Definition \ref{def:periodic_gsig}. The quotient of $(B^4,S)$ by $\rho$ is $(B,\overline{S})$ where $B$ is a topological 4-ball by Proposition \ref{prop:quotient_top_b4} and $\overline{S}$ is an orientable surface with boundary $\overline{K}$.

Let $G$ be a finite group of order $n$ acting on a compact connected oriented 4-manifold $X$. Then by \cite[Section 6]{Gordongsig}, 
\[
n \cdot \sigma(X/G) - \sigma(X) = \sum_{g \in G \backslash \{Id\}} \widetilde{\sigma}(X,g).
\]
Taking $X = \Sigma(B^4,S)$ and $G = \langle\rho \rangle$, we get that $X/G = \Sigma(B, \overline{S})$ and hence
\[
n \cdot \sigma(\overline{K}) - \sigma(K) = \sum_{i=1}^{n-1} \widetilde{\sigma}(\Sigma(B^4,S),\widetilde{\rho}^{\, i}),
\]
where we use that $\sigma(X) = \sigma(K)$ and $\sigma(X/G) = \sigma(\overline{K})$ by \cite{KauffmanTaylor}. Dividing both sides by $n-1$ gives the desired equality. 
\end{proof}
As a corollary of Theorem \ref{thm:gsigbound} and Theorem \ref{thm:periodicgsig} we obtain the following key theorem.

\eqgsigineq*

The following two examples use Theorem \ref{thm:eq4g_sig_ineq} to obtain a difference of more than 1 between the 4-genus and the equivariant 4-genus. The first example, which we state as a theorem, gives a family of 2-periodic knots $K_n$ for which $|\widetilde{g}_4(K_n) - g_4(K_n)|$ is unbounded.

\begin{thm:Montesinos}
Let $\{K_n\}$ be the family of 2-periodic Montesinos knots\footnote{We follow the convention for Montesinos knots notation from \cite{MR3825858}.} with
\[
K_n = M\left(1;-n, 2n+2,-n\right)
\] 
shown in Figure \ref{fig:UnboundedExample}, where $n$ is odd and positive. The difference $\widetilde{g}_4(K_n) - g_4(K_n)$ is unbounded. In fact, $\widetilde{g}_4(K_n) = 2n$ and $g_4(K_n) = 1$.
\end{thm:Montesinos}

\begin{figure}[!htbp]
\begin{overpic}[width=350pt, grid=false]{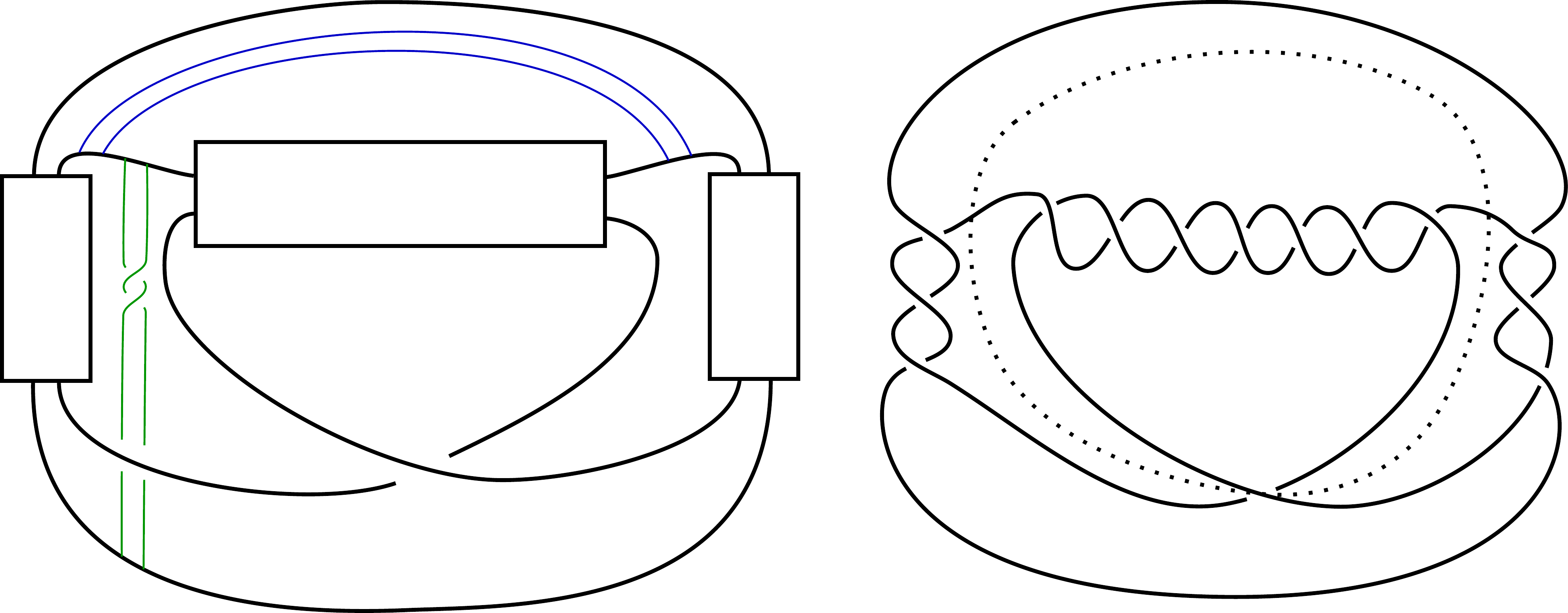}
\put (20.5, 26) {$2n+1$}
\put (1.1, 21) {-$n$}
\put (46.3, 21) {-$n$}
\end{overpic}
\caption{A family of 2-periodic knots $K_n$ (left) for $n$ odd. The boxes are twist regions with the labeled number of half-twists. When $n = 3$, we have $K_3 =$ K14a19410 (right). The period can be seen by performing a flype on the central crossing region (enclosed by a dotted loop in the right diagram), then rotating the entire diagram by $\pi$ within the plane of the diagram. Performing the green and blue band moves shown gives the unknot so that $g_4(K_n) = 1$.}
\label{fig:UnboundedExample}
\end{figure}

\begin{proof}
For all positive odd $n$, $K_n$ has signature $-2$, and the quotient knot is the left-handed torus knot $T(2,n)$ which has signature $n-1$. Thus by Theorem \ref{thm:eq4g_sig_ineq}, we have that $\widetilde{g}_4(K_n) \geq n$. On the other hand, performing the pair of band moves shown in Figure \ref{fig:UnboundedExample} gives the unknot and hence $g_4(K_n) \leq 1$. We also have $g_4(K_n) \geq 1$, since $\sigma(K) = -2$ so that $g_4(K_n) = 1$. For an upper bound on the equivariant 4-genus, we note that performing Seifert's algorithm on the diagram in Figure \ref{fig:UnboundedExample} gives a genus $2n$ surface for $K_n$. Hence by Edmonds' theorem, $\widetilde{g}_4(K_n) \leq 2n$.  We note that since the linking number with the axis is $\pm(2n+3)$, Proposition \ref{prop:RH} gives the stronger lower bound $\widetilde{g}_4(K_n) \geq 2n$ so that in fact $\widetilde{g}_4(K) = 2n$.
\end{proof}

\begin{figure}[!htbp]
\scalebox{.5}{\includegraphics{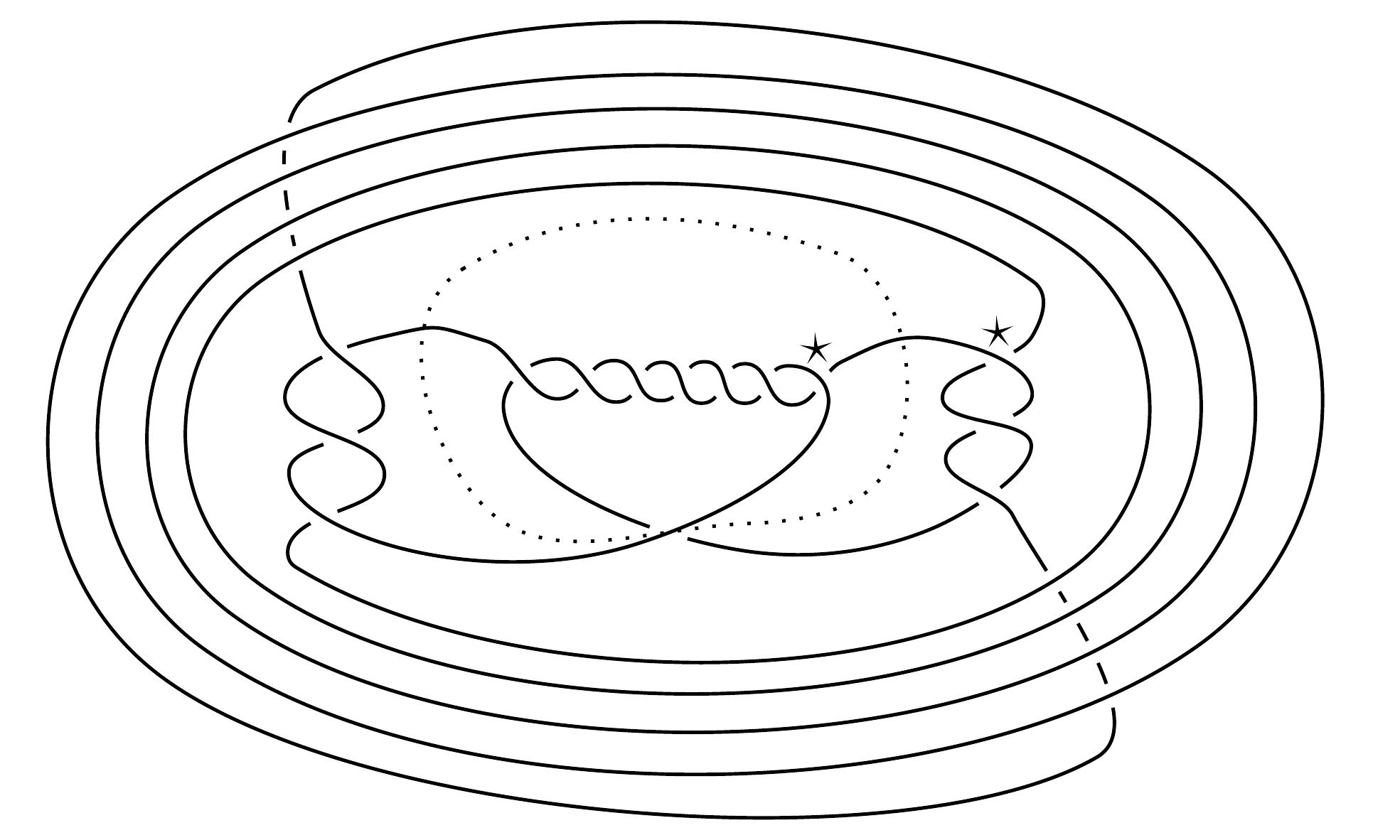}}
\caption{A 22-crossing periodic knot. The period can be seen by performing a flype on the tangle enclosed by the dotted loop, then rotating the entire diagram by $\pi$ within the plane of the diagram.}
\label{fig:gsignature22}
\end{figure}

As in the previous example, Proposition \ref{prop:RH} sometimes gives a better lower bound than Theorem \ref{thm:eq4g_sig_ineq}. However, we provide the following example where Theorem \ref{thm:eq4g_sig_ineq} gives a stronger bound than Proposition \ref{prop:RH}.

\begin{example}
\label{ex:periodic2}
Consider the periodic knot $K$ shown in Figure \ref{fig:gsignature22} which has $\sigma(K) = -4$, and observe that the quotient knot is the left-handed trefoil which has signature $2$. By Theorem \ref{thm:eq4g_sig_ineq}, we have that $\widetilde{g}_4(K) \geq 4$. On the other hand, the linking number between $K$ and the axis of symmetry is 1 and the genus of the trefoil is 1, so that Proposition \ref{prop:RH} only gives that $\widetilde{g}_4(K) \geq 2$. Furthermore, changing the starred crossings gives a genus 1 cobordism (see for example \cite[Lemma 5(ii)]{MR3938580}) to 12n466 which has 4-genus 2 (see for example \cite{knotinfo}) so that $g_4(K) \leq 3$. 
\end{example}

\subsection{Strongly invertible knots and the g-signature}
\label{subsec:strongly_invertible_gsig}
\begin{definition}
\label{def:strongly_invertible_gsig}
Let $(K,\tau)$ be a directed strongly invertible knot which bounds a butterfly surface $F \subset B^4$, and let $\widetilde{\tau}\colon \Sigma(B^4,F) \to \Sigma(B^4,F)$ be the distinguished lift from Definition \ref{def:invertible_unique_lift}. The $g$-signature $\widetilde{\sigma}(K,\tau)$ of $K$ is the $g$-signature $\widetilde{\sigma}(\Sigma(B^4,F),\widetilde{\tau})$. If the directed strong inversion is clear we write $\widetilde{\sigma}(K)$ for $\widetilde{\sigma}(K,\tau)$, and if $(K,\tau)$ does not bound a butterfly surface, we define $\widetilde{\sigma}(K,\tau) = \infty$.
\end{definition}
To prove that the $g$-signature of a directed strongly invertible knot is well-defined as stated, we need the following lemma.
\begin{lemma}
\label{lemma:butterfly_bounds_eq3}
Let $(S^4,F) = (B^4,F_1) \cup (-B^4, -F_2)$ where $F_1$ and $F_2$ are butterfly surfaces for a directed strongly invertible knot $(K,\tau)$. Then $F \subset S^4 = \partial B^5$ bounds a properly embedded $3$-manifold $M \subset B^5$ which is equivariant with respect to an involution on $B^5$ extending the involution of $(S^4,F)$. 
\end{lemma}
\begin{proof}
For $i \in \{1,2\}$, applying Proposition \ref{prop:BL_bounds_surface} to $F_i$ gives that the butterfly link $\BL(K)$ is a 2-periodic link bounding an orientable equivariant surface $F_i'$. Then $F' = F_1' \cup F_2'$ is a closed surface in $S^4$. Since $F_1$ and $F_2$ are butterfly surfaces, the quotient of $F'$ by $\tau |_{F'}$ is an orientable surface. Hence the argument proving Proposition 4.1 of \cite{Naik} applies, giving that the surface $F'$ bounds an equivariant 3-manifold $M' \subset S^4$. 

Thinking of $S^4 = \partial B^5$, we can extend $\tau$ to an involution (which we again refer to as $\tau$) of $B^5$ by taking the cone of $\tau\colon S^4 \to S^4$. Under this extension, we can perform an equivariant isotopy pushing $M'$ into the interior of $B^5$ fixing $\partial M'$ to get a 3-manifold $\widehat{M'}$ properly embedded in $B^5$. Now for $i \in\{1,2\}$, the last part of Proposition \ref{prop:BL_bounds_surface} applied to $F_i$ gives an $(I \times D^2) \subset B^4$. Gluing these together along their shared boundary in $S^3$, we obtain an equivariant 3-manifold $N = (I \times D^2) \subset S^4$ such that $F' = [F \backslash (I \times \partial D^2)] \cup [(\partial I) \times D^2]$. Then performing an equivariant isotopy to $(\widehat{M'} \cup N) \subset B^5$ so that it is properly embedded in $B^5$ gives the desired 3-manifold $M$.
\end{proof}

The following theorem shows that the $g$-signature is well-defined.
\begin{theorem}
The $g$-signature of a directed strongly invertible knot $(K,\tau)$ is independent of the choice of butterfly surface and extension of the strong inversion to $B^4$. 
\end{theorem}
\begin{proof}
For $i \in \{1,2\}$, let $F_i$ be a butterfly surface for $K$ in $B^4$ and let $\overline{\tau}_i\colon (B^4,F_i) \to (B^4,F_i)$ be an extension of $\tau$ to $B^4$. Then let $(S^4,F) = (B^4,F_1) \cup (-B^4, -F_2)$ with $\overline{\tau}:(S^4,F) \to (S^4,F)$ the symmetry which restricts to $\overline{\tau}_i$ on $(B^4,F_i)$ for $i \in \{1,2\}$. By Lemma \ref{lemma:butterfly_bounds_eq3}, the surface $F$ bounds a properly embedded 3-manifold $M \subset B^5$ which is equivariant under some extension of $\overline{\tau}$ to $B^5$. 

The symmetry $\overline{\tau}$ on $B^5$ lifts to a symmetry $\widetilde{\tau}$ on the double branched cover $\Sigma(B^5,\widehat{M})$ of $B^5$ over $\widehat{M}$ by a similar argument to that of Proposition \ref{prop:symmetry_lifts}. This lift can be chosen to restrict to the lift of $\overline{\tau}_i$ on each $\Sigma(B^4,F_i) \subset \Sigma(S^4,F)$ used in the definition of $\widetilde{\sigma}(K)$. Since $\partial (\Sigma(B^5,\widehat{M}),\widetilde{\tau}) = (\Sigma(S^4,F), \widetilde{\tau})$, we have that $\widetilde{\sigma}(\Sigma(S^4,F), \widetilde{\tau}) = 0$ by \cite[Section 1]{Gordongsig}. Hence by Novikov additivity (again, see \cite[Section 1]{Gordongsig}),
\[
0 = \widetilde{\sigma}(\Sigma(S^4,F), \widetilde{\tau}) = \widetilde{\sigma}(\Sigma(B^4,F_1), \widetilde{\tau}) - \widetilde{\sigma}(\Sigma(B^4,F_2), \widetilde{\tau}).
\]
\end{proof}
\begin{remark} \label{rmk:antipode_gsig_negate}
In general, the $g$-signature depends on the direction of $K$. In fact by Proposition \ref{prop:GL_lattice_lift}, if $K$ has a butterfly Seifert surface then $\widetilde{\sigma}(K) = -\widetilde{\sigma}(K^-)$.
\end{remark}
The same proof as for Proposition \ref{prop:gsig_concordance}, but applied to butterfly surfaces, gives the following proposition.
\begin{proposition}
The $g$-signature of a directed strongly invertible knot is an equivariant concordance invariant.
\end{proposition}
As for the usual knot signature, we have the following additivity property.
\begin{proposition}
\label{prop:gsigsum}
The $g$-signature (for strongly invertible knots) is additive under equivariant connect sum. 
\end{proposition}
\begin{proof}
Let $K_1$ and $K_2$ be directed strongly invertible knots which bound butterfly surfaces $S_1$ and $S_2$ in $B^4$ respectively. Then $K = K_1 \widetilde{\#} K_2$ bounds a butterfly surface $S$ in $B^4$ where $(B^4,S)$ is the boundary connected sum of $(B^4,S_1)$ with $(B^4,S_2)$. Passing to the double branched covers, $H_2(\Sigma(B^4,S))$ is isomorphic to $H_2(\Sigma(B^4,S_1)) \oplus H_2(\Sigma(B^4,S_2))$ and this isomorphism respects the corresponding involutions from the definition of the $g$-signature. Hence $\widetilde{\sigma}(K) = \widetilde{\sigma}(K_1) + \widetilde{\sigma}(K_2)$.
\end{proof}

We now have the main theorem of Section \ref{subsec:strongly_invertible_gsig}, giving a lower bound on the butterfly genus.

\begin{thm:siconcordancebound}
If $K$ is a directed strongly invertible knot which bounds a butterfly surface in $B^4$, then $\bg_4(K) \geq |\widetilde{\sigma}(K)|/2$.
\end{thm:siconcordancebound}

\begin{proof}
The proof is identical to Theorem \ref{thm:gsigbound}.
\end{proof}

Note that $\bg_4(K)$ does not depend on the direction of $K$ so that this theorem gives a lower bound on the butterfly 4-genus for each choice of direction. 

As the following example shows, the $g$-signature can be used to prove that the butterfly 4-genus may be arbitrarily larger than the 4-genus.

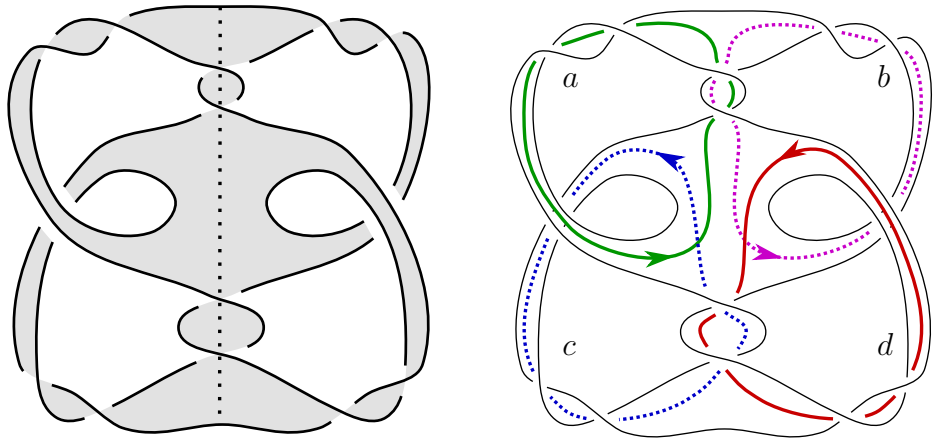
\begin{figure}
\begin{overpic}[width=350pt, grid=false]{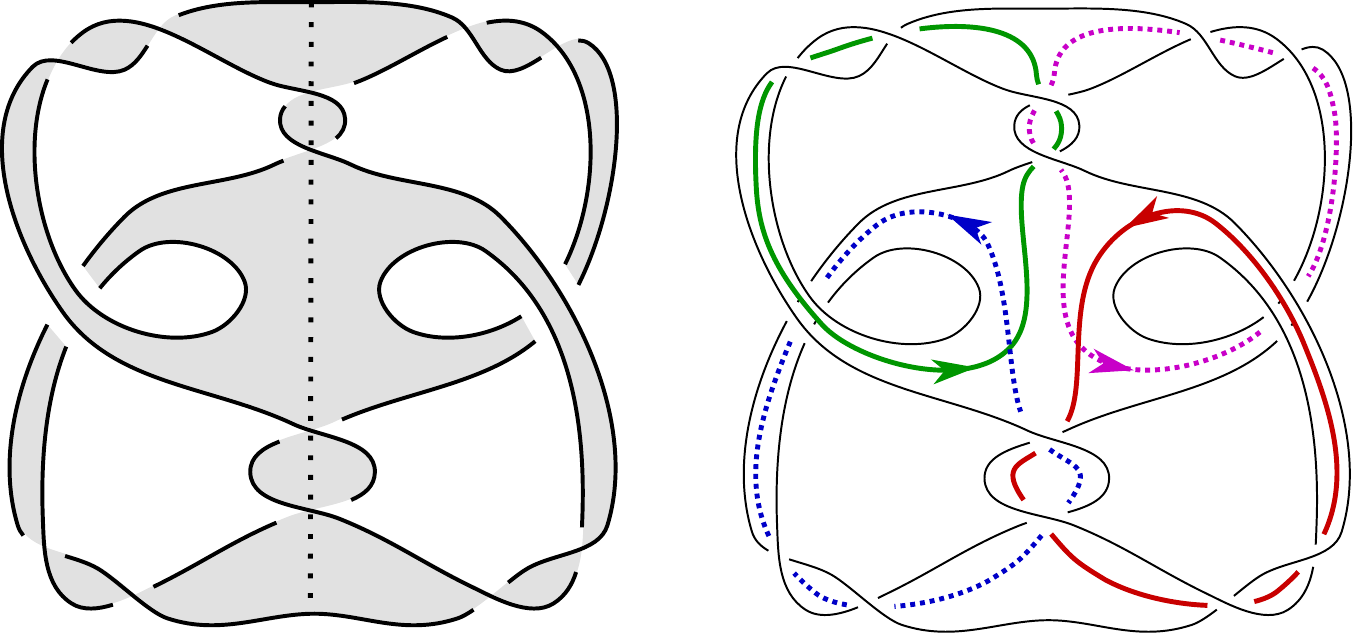}
    \put (60, 38) {$a$}
    \put (94, 38) {$b$}
    \put (60, 9) {$c$}
    \put (94, 9) {$d$}
  \end{overpic}
\caption{The directed strong inversion $9_{46}a^-$ (left). The chosen half-axis is shown as a dotted line and the shaded surface is a butterfly surface $S$. On the right is a basis for $H_1(S)$.}
\label{fig:9_46}
\end{figure}

\begin{example}
\label{ex:invertible}
Consider the butterfly surface $S$ for the directed strongly invertible knot $(K,\tau) = 9_{46}a^-$ shown in Figure \ref{fig:9_46}. As in Definition \ref{def:strongly_invertible_gsig}, $\widetilde{\sigma}(K) = \widetilde{\sigma}(\Sigma(B^4,S),\widetilde{\tau})$, and by Proposition \ref{prop:GL_lattice_lift} we identify $\widetilde{\tau}_* \colon H_2(\Sigma(B^4,S)) \to H_2(\Sigma(B^4,S))$ with $\tau_*\colon H_1(S) \to H_1(S)$ (using $\mathbb{C}$ coefficients) where the intersection form is identified with the Gordon-Litherland pairing $\mathcal{G}_S$.

Take the basis $\{a,b,c,d\}$ for $H_1(S)$ shown in Figure \ref{fig:9_46}. The strong inversion exchanges $a$ with $-b$ and $c$ with $-d$. We will now compute the $g$-signature using Proposition \ref{prop:compute_gsig}. The $(+1)$-eigenspace $H(+1)$ of $\tau_*$ has basis $\{a-b,c-d\}$, and the $(-1)$-eigenspace $H(-1)$ of $\tau_*$ has basis $\{a+b,c+d\}$. We then have 
\[
   \restr{\mathcal{G}_S}{H(+1)}=
  \left[ {\begin{array}{cc}
   -4 & -2 \\
   -2 & -4 \\
  \end{array} } \right] \mbox{ and }
  \restr{\mathcal{G}_S}{H(-1)}=
  \left[ {\begin{array}{cc}
   4 & -2 \\
   -2 & 4 \\
  \end{array} } \right],
\]
where $\mathcal{G}_S$ is the Gordon-Litherland pairing. In particular, $\sigma(\restr{\mathcal{G}_S}{H(+1)}) = -2$ and $\sigma(\restr{\mathcal{G}_S}{H(-1)}) = 2$, so that $\widetilde{\sigma}(9_{46}a^-) = -4$. 

Now let $K_n$ be the equivariant connect sum of $K$ with itself $n$ times. Since $\widetilde{\sigma}$ is additive under equivariant connect sum by Proposition \ref{prop:gsigsum}, we see that $\widetilde{\sigma}(K_n) = -4n$. By Theorem \ref{thm:gsigbound}, $\bg_4(K) \geq 2n$, and since $S$ is a genus 2 butterfly surface for $K$, $\bg_4(K) \leq 2n$. Hence $\bg_4(K) = 2n$. However $K$ is (non-equivariantly) slice, so $K_n$ is as well. Thus the difference between the butterfly 4-genus and the non-equivariant 4-genus can be arbitrarily large.
\end{example}

\clearpage
\setcounter{footnote}{0}
\appendix

\section{Table of genera and invariants for strongly invertible knots}
\label{sec:table}
In this appendix we provide a table of strongly invertible concordance invariants for low-crossing directed strongly invertible knots. First, we summarize the notation used in the table.
\begin{itemize}
\item $\BL(K)$ is the butterfly link; see Definition \ref{def:butterfly_link}.
\item $\QBL(K)$ is the quotient butterfly link; see Definition \ref{def:quotient_butterfly_link}.
\item $\mathfrak{f}$ is the underlying knot, forgetting the strong inversion.
\item $\mathfrak{b}$ is one component of the butterfly link $\BL(K)$; see Theorem \ref{thm:bandqb}.
\item $\mathfrak{qb}$ is the non-axis component of the quotient butterfly link $\QBL(K)$; see Theorem \ref{thm:bandqb}.
\item $\lk(K)$ is the linking number between one component of $\BL(K)$ and the axis of symmetry; see Definition \ref{def:axis_linking}.
\item $\eta$ is Sakuma's $\eta$-polynomial \cite{Sakuma}. 
\item $\eta(\BL(K))$ is the Kojima-Yamasaki $\eta$-polynomial of the butterfly link; see Definition \ref{def:KY}.
\item $\eta(\QBL(K))$ is the Kojima-Yamasaki $\eta$-polynomial of the quotient butterfly link; see Definition \ref{def:KY}.
\item $g_4$ is the (non-equivariant) smooth 4-genus.
\item $\widetilde{g}_4$ is the equivariant 4-genus; see Definition \ref{def:equivariant_4genus}.
\item $\bg_4(K)$ is the butterfly 4-genus; see Definition \ref{def:butterfly_4genus}.
\item $\widetilde{\sigma}$ is the $g$-signature; see Definition \ref{def:strongly_invertible_gsig}.
\end{itemize}
The $\eta$-polynomial $\sum a_i t^i$ is denoted $[a_0,a_1,a_2,\dots]$ (recall that $a_i = a_{-i}$). A dash ``-'' indicates that the invariant is not defined. Specifically, $\eta(\QBL(K))$ is undefined when $\lk(K) \neq 0$, and $\widetilde{\sigma}(K)$ is undefined when $\lk(K) \neq 0$ or $\lk(K^-) \neq 0$. We compute $\eta(\BL(K))$ only in the cases where one component of $\BL(K)$ is the unknot, in which case the polynomial can be computed using the method in \cite{Sakuma}. The strong inversions in the table are given by rotation around a vertical axis, and the half-axes distinguished by the directions are indicated with a dotted line segment. The superscript $+$ and $-$ indicate opposite choices of half-axes (antipodes). 

Mirrors are not included since the invariants are essentially the same (see for example Proposition \ref{prop:reverse_mirror_invariants}). Similarly, the orientation on the axis is omitted since none of the numerical invariants depend on it. \\


\begin{longtable}{c|c||rc|rc|rcc}
 Name & Diagram&\multicolumn{1}{c}{}&\multicolumn{1}{c}{}&Invariants&\multicolumn{1}{c}{}&\multicolumn{1}{c}{}&\multicolumn{1}{c}{}\\* \cline{1-9}

 & \multirow{8}{*}{\scalebox{.6}{\includegraphics{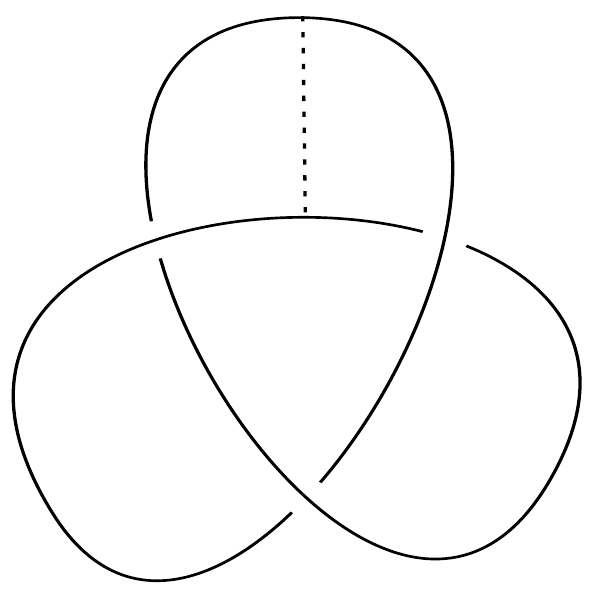}}}&&&&&&\\*
 $3_1^+$&&$\mathfrak{f}:$&$3_1$&&&$g_4:$&1&\\*\cline{3-9}
 &&&&&&&& \\*
 &&$\mathfrak{b}:$&$0_1$&$\eta:$&$[-2,0,1]$&$\widetilde{g}_4:$&1&\\*\cline{3-9}
 &&&&&&&& \\*
 &&$\mathfrak{qb}:$&$0_1$&$\eta(\BL):$&$[0]$&$\bg_4:$&$\infty$&\\*\cline{3-9}
 &&&&&&&& \\*
 &&$\lk:$&$2$&$\eta(\QBL):$&-&$\widetilde{\sigma}:$&-& \\\hline\hline

  & \multirow{8}{*}{\scalebox{.6}{\includegraphics{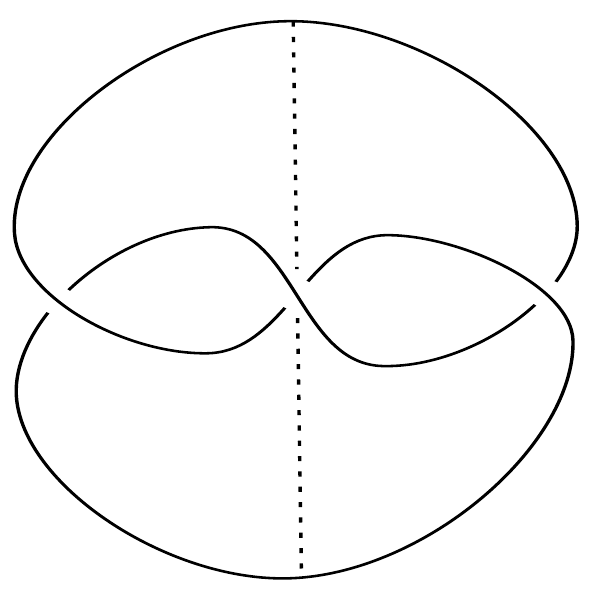}}}&&&&&&&\\*
 $3_1^-$&&$\mathfrak{f}:$&$3_1$&&&$g_4:$&1&\\*\cline{3-9}
 &&&&&&&& \\*
 &&$\mathfrak{b}:$&$0_1$&$\eta:$&$[-2,0,1]$&$\widetilde{g}_4:$&1&\\*\cline{3-9}
 &&&&&&&& \\*
 &&$\mathfrak{qb}:$&$m3_1$&$\eta(\BL):$&$[-2,0,1]$&$\bg_4:$&$\infty$&\\*\cline{3-9}
 &&&&&&&& \\*
 &&$\lk:$&$2$&$\eta(\QBL):$&-&$\widetilde{\sigma}:$&-& \\\hline\hline

  & \multirow{8}{*}{\scalebox{.6}{\includegraphics{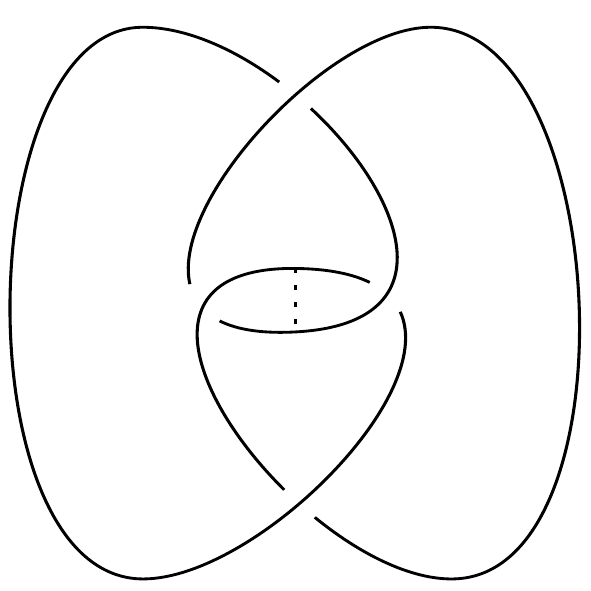}}}&&&&&&&\\*
 $4_1^+$&&$\mathfrak{f}:$&$4_1$&&&$g_4:$&1&\\*\cline{3-9}
 &&&&&&&& \\*
 &&$\mathfrak{b}:$&$0_1$&$\eta:$&$[-2,1,1,-1]$&$\widetilde{g}_4:$&1\\*\cline{3-9}
 &&&&&&&& \\*
 &&$\mathfrak{qb}:$&$0_1$&$\eta(\BL):$&$[0]$&$\bg_4:$&$\infty$&\\*\cline{3-9}
 &&&&&&&& \\*
 &&$\lk:$&$-2$&$\eta(\QBL):$&-&$\widetilde{\sigma}:$&-& \\\hline\hline

  & \multirow{8}{*}{\scalebox{.6}{\includegraphics{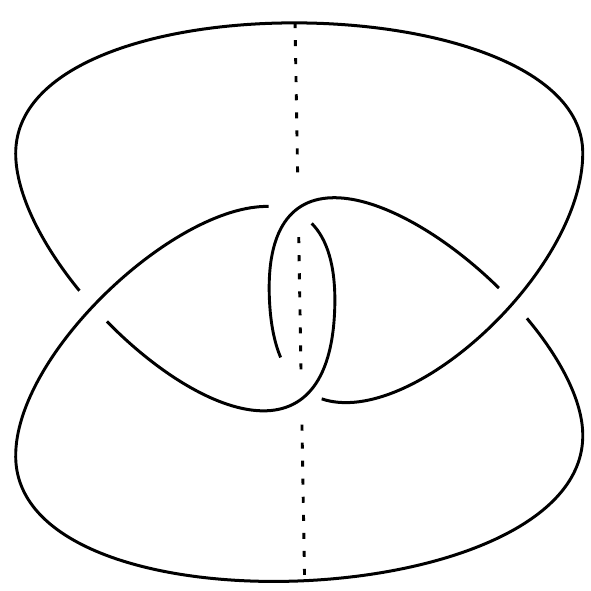}}}&&&&&&&\\*
 $4_1^-$&&$\mathfrak{f}:$&$4_1$&&&$g_4:$&1&\\*\cline{3-9}
 &&&&&&&& \\*
 &&$\mathfrak{b}:$&$3_1$&$\eta:$&$[-2,1,1,-1]$&$\widetilde{g}_4:$&1&\\*\cline{3-9}
 &&&&&&&& \\*
 &&$\mathfrak{qb}:$&$5_1$&$\eta(\BL):$&not computed&$\bg_4:$&$\infty$&\\*\cline{3-9}
 &&&&&&&& \\*
 &&$\lk:$&$2$&$\eta(\QBL):$&-&$\widetilde{\sigma}:$&-& \\\hline\hline

   & \multirow{8}{*}{\scalebox{.6}{\includegraphics{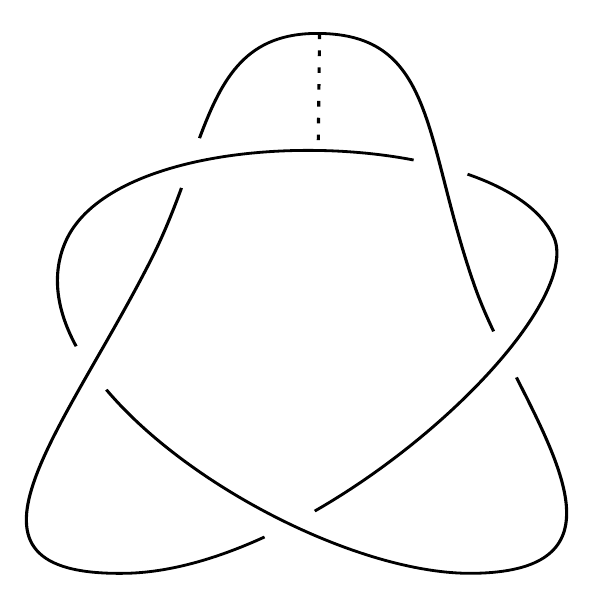}}}&&&&&&&\\*
 $5_1^+$&&$\mathfrak{f}:$&$5_1$&&&$g_4:$&2&\\*\cline{3-9}
 &&&&&&&& \\*
 &&$\mathfrak{b}:$&$0_1$&$\eta:$&$[-2,0,1]$&$\widetilde{g}_4:$&2&\\*\cline{3-9}
 &&&&&&&& \\*
 &&$\mathfrak{qb}:$&$0_1$&$\eta(\BL):$&$[0]$&$\bg_4:$&$\infty$&\\*\cline{3-9}
 &&&&&&&& \\*
 &&$\lk:$&$2$&$\eta(\QBL):$&-&$\widetilde{\sigma}:$&-& \\\hline\hline

 & \multirow{8}{*}{\scalebox{.6}{\includegraphics{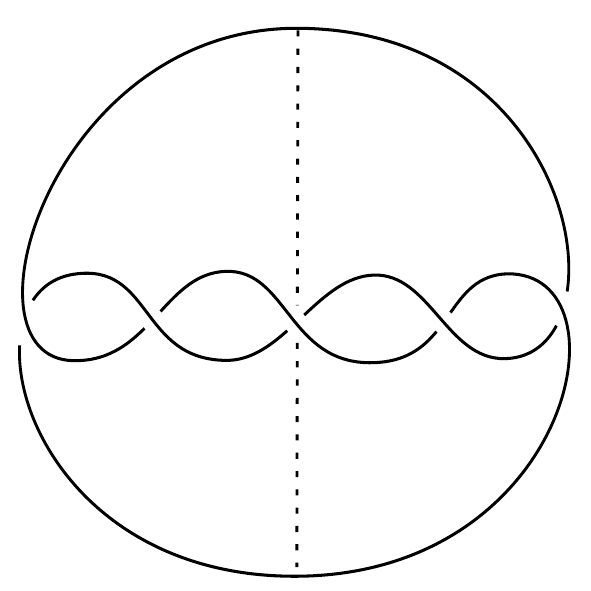}}}&&&&&&&\\*
 $5_1^-$&&$\mathfrak{f}:$&$5_1$&&&$g_4:$&2&\\*\cline{3-9}
 &&&&&&&& \\*
 &&$\mathfrak{b}:$&$0_1$&$\eta:$&$[-2,0,1]$&$\widetilde{g}_4:$&2&\\*\cline{3-9}
 &&&&&&&& \\*
 &&$\mathfrak{qb}:$&$4_1$&$\eta(\BL):$&$[0]$&$\bg_4:$&$\infty$&\\*\cline{3-9}
 &&&&&&&& \\*
 &&$\lk:$&$2$&$\eta(\QBL):$&-&$\widetilde{\sigma}:$&-& \\\hline\hline

& \multirow{8}{*}{\scalebox{.6}{\includegraphics{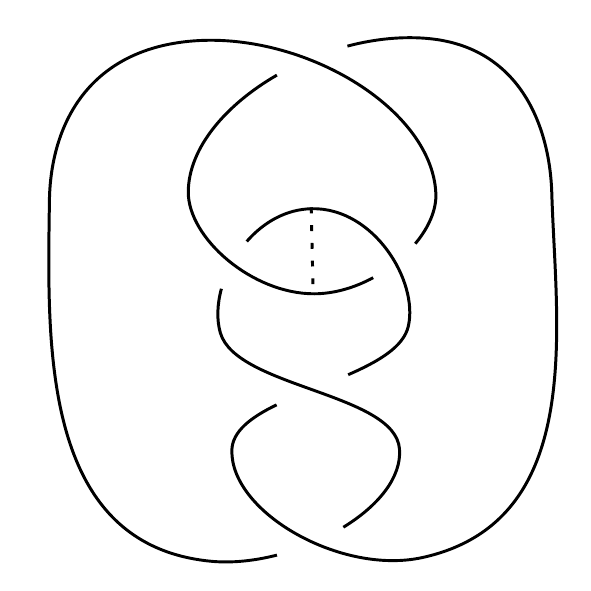}}}&&&&&&&\\*
 $5_2a^+$&&$\mathfrak{f}:$&$5_2$&&&$g_4:$&1&\\*\cline{3-9}
 &&&&&&&& \\*
 &&$\mathfrak{b}:$&$0_1$&$\eta:$&$[-4,1,2,-1]$&$\widetilde{g}_4:$&1&\\*\cline{3-9}
 &&&&&&&& \\*
 &&$\mathfrak{qb}:$&$0_1$&$\eta(\BL):$&$[-2,0,1]$&$\bg_4:$&$\infty$&\\*\cline{3-9}
 &&&&&&&& \\*
 &&$\lk:$&$4$&$\eta(\QBL):$&-&$\widetilde{\sigma}:$&-& \\\hline\hline

 & \multirow{8}{*}{\scalebox{.6}{\includegraphics{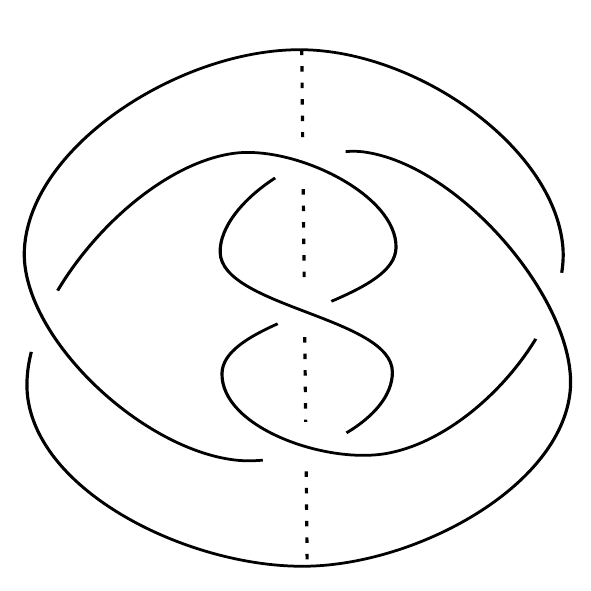}}}&&&&&&&\\*
 $5_2a^-$&&$\mathfrak{f}:$&$5_2$&&&$g_4:$&1&\\*\cline{3-9}
 &&&&&&&& \\*
 &&$\mathfrak{b}:$&$m3_1$&$\eta:$&$[-4,1,2,-1]$&$\widetilde{g}_4:$&1&\\*\cline{3-9}
 &&&&&&&& \\*
 &&$\mathfrak{qb}:$&$m7_1$&$\eta(\BL):$&not computed&$\bg_4:$&$\infty$&\\*\cline{3-9}
 &&&&&&&& \\*
 &&$\lk:$&$3$&$\eta(\QBL):$&-&$\widetilde{\sigma}:$&-& \\\hline\hline

  & \multirow{8}{*}{\scalebox{.6}{\includegraphics{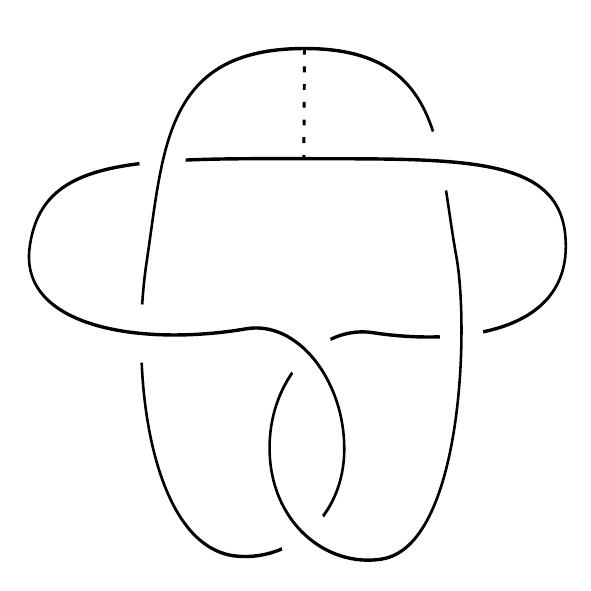}}}&&&&&&\\*
 $5_2b^+$&&$\mathfrak{f}:$&$5_2$&&&$g_4:$&1&\\*\cline{3-9}
 &&&&&&&& \\*
 &&$\mathfrak{b}:$&$0_1$&$\eta:$&$[-2,1,0,-1,1]$&$\widetilde{g}_4:$&1&\\*\cline{3-9}
 &&&&&&&& \\*
 &&$\mathfrak{qb}:$&$0_1$&$\eta(\BL):$&$[0]$&$\bg_4:$&$2$\hfootnote{The butterfly 4-genus $\widetilde{bg}_4$ does not depend on the choice of half-axis. Hence $\widetilde{bg}_4(5_2b^+)= \widetilde{bg}_4(5_2b^-) = 2$.}&\\*\cline{3-9}
 &&&&&&&& \\*
 &&$\lk:$&$0$&$\eta(\QBL):$&$[-2,0,1]$&$\widetilde{\sigma}:$&2\hfootnote{This follows from the computation of $\widetilde{\sigma}(5_2b^-)$, see Remark \ref{rmk:antipode_gsig_negate}.}& \\\hline\hline

  & \multirow{8}{*}{\scalebox{.6}{\includegraphics{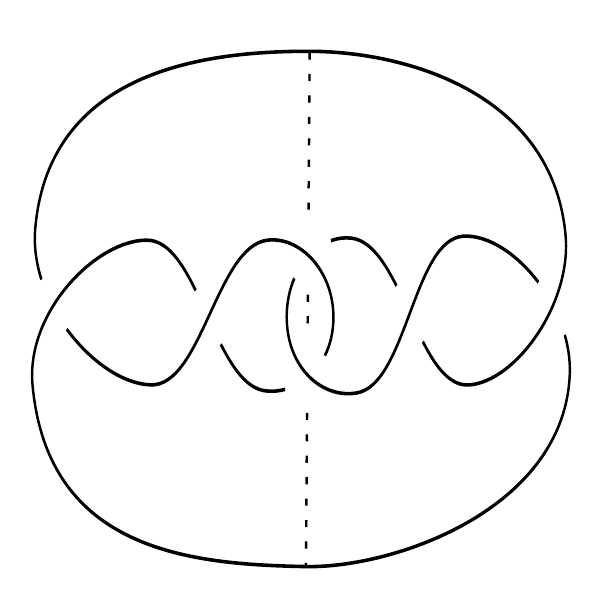}}}&&&&&&\\*
 $5_2b^-$&&$\mathfrak{f}:$&$5_2$&&&$g_4:$&1&\\*\cline{3-9}
 &&&&&&&& \\*
 &&$\mathfrak{b}:$&$3_1$&$\eta:$&$[-2,1,0,-1,1]$&$\widetilde{g}_4:$&1&\\*\cline{3-9}
 &&&&&&&& \\*
 &&$\mathfrak{qb}:$&$5_2$&$\eta(\BL):$&not computed&$\bg_4:$&$2$\hfootnote{The butterfly Seifert surface in Figure \ref{fig:52butterfly} (left) can be pushed equivariantly into $B^4$. To see that the symmetry on $5_2$ shown in Figure \ref{fig:52butterfly} depicts $5_2b^-$, it is sufficient to calculate $\mathfrak{b}$.}&\\*\cline{3-9}
 &&&&&&&& \\*
 &&$\lk:$&$0$&$\eta(\QBL):$&$[-2,0,1]$&$\widetilde{\sigma}:$&-2& \\\hline\hline

   & \multirow{8}{*}{\scalebox{.6}{\includegraphics{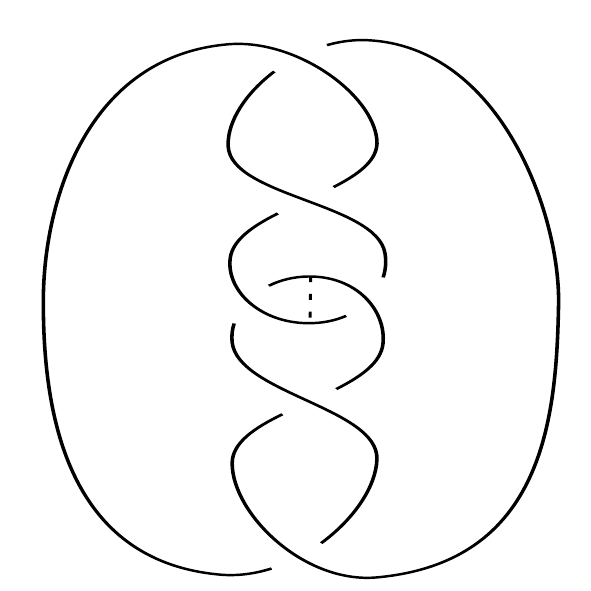}}}&&&&&&\\*
 $6_1a^+$&&$\mathfrak{f}:$&$6_1$&&&$g_4:$&0&\\*\cline{3-9}
 &&&&&&&& \\*
 &&$\mathfrak{b}:$&$0_1$&$\eta:$&$[-2,1,0,0,1,-1]$&$\widetilde{g}_4:$&1&\\*\cline{3-9}
 &&&&&&&& \\*
 &&$\mathfrak{qb}:$&$0_1$&$\eta(\BL):$&$[-2,0,1]$&$\bg_4:$&$\infty$&\\*\cline{3-9}
 &&&&&&&& \\*
 &&$\lk:$&$4$&$\eta(\QBL):$&$-$&$\widetilde{\sigma}:$&-& \\\hline\hline

  & \multirow{8}{*}{\scalebox{.6}{\includegraphics{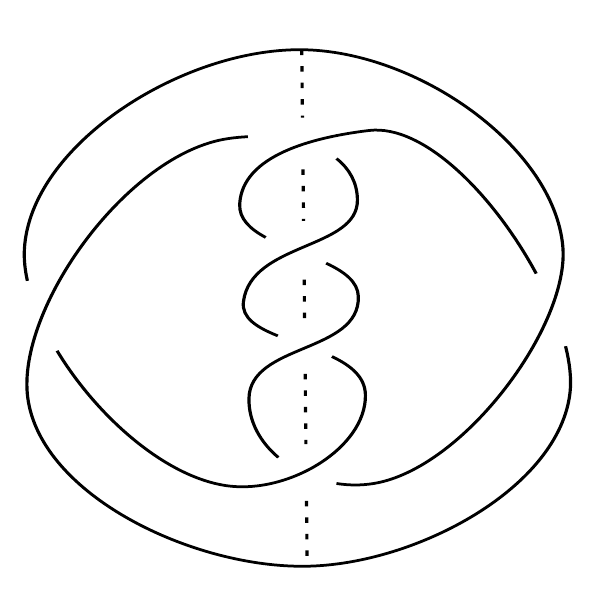}}}&&&&&&\\*
 $6_1a^-$&&$\mathfrak{f}:$&$6_1$&&&$g_4:$&0&\\*\cline{3-9}
 &&&&&&&& \\*
 &&$\mathfrak{b}:$&$m5_1$&$\eta:$&$[-2,1,0,0,1,-1]$&$\widetilde{g}_4:$&1&\\*\cline{3-9}
 &&&&&&&& \\*
 &&$\mathfrak{qb}:$&$9_1$&$\eta(\BL):$&not computed&$\bg_4:$&$\infty$&\\*\cline{3-9}
 &&&&&&&& \\*
 &&$\lk:$&$4$&$\eta(\QBL):$&$-$&$\widetilde{\sigma}:$&-& \\\hline\hline

   & \multirow{8}{*}{\scalebox{.6}{\includegraphics{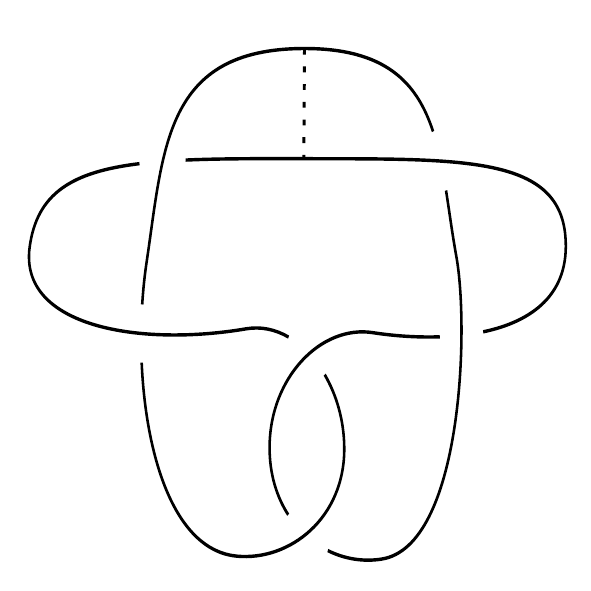}}}&&&&&&\\*
 $6_1b^+$&&$\mathfrak{f}:$&$6_1$&&&$g_4:$&0&\\*\cline{3-9}
 &&&&&&&& \\*
 &&$\mathfrak{b}:$&$0_1$&$\eta:$&$[4,-1,-2,1]$&$\widetilde{g}_4:$&1&\\*\cline{3-9}
 &&&&&&&& \\*
 &&$\mathfrak{qb}:$&$0_1$&$\eta(\BL):$&$[0]$&$\bg_4:$&$2$&\\*\cline{3-9}
 &&&&&&&& \\*
 &&$\lk:$&$0$&$\eta(\QBL):$&$[-2,0,1]$&$\widetilde{\sigma}:$&0\hfootnote{This follows from the computation of $\widetilde{\sigma}(6_1b^-)$, see Remark \ref{rmk:antipode_gsig_negate}.}& \\\hline\hline

  & \multirow{8}{*}{\scalebox{.6}{\includegraphics{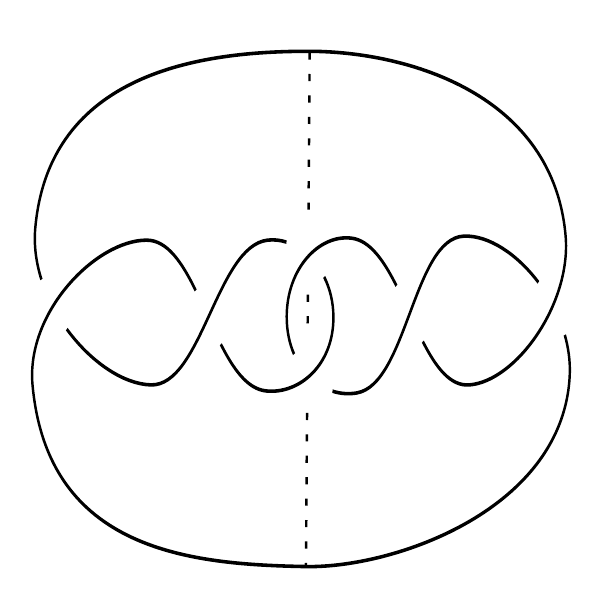}}}&&&&&&\\*
 $6_1b^-$&&$\mathfrak{f}:$&$6_1$&&&$g_4:$&0&\\*\cline{3-9}
 &&&&&&&& \\*
 &&$\mathfrak{b}:$&$4_1$&$\eta:$&$[4,-1,-2,1]$&$\widetilde{g}_4:$&1&\\*\cline{3-9}
 &&&&&&&& \\*
 &&$\mathfrak{qb}:$&$m5_2$&$\eta(\BL):$&not computed&$\bg_4:$&$2$\hfootnote{The butterfly Seifert surface in Figure \ref{fig:52butterfly} (right) can be pushed equivariantly into $B^4$. To see that the symmetry on $6_1$ shown in Figure \ref{fig:52butterfly} depicts $6_1b^-$, it is sufficient to calculate $\mathfrak{b}$.}&\\*\cline{3-9}
 &&&&&&&& \\*
 &&$\lk:$&$0$&$\eta(\QBL):$&$[-2,0,1]$&$\widetilde{\sigma}:$&0& \\\hline\hline

    & \multirow{8}{*}{\scalebox{.6}{\includegraphics{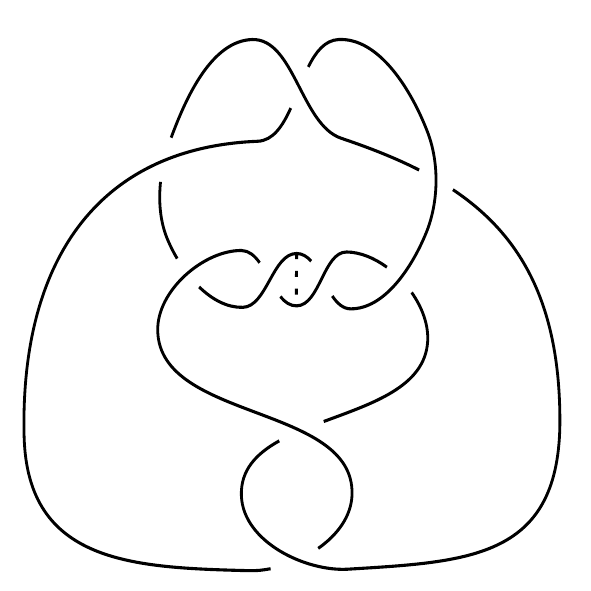}}}&&&&&&\\*
 $8_9^+$&&$\mathfrak{f}:$&$8_9$&&&$g_4:$&0&\\*\cline{3-9}
 &&&&&&&& \\*
 &&$\mathfrak{b}:$&$0_1$&$\eta:$&$[0]$&$\widetilde{g}_4:$&$1$&\\*\cline{3-9}
 &&&&&&&& \\*
 &&$\mathfrak{qb}:$&$0_1$&$\eta(\BL):$&$[2,0,-1]$&$\bg_4:$&$\infty$&\\*\cline{3-9}
 &&&&&&&& \\*
 &&$\lk:$&$0$&$\eta(\QBL):$&$[0]$&$\widetilde{\sigma}:$&-& \\\hline\hline

  & \multirow{8}{*}{\scalebox{.6}{\includegraphics{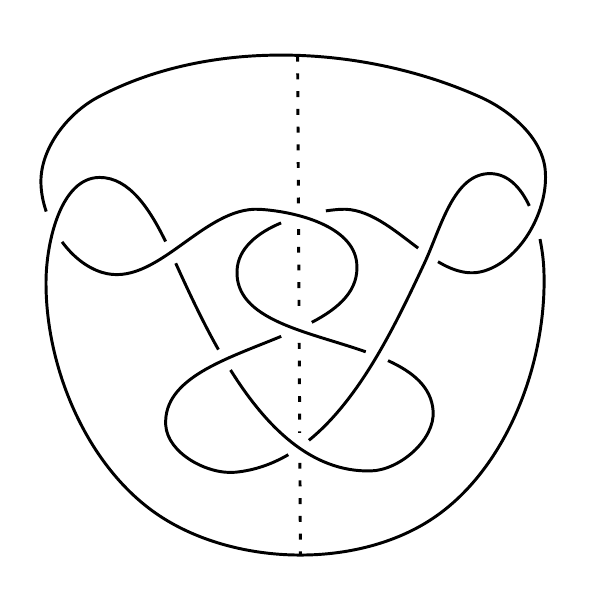}}}&&&&&&\\*
 $8_9^-$&&$\mathfrak{f}:$&$8_9$&&&$g_4:$&0&\\*\cline{3-9}
 &&&&&&&& \\*
 &&$\mathfrak{b}:$&$0_1$&$\eta:$&$[0]$&$\widetilde{g}_4:$&$1$&\\*\cline{3-9}
 &&&&&&&& \\*
 &&$\mathfrak{qb}:$&$m8_8$&$\eta(\BL):$&not computed&$\bg_4:$&$\infty$&\\*\cline{3-9}
 &&&&&&&& \\*
 &&$\lk:$&$-1$&$\eta(\QBL):$&$-$&$\widetilde{\sigma}:$&-& \\\hline\hline
\end{longtable}
\newpage
\setcounter{footnote}{0}
\section{Table of examples with non-trivial Donaldson's theorem obstruction} \label{sec:Donaldson_table}
In this appendix we give a table of all periodic and strongly invertible knots with a symmetric alternating diagram having at most 11 crossings for which Theorem \ref{thm:eq_embedding} provides an obstruction to the equivariant 4-genus $\widetilde{g}_4(K)$ being equal to  $|\sigma(K)|/2$. In all cases, we in fact have $g_4(K) = |\sigma(K)|/2$ so that $\widetilde{g}_4(K) > g_4(K)$. The notation $\star \to K'$ indicates that $K'$ is the knot obtained by equivariantly changing the crossings marked with a $\star$. We then compute upper bounds on $\widetilde{g}_4$ using Proposition \ref{prop:crossing_change_genus} and the computations of $\widetilde{g}_4$ in Appendix \ref{sec:table}.  The transvergent symmetries are shown as rotation around a vertical dotted line, and the intravergent symmetries are shown as rotation around a dot in the center of the diagram. Unlike in Appendix \ref{sec:table}, the strongly invertible knots here are not directed. \\

\begin{longtable}{l|c||c|l}
\hline \hline

\multicolumn{2}{c||}{$9_{28}:$ transvergent, periodic}&\multicolumn{2}{c}{$9_{28}:$ intravergent, strongly invertible}\\*\cline{1-4}
 &\multirow{12}{*}{\scalebox{.2}{\includegraphics{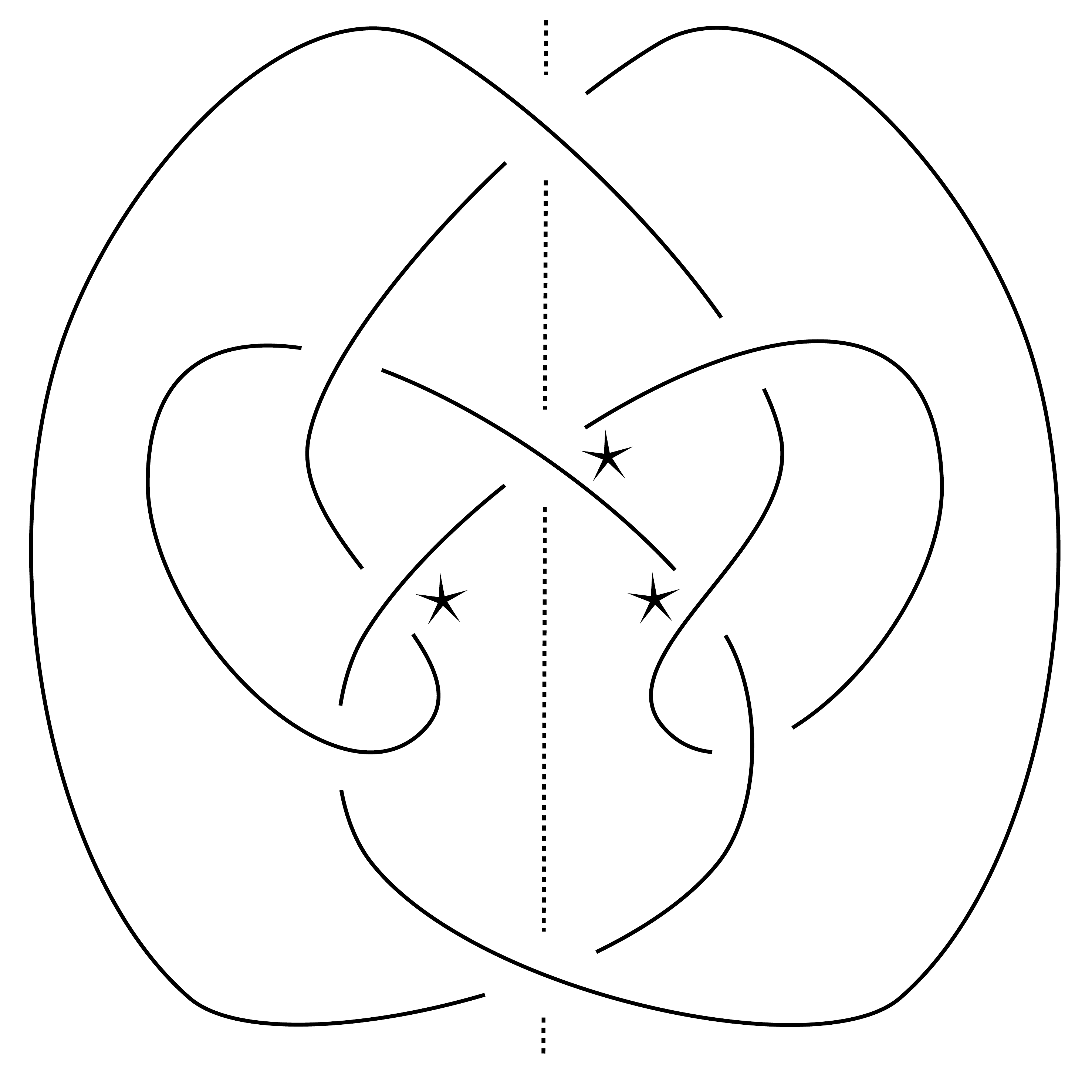}}}&\multirow{12}{*}{\scalebox{.2}{\includegraphics{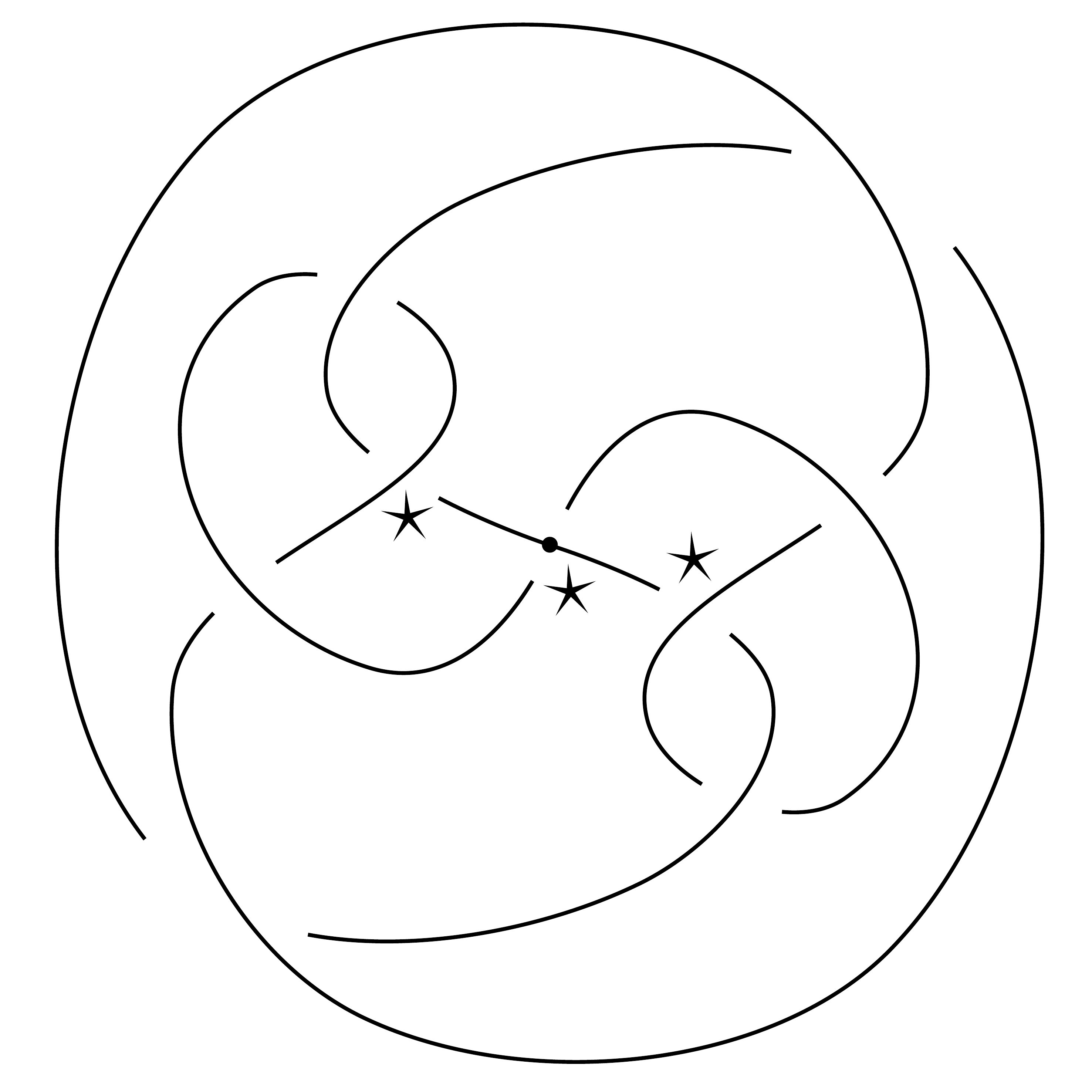}}}&\\*
 $\sigma : -2$&&&$\sigma : -2$\\*\cline{1-1}\cline{4-4}
 &&&\\*
 $g_4:\ 1$&&&$g_4:\ 1$ \\*\cline{1-1}\cline{4-4}
 &&& \\*
 $\widetilde{g}_4 \geq 2$&&&$\widetilde{g}_4 \geq 2$ \\*\cline{1-1}\cline{4-4}
 &&& \\*
 $\widetilde{g}_4\leq 3$&&&$\widetilde{g}_4 \leq 3$ \\*\cline{1-1}\cline{4-4}
 &&& \\*
 $\star \to$&&&$\star \to$ \\*
 unknot&&&unknot \\*
 &&&\\ \hline\hline 

\multicolumn{2}{c||}{$9_{29}:$ intravergent, strongly invertible}&\multicolumn{2}{c}{$9_{40}:$ intravergent, strongly invertible}\\*\cline{1-4}
 &\multirow{12}{*}{\scalebox{.2}{\includegraphics{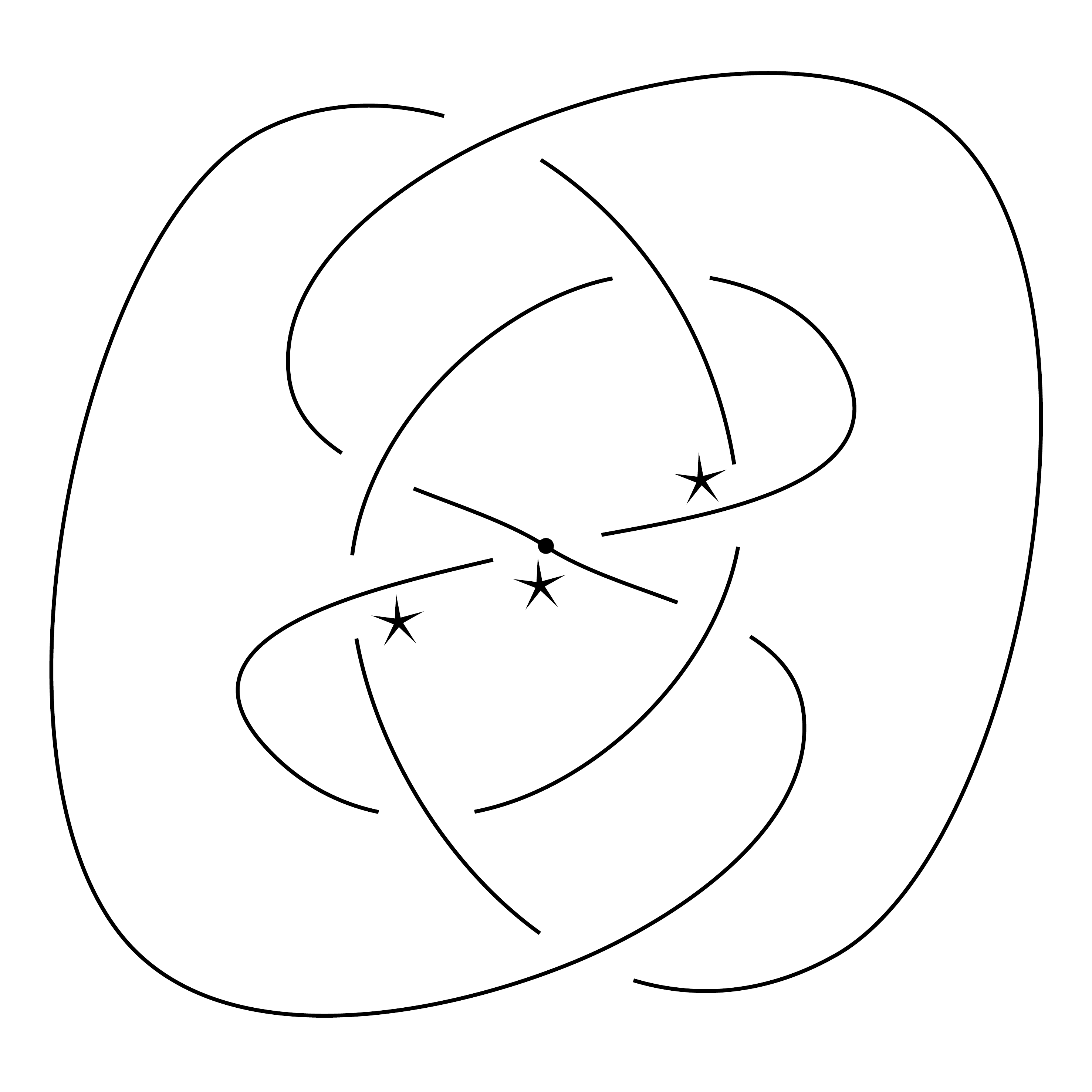}}}&\multirow{12}{*}{\scalebox{.2}{\includegraphics{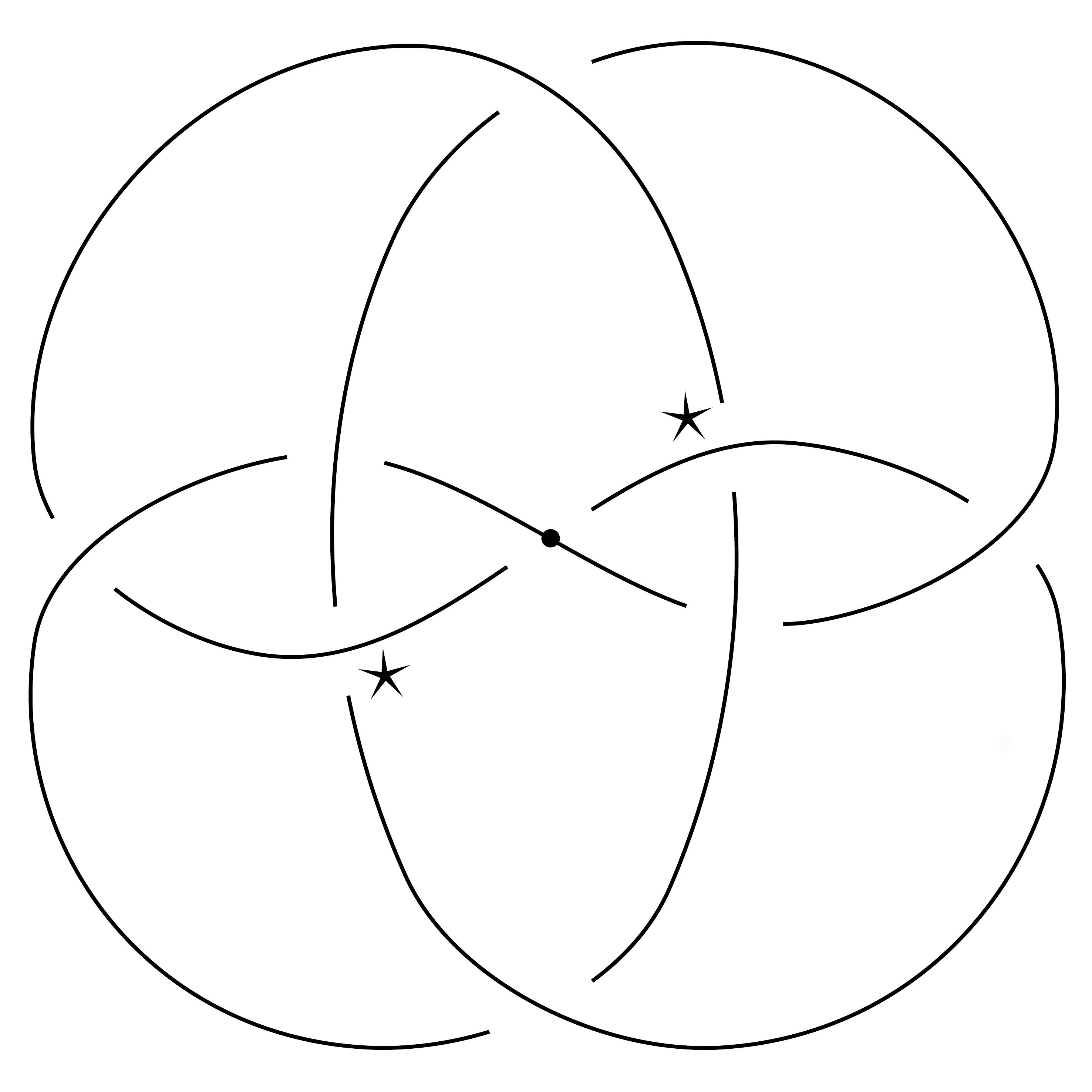}}}&\\*
 $\sigma : -2$&&&$\sigma : -2$\\*\cline{1-1}\cline{4-4}
 &&&\\*
 $g_4:\ 1$&&&$g_4:\ 1$ \\*\cline{1-1}\cline{4-4}
 &&& \\*
 $\widetilde{g}_4 \geq 2$&&&$\widetilde{g}_4 \geq 2$ \\*\cline{1-1}\cline{4-4}
 &&& \\*
 $\widetilde{g}_4\leq 3$&&&$\widetilde{g}_4 \leq 2$ \\*\cline{1-1}\cline{4-4}
 &&& \\*
 $\star \to$&&&$\star \to$ \\*
 unknot&&&unknot \\*
 &&&\\ \hline\hline 

 \multicolumn{2}{c||}{$9_{40}:$ transvergent, periodic}&\multicolumn{2}{c}{$9_{41}:$ transvergent, strongly invertible}\\*\cline{1-4}
 &\multirow{12}{*}{\scalebox{.2}{\includegraphics{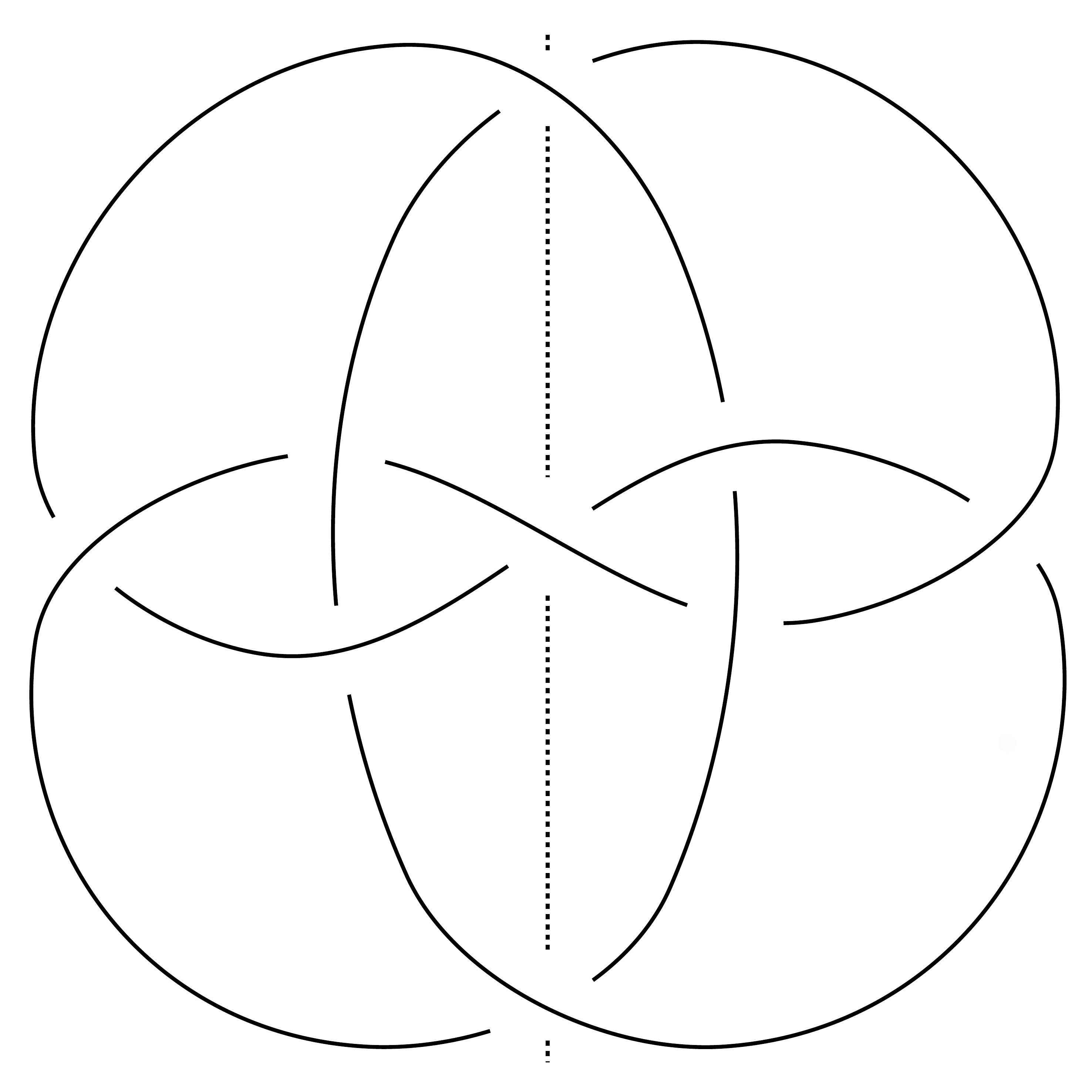}}}&\multirow{12}{*}{\scalebox{.2}{\includegraphics{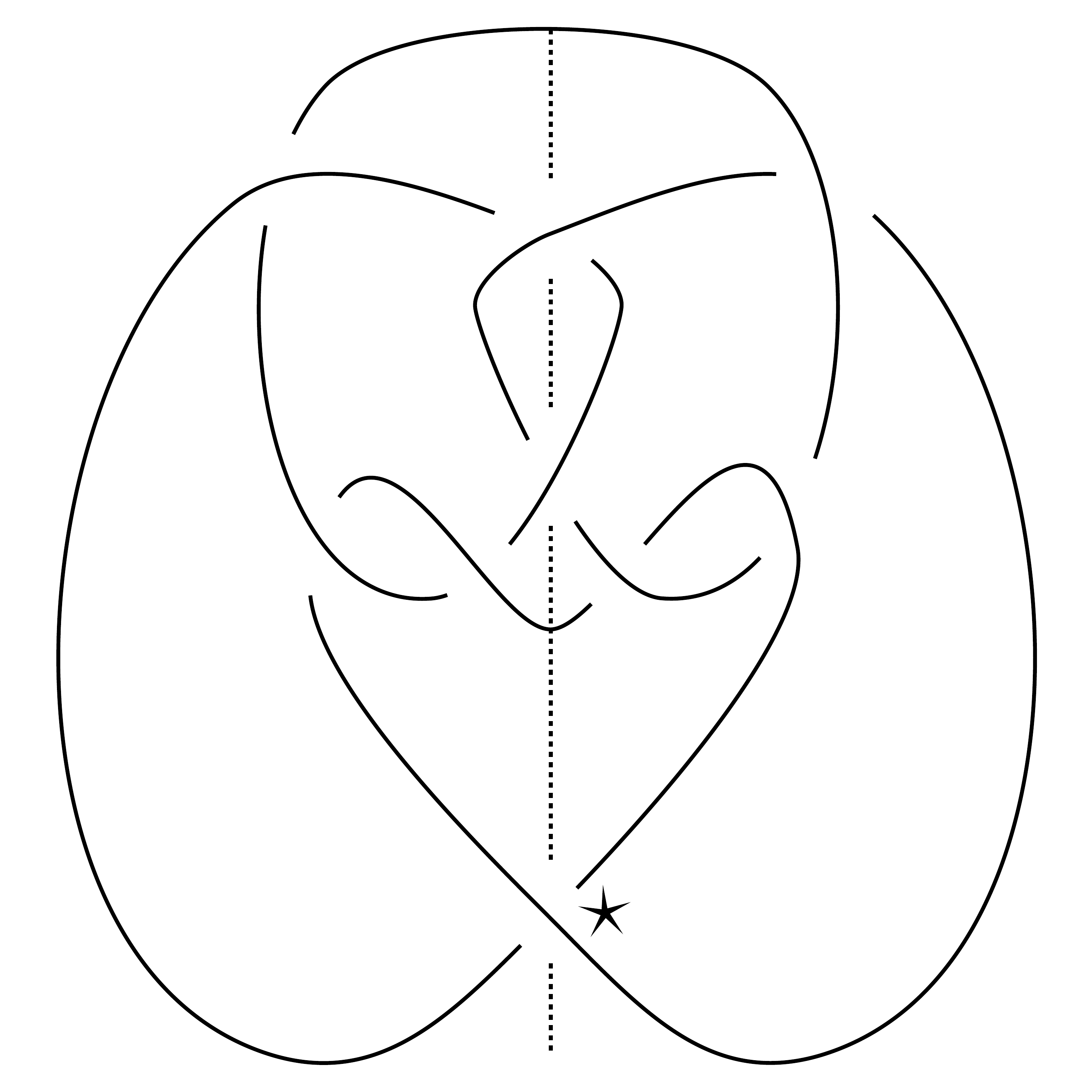}}}&\\*
 $\sigma : -2$&&&$\sigma : 0$\\*\cline{1-1}\cline{4-4}
 &&&\\*
 $g_4:\ 1$&&&$g_4:\ 0$ \\*\cline{1-1}\cline{4-4}
 &&& \\*
 $\widetilde{g}_4 \geq 3$&&&$\widetilde{g}_4 \geq 1\hfootnote{This was also shown in \cite[Appendix II]{Sakuma} and \cite[Theorem 1.11]{CorksInvolutions}.}$ \\*\cline{1-1}\cline{4-4}
 &&& \\*
 $\widetilde{g}_4\leq 3$&&&$\widetilde{g}_4 \leq 2$ \\*\cline{1-1}\cline{4-4}
 &&& \\*
 See&&&$\star \to 5_2$ \\*
 note.\footnote{The lower bound $\widetilde{g}_4 \geq 3$ is obtained by Proposition \ref{prop:RH}, and the upper bound $\widetilde{g}_4 \leq g_3 = 3$ is obtained by Edmonds' theorem \cite{Edmonds}.}&&& \\*
 &&&\\ \hline\hline 

\multicolumn{2}{c||}{$10_{64}:$ transvergent, strongly invertible}&\multicolumn{2}{c}{$10_{74}:$ transvergent, strongly invertible}\\*\cline{1-4}
 &\multirow{12}{*}{\scalebox{.2}{\includegraphics{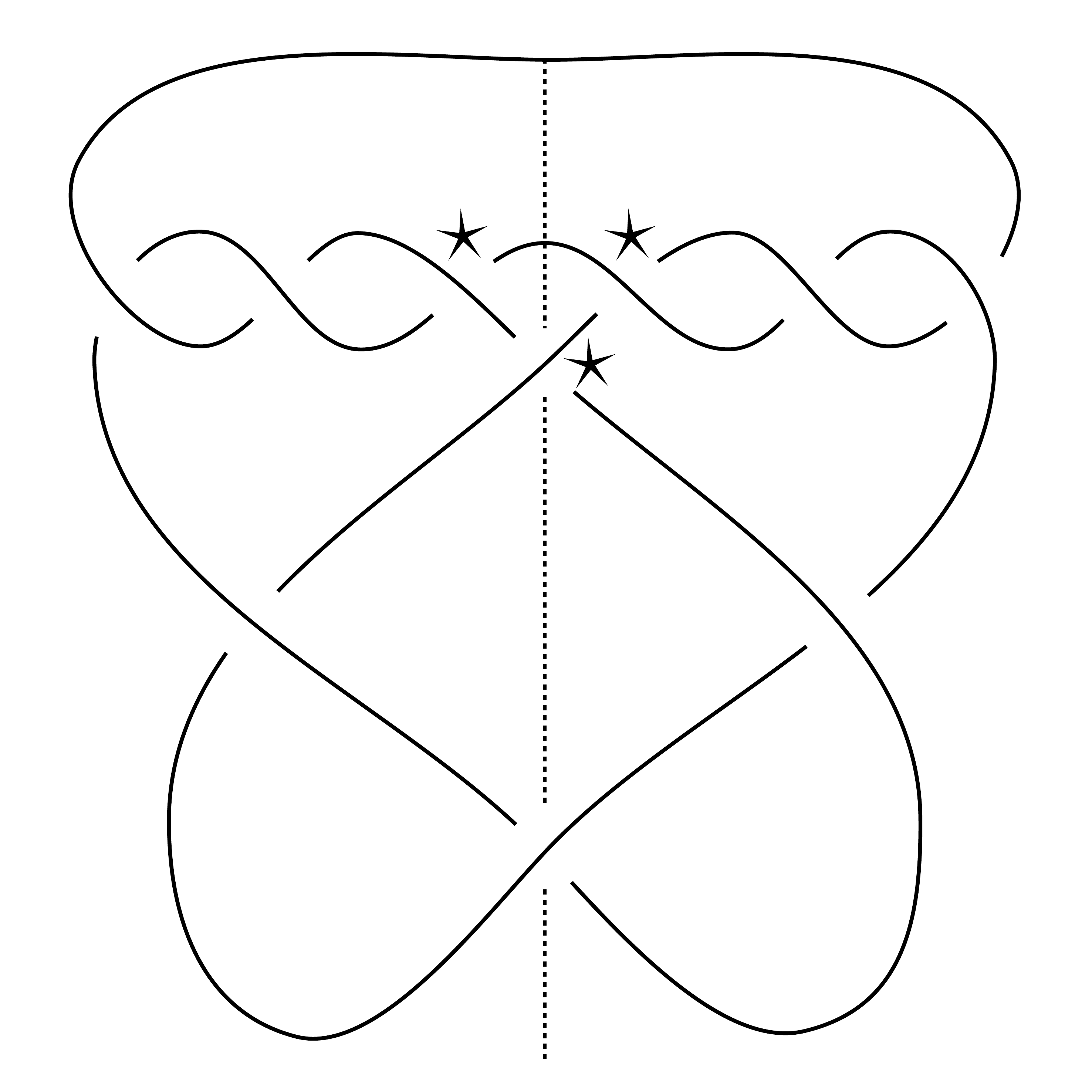}}}&\multirow{12}{*}{\scalebox{.2}{\includegraphics{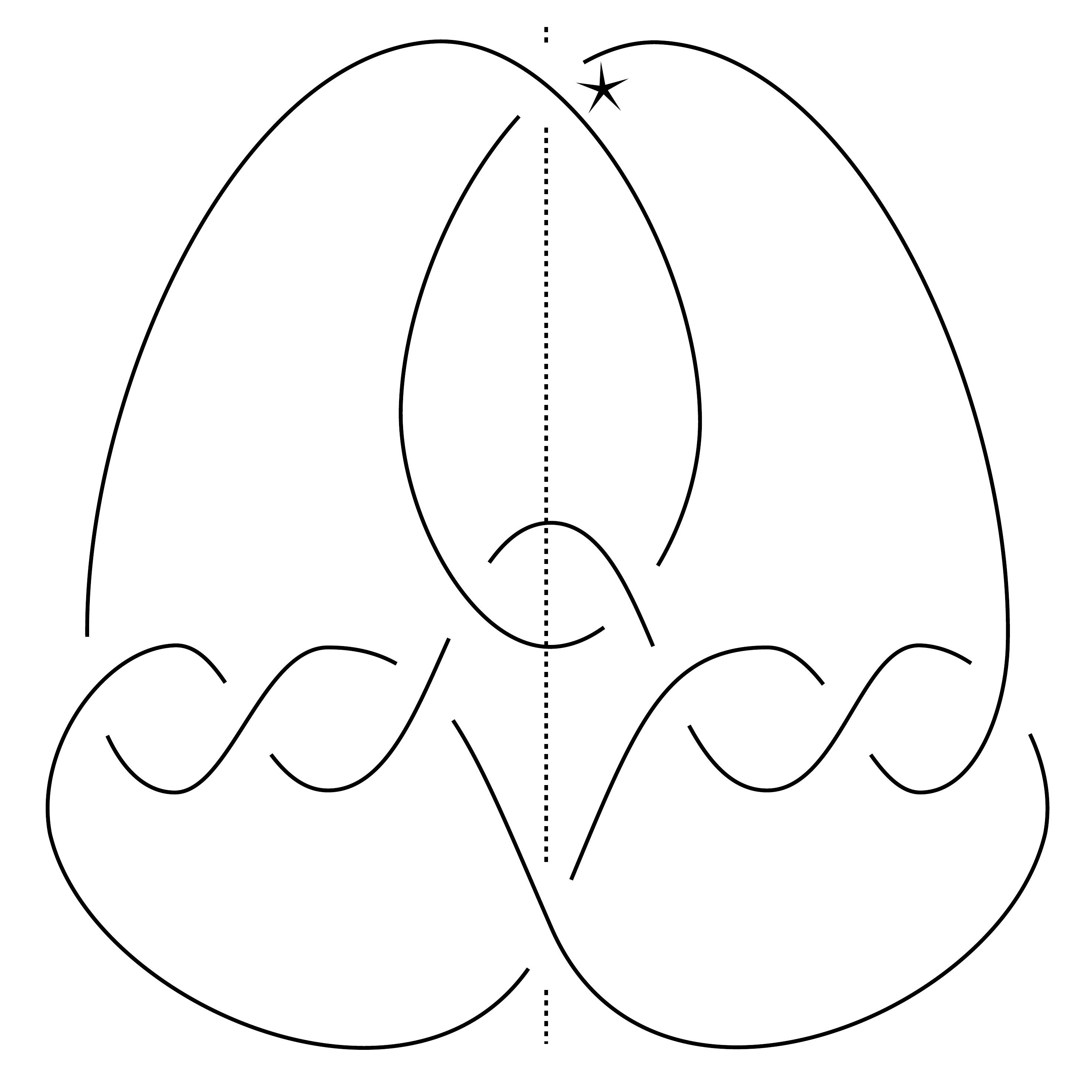}}}&\\*
 $\sigma : -2$&&&$\sigma : -2$\\*\cline{1-1}\cline{4-4}
 &&&\\*
 $g_4:\ 1$&&&$g_4:\ 1$ \\*\cline{1-1}\cline{4-4}
 &&& \\*
 $\widetilde{g}_4 \geq 2$&&&$\widetilde{g}_4 \geq 2$ \\*\cline{1-1}\cline{4-4}
 &&& \\*
 $\widetilde{g}_4\leq 3$&&&$\widetilde{g}_4 \leq 2$ \\*\cline{1-1}\cline{4-4}
 &&& \\*
 $\star \to$&&&$\star \to 3_1$ \\*
 unknot&&& \\*
 &&&\\ \hline\hline 

\multicolumn{2}{c||}{$10_{75}:$ transvergent, strongly invertible}&\multicolumn{2}{c}{$10_{103}:$ transvergent, strongly invertible}\\*\cline{1-4}
 &\multirow{12}{*}{\scalebox{.2}{\includegraphics{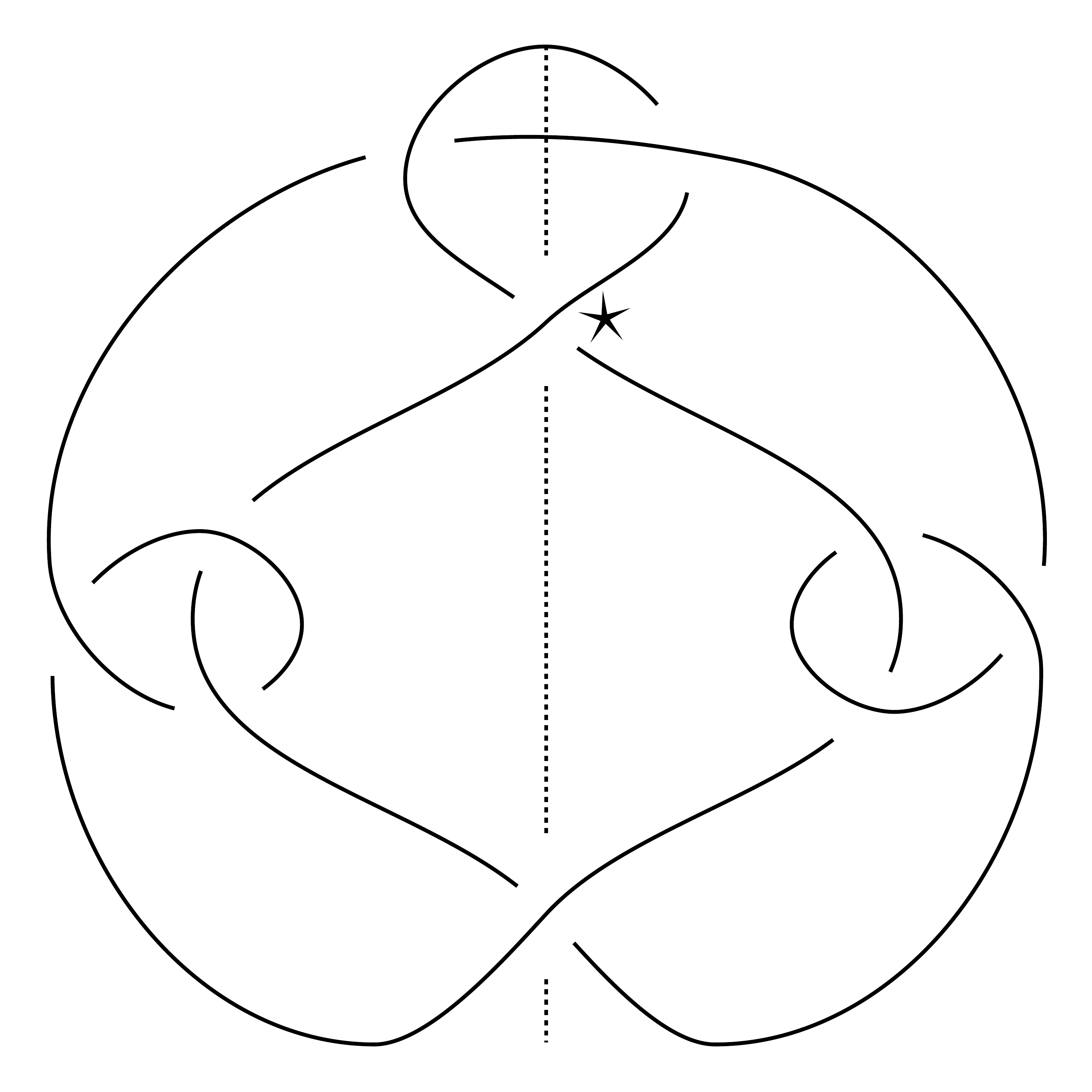}}}&\multirow{12}{*}{\scalebox{.2}{\includegraphics{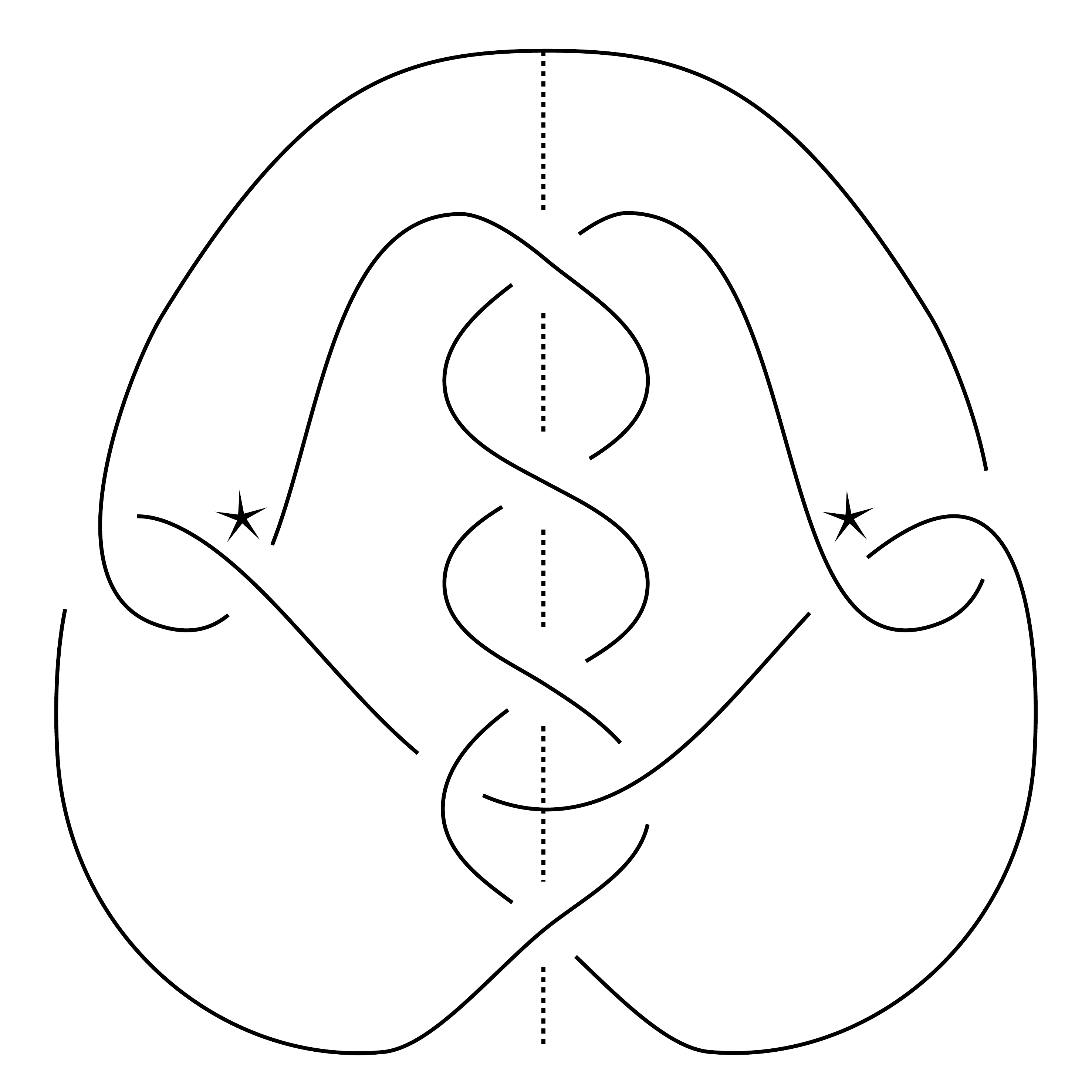}}}&\\*
 $\sigma : 0$&&&$\sigma : -2$\\*\cline{1-1}\cline{4-4}
 &&&\\*
 $g_4:\ 0$&&&$g_4:\ 1$ \\*\cline{1-1}\cline{4-4}
 &&& \\*
 $\widetilde{g}_4 \geq 1\, \ufootnote{This was also shown in \cite[Theorem 1.11]{CorksInvolutions}.}$&&&$\widetilde{g}_4 \geq 2$ \\*\cline{1-1}\cline{4-4}
 &&& \\*
 $\widetilde{g}_4\leq 2$&&&$\widetilde{g}_4 \leq 3$ \\*\cline{1-1}\cline{4-4}
 &&& \\*
 $\star \to 3_1$&&&$\star \to 3_1$ \\*
 &&& \\*
 &&& \\ \hline\hline

\multicolumn{2}{c||}{$10_{123}:$ transvergent, strongly invertible}&\multicolumn{2}{c}{$11_{125}:$ intravergent, strongly invertible}\\*\cline{1-4}
 &\multirow{12}{*}{\scalebox{.2}{\includegraphics{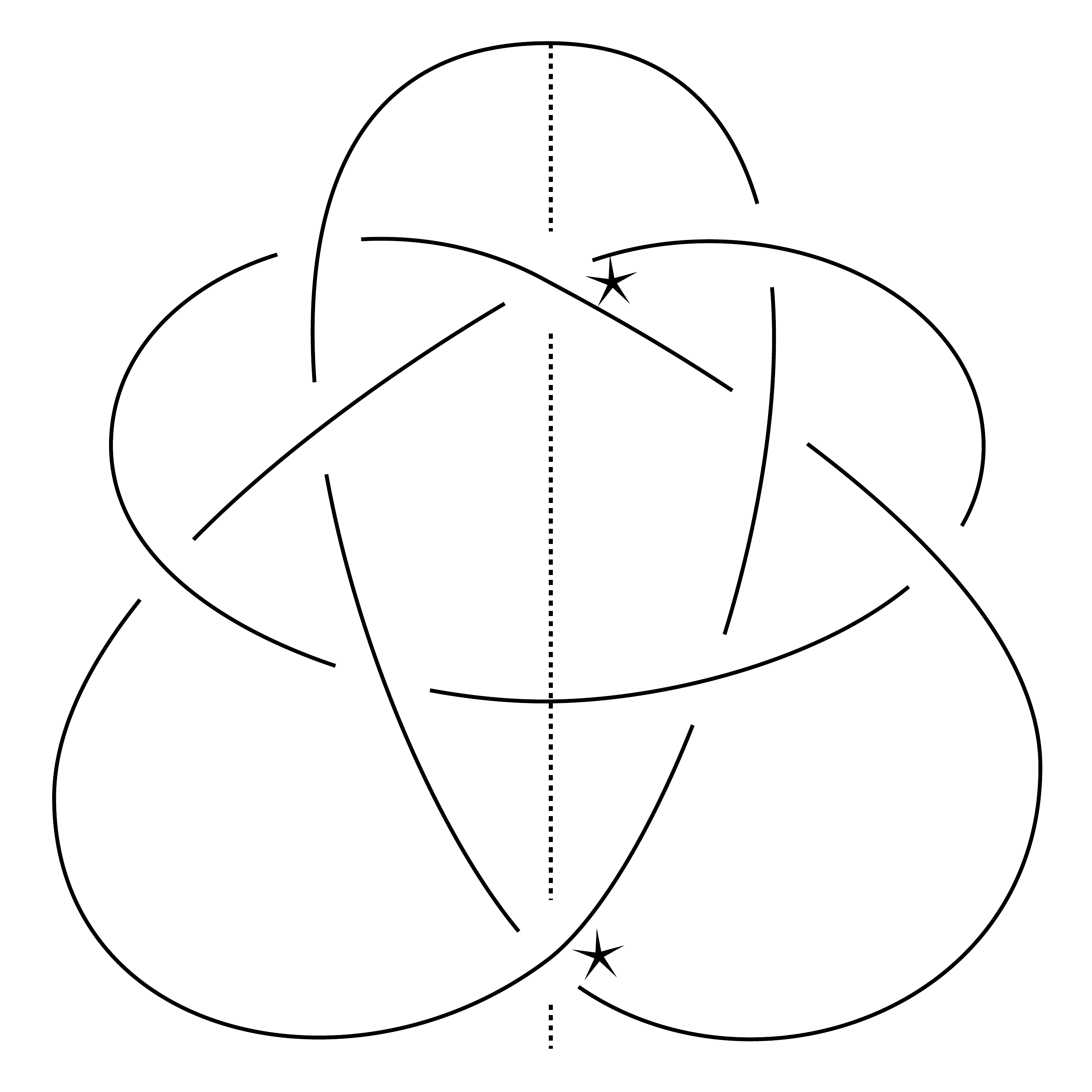}}}&\multirow{12}{*}{\scalebox{.2}{\includegraphics{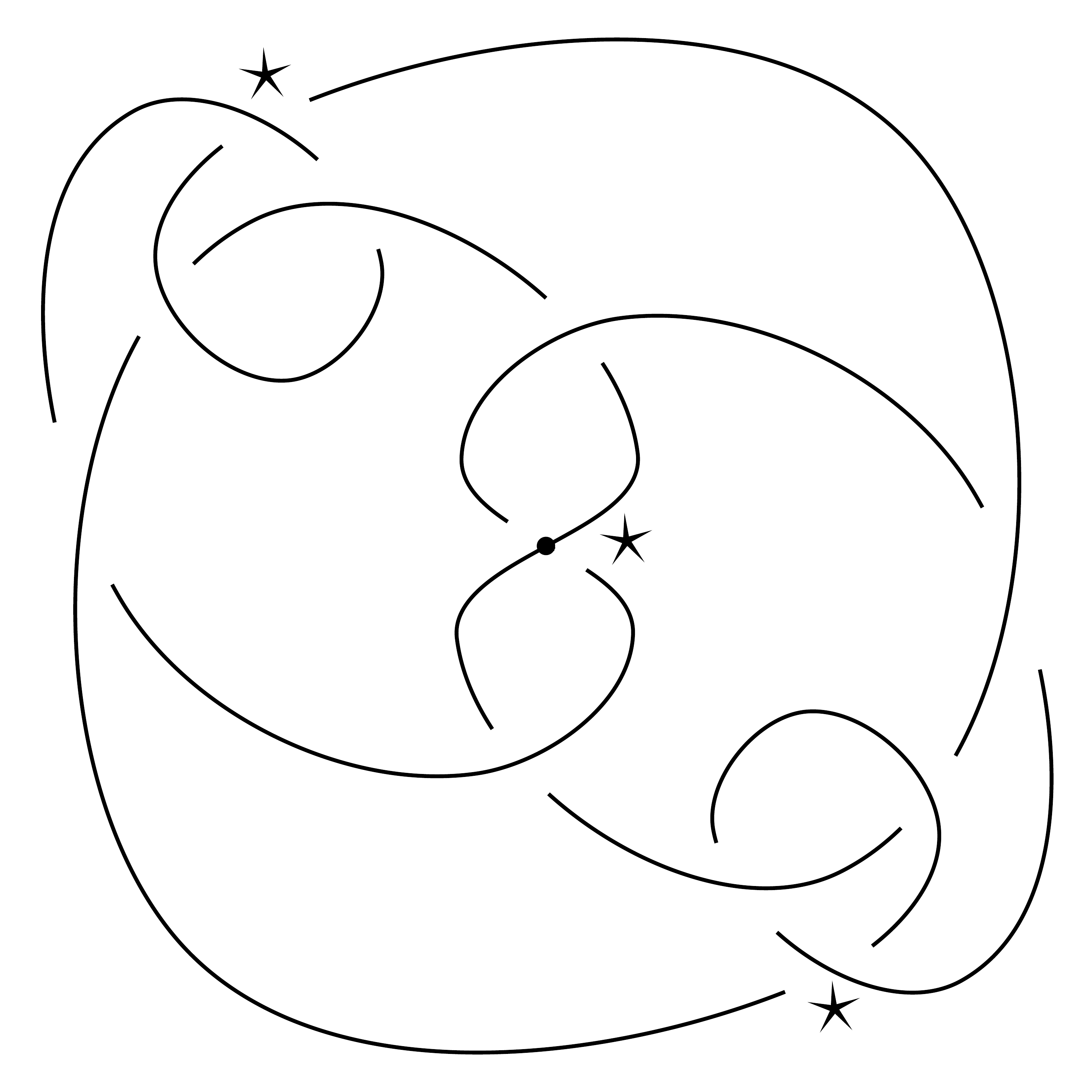}}}&\\*
 $\sigma : 0$&&&$\sigma : -2$\\*\cline{1-1}\cline{4-4}
 &&&\\*
 $g_4:\ 0$&&&$g_4:\ 1$ \\*\cline{1-1}\cline{4-4}
 &&& \\*
 $\widetilde{g}_4 \geq 1$&&&$\widetilde{g}_4 \geq 2$ \\*\cline{1-1}\cline{4-4}
 &&& \\*
 $\widetilde{g}_4\leq 2$&&&$\widetilde{g}_4 \leq 3$ \\*\cline{1-1}\cline{4-4}
 &&& \\*
 $\star \to$&&&$\star \to$ \\*
 unknot&&&unknot \\*
 &&& \\ \hline\hline

\multicolumn{2}{c||}{$11_{125}:$ transvergent, strongly invertible}&\multicolumn{2}{c}{$11_{157}:$ intravergent, strongly invertible}\\*\cline{1-4}
 &\multirow{12}{*}{\scalebox{.2}{\includegraphics{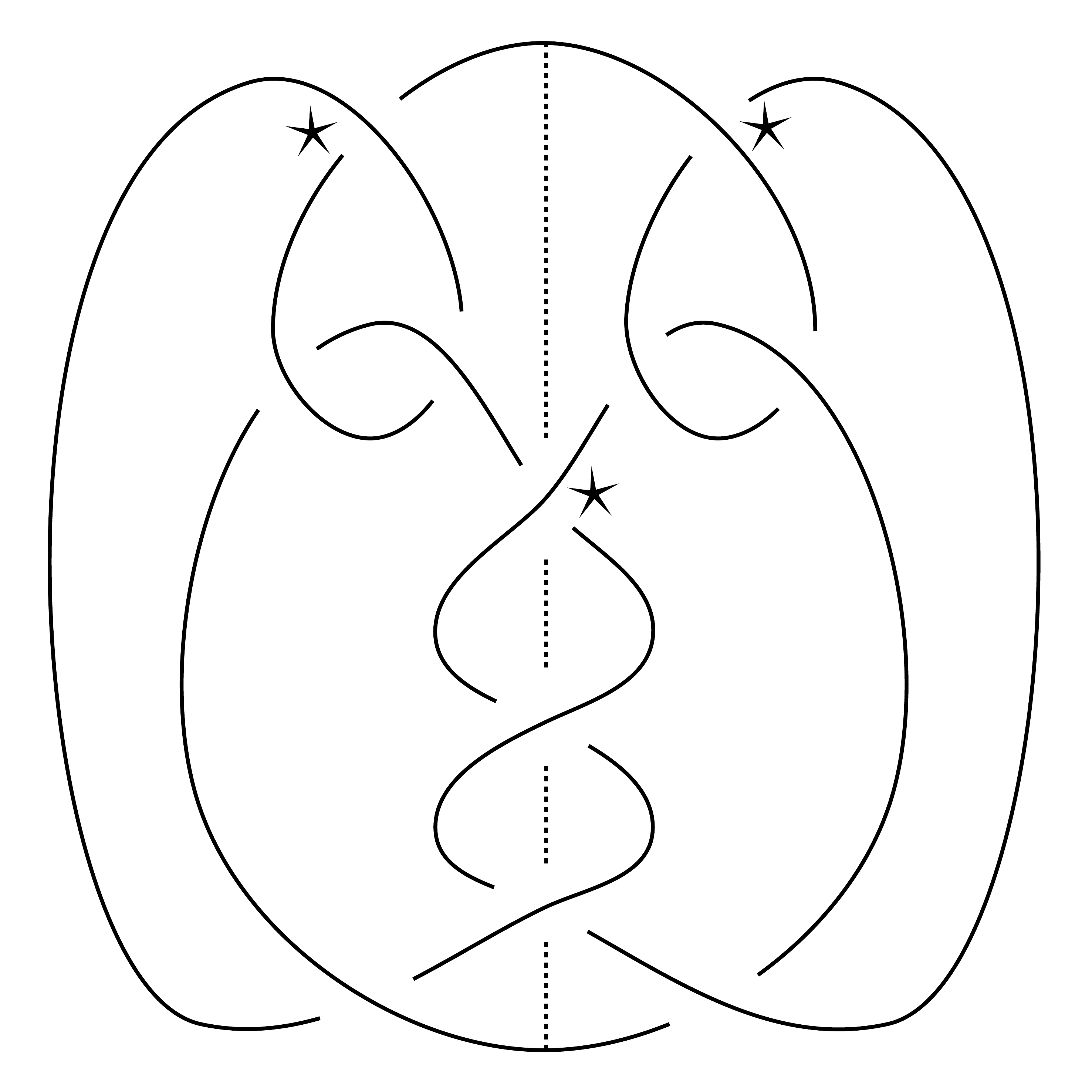}}}&\multirow{12}{*}{\scalebox{.2}{\includegraphics{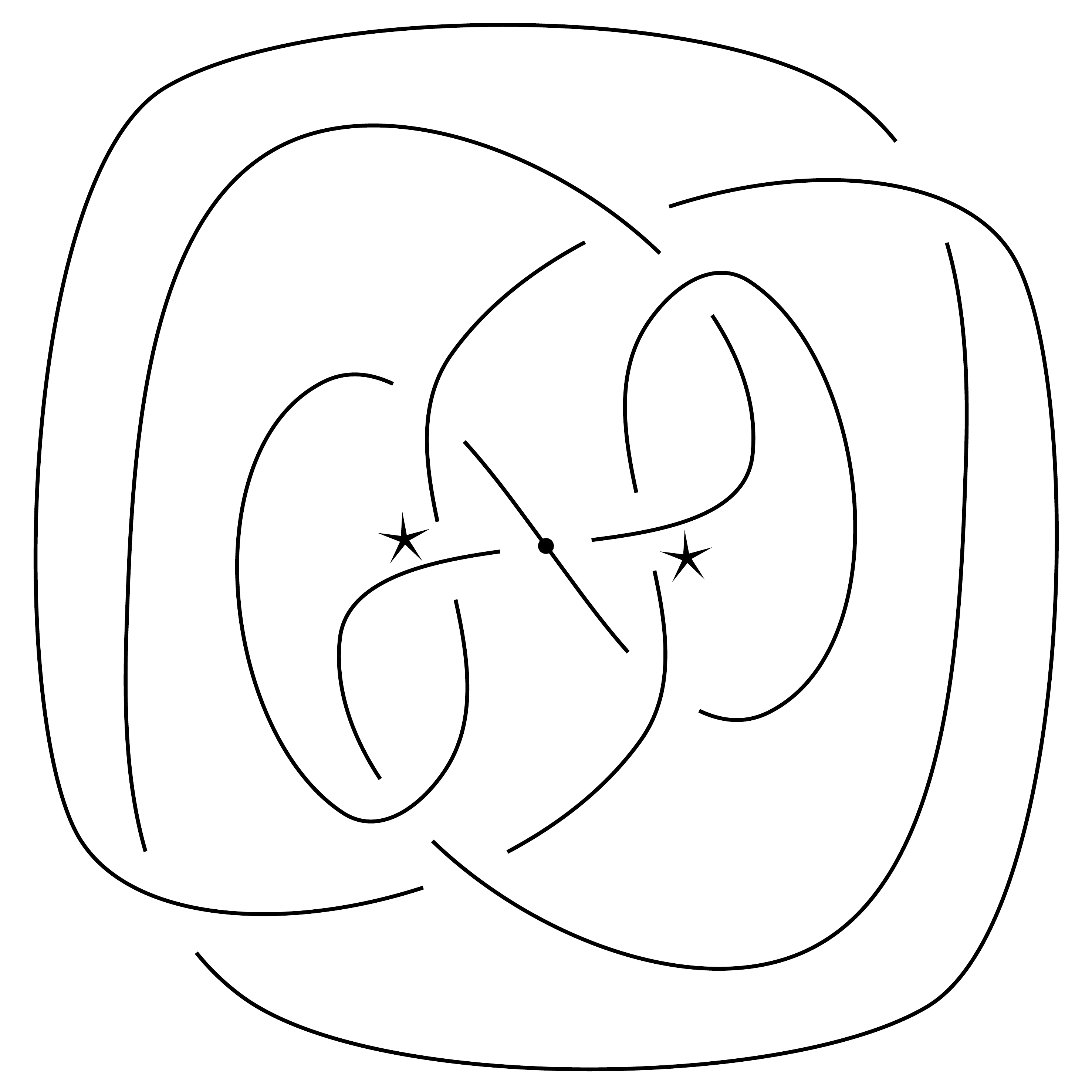}}}&\\*
 $\sigma : -2$&&&$\sigma : -2$\\*\cline{1-1}\cline{4-4}
 &&&\\*
 $g_4:\ 1$&&&$g_4:\ 1$ \\*\cline{1-1}\cline{4-4}
 &&& \\*
 $\widetilde{g}_4 \geq 2$&&&$\widetilde{g}_4 \geq 2$ \\*\cline{1-1}\cline{4-4}
 &&& \\*
 $\widetilde{g}_4\leq 3$&&&$\widetilde{g}_4 \leq 3$ \\*\cline{1-1}\cline{4-4}
 &&& \\*
 $\star \to$&&&$\star \to 5_2$ \\*
 unknot&&& \\*
 &&& \\ \hline\hline

\multicolumn{2}{c||}{$11_{157}:$ transvergent, strongly invertible}&\multicolumn{2}{c}{$11_{161}:$ intravergent, strongly invertible}\\*\cline{1-4}
 &\multirow{12}{*}{\scalebox{.2}{\includegraphics{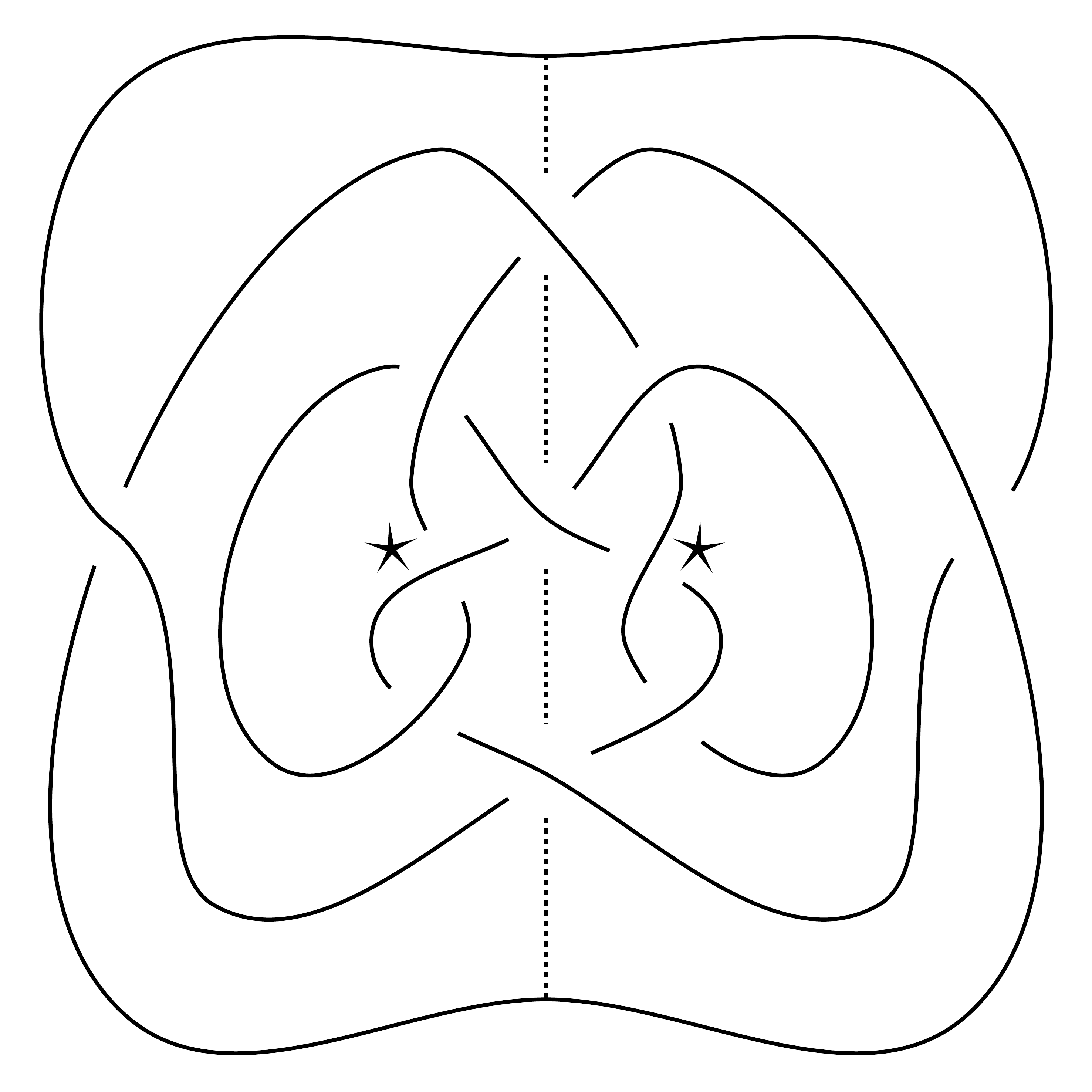}}}&\multirow{12}{*}{\scalebox{.2}{\includegraphics{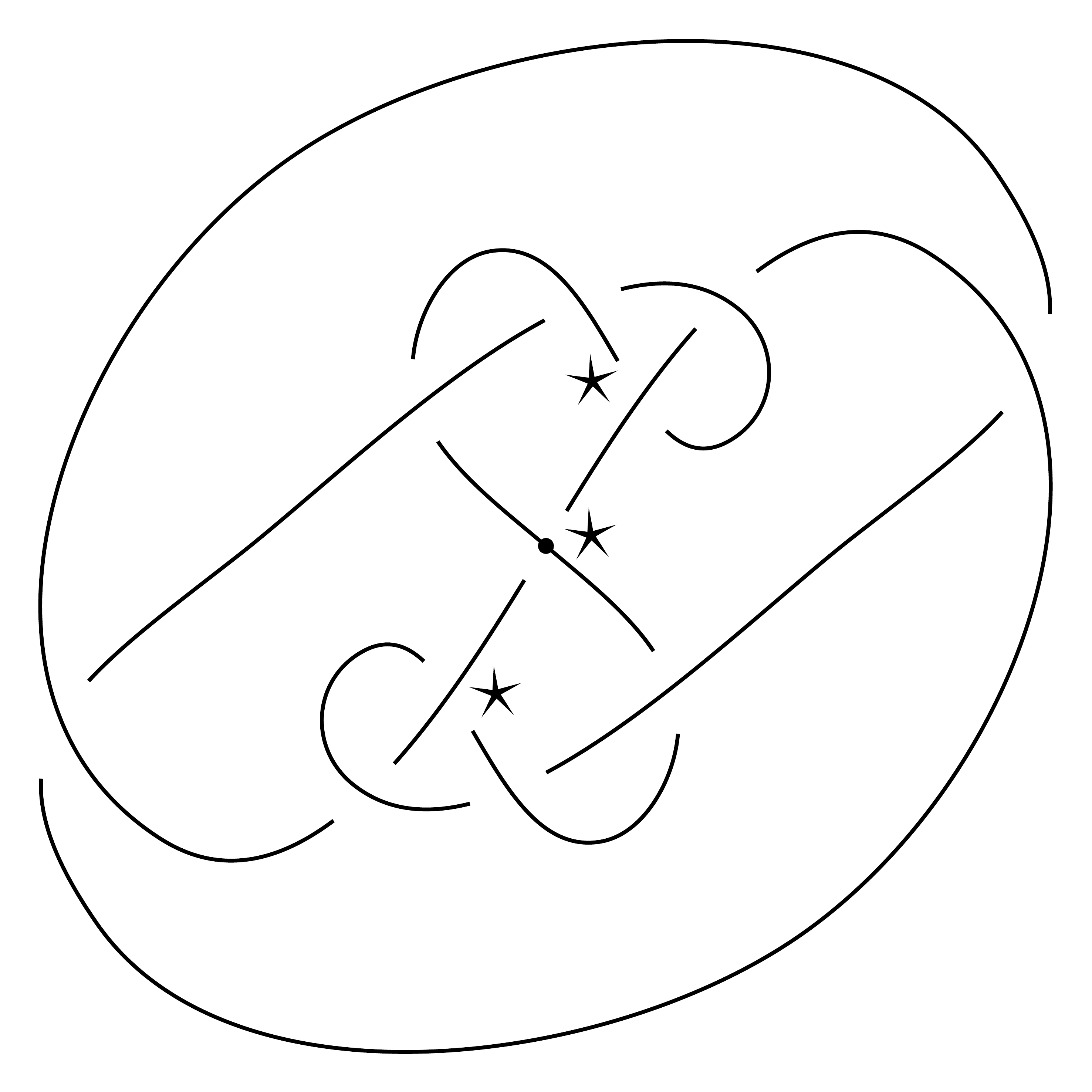}}}&\\*
 $\sigma : -2$&&&$\sigma : -2$\\*\cline{1-1}\cline{4-4}
 &&&\\*
 $g_4:\ 1$&&&$g_4:\ 1$ \\*\cline{1-1}\cline{4-4}
 &&& \\*
 $\widetilde{g}_4 \geq 2$&&&$\widetilde{g}_4 \geq 2$ \\*\cline{1-1}\cline{4-4}
 &&& \\*
 $\widetilde{g}_4\leq 3$&&&$\widetilde{g}_4 \leq 3$ \\*\cline{1-1}\cline{4-4}
 &&& \\*
 $\star \to 5_2$&&&$\star \to$ \\*
 &&&unknot \\*
 &&& \\ \hline\hline

\multicolumn{2}{c||}{$11_{161}:$ transvergent, periodic}&\multicolumn{2}{c}{$11_{192}:$ intravergent, strongly invertible}\\*\cline{1-4}
 &\multirow{12}{*}{\scalebox{.2}{\includegraphics{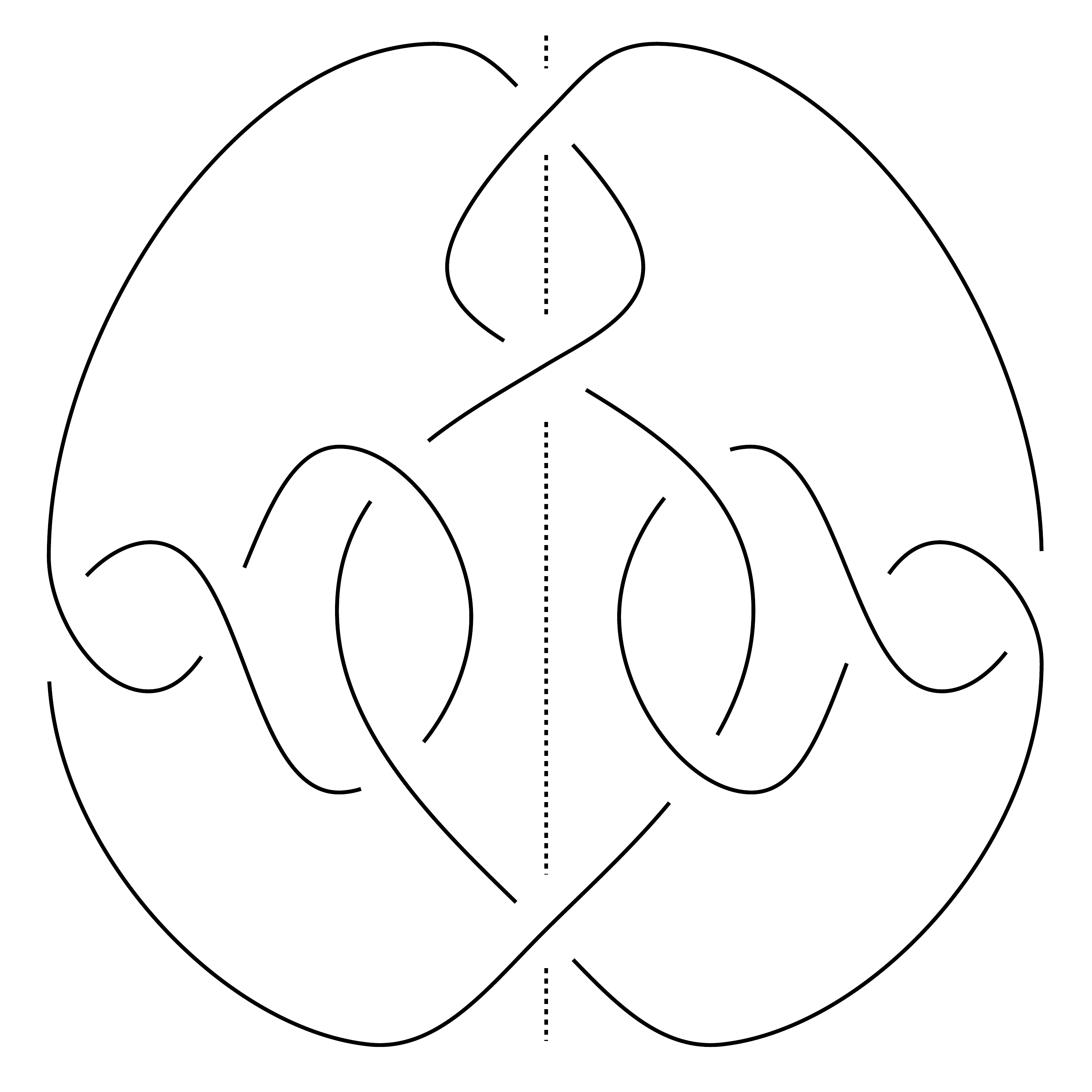}}}&\multirow{12}{*}{\scalebox{.2}{\includegraphics{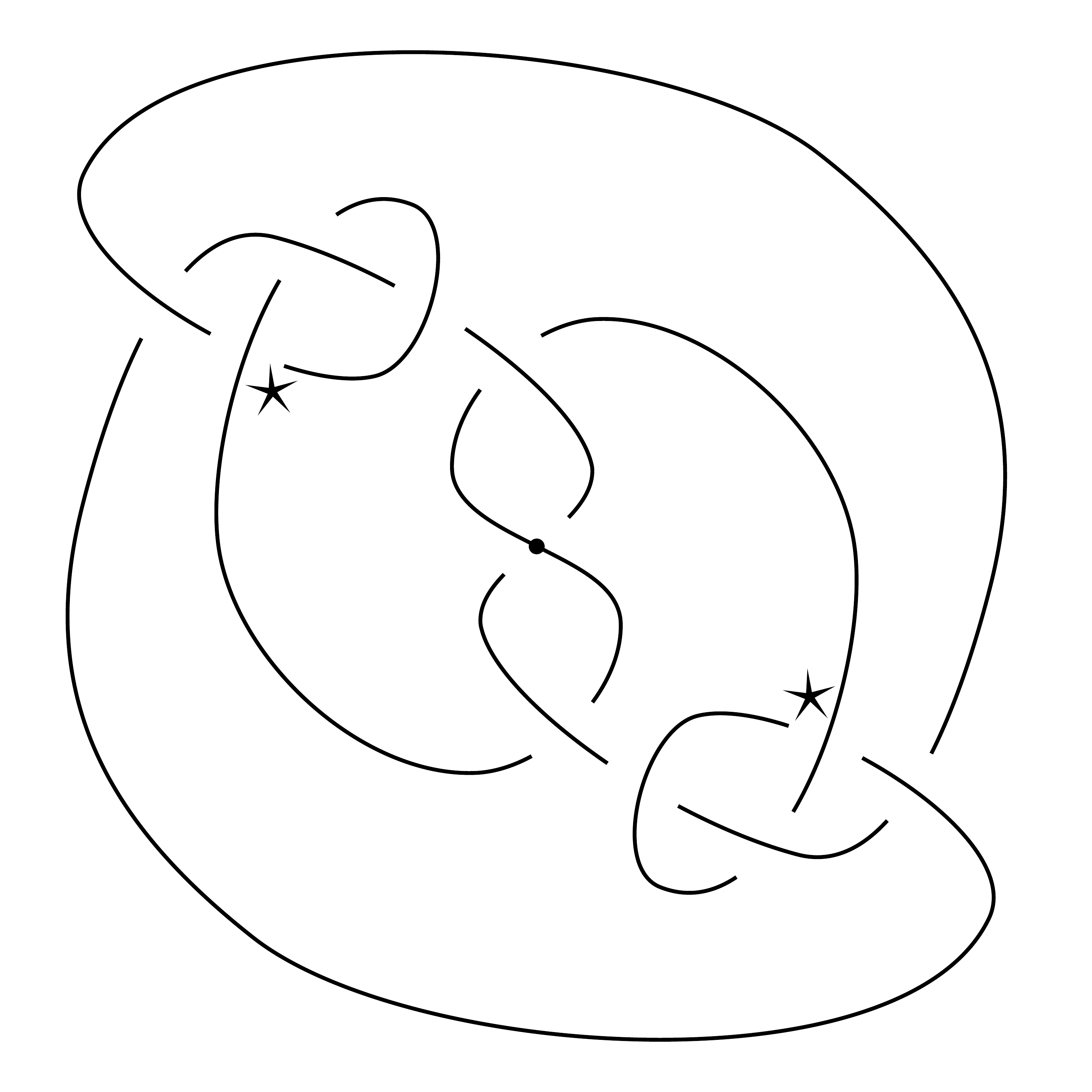}}}&\\*
 $\sigma : -2$&&&$\sigma : -2$\\*\cline{1-1}\cline{4-4}
 &&&\\*
 $g_4:\ 1$&&&$g_4:\ 1$ \\*\cline{1-1}\cline{4-4}
 &&& \\*
 $\widetilde{g}_4 \geq 3$&&&$\widetilde{g}_4 \geq 2$ \\*\cline{1-1}\cline{4-4}
 &&& \\*
 $\widetilde{g}_4\leq 3$&&&$\widetilde{g}_4 \leq 3$ \\*\cline{1-1}\cline{4-4}
 &&& \\*
 See&&&$\star \to 4_1$ \\*
 note.\footnote{The lower bound $\widetilde{g}_4 \geq 3$ is obtained by Proposition \ref{prop:RH}, and the upper bound $\widetilde{g}_4 \leq g_3 = 3$ is obtained by Edmonds' theorem \cite{Edmonds}.}&&& \\*
 &&& \\ \hline\hline

 \multicolumn{2}{c||}{$11_{192}:$ transvergent, strongly invertible}&\multicolumn{2}{c}{$11_{195}:$ transvergent, strongly invertible}\\*\cline{1-4}
 &\multirow{12}{*}{\scalebox{.2}{\includegraphics{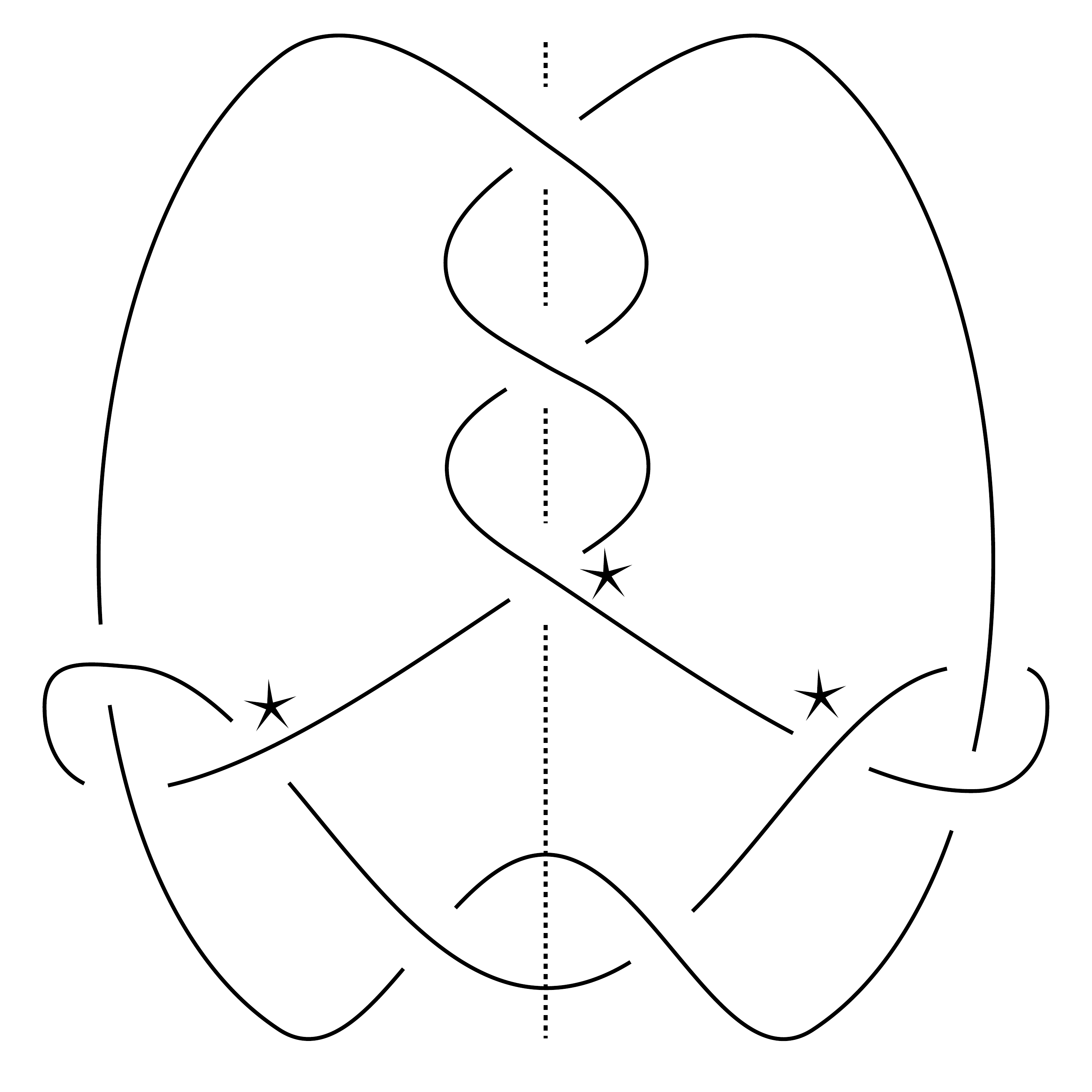}}}&\multirow{12}{*}{\scalebox{.2}{\includegraphics{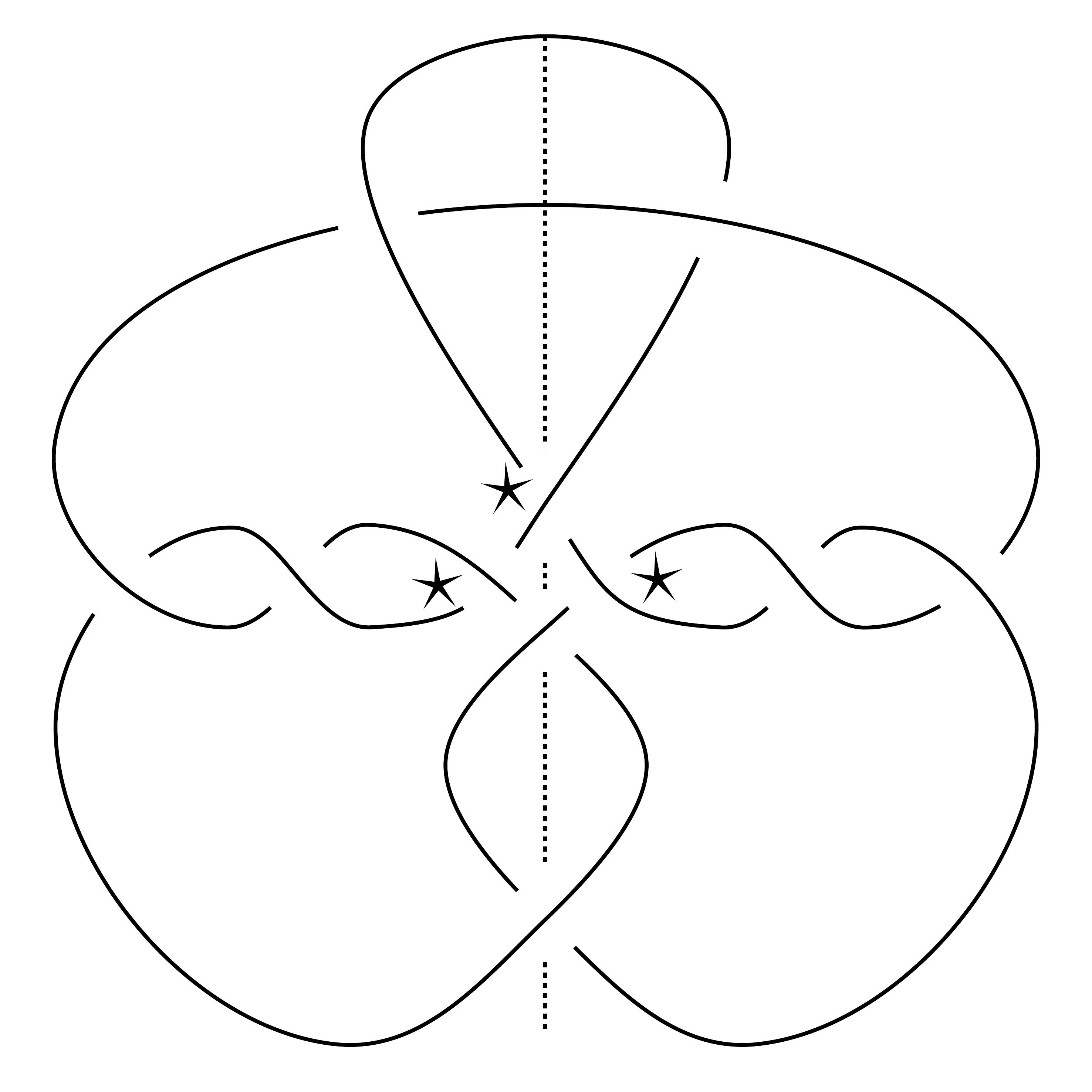}}}&\\*
 $\sigma : -2$&&&$\sigma : -2$\\*\cline{1-1}\cline{4-4}
 &&&\\*
 $g_4:\ 1$&&&$g_4:\ 1$ \\*\cline{1-1}\cline{4-4}
 &&& \\*
 $\widetilde{g}_4 \geq 2$&&&$\widetilde{g}_4 \geq 2$ \\*\cline{1-1}\cline{4-4}
 &&& \\*
 $\widetilde{g}_4\leq 3$&&&$\widetilde{g}_4 \leq 3$ \\*\cline{1-1}\cline{4-4}
 &&& \\*
 $\star \to$&&&$\star \to$ \\*
 unknot&&&unknot \\*
 &&& \\ \hline\hline

\multicolumn{2}{c||}{$11_{313}:$ transvergent, strongly invertible}&\multicolumn{2}{c}{$11_{360}:$ transvergent, strongly invertible}\\*\cline{1-4}
 &\multirow{12}{*}{\scalebox{.2}{\includegraphics{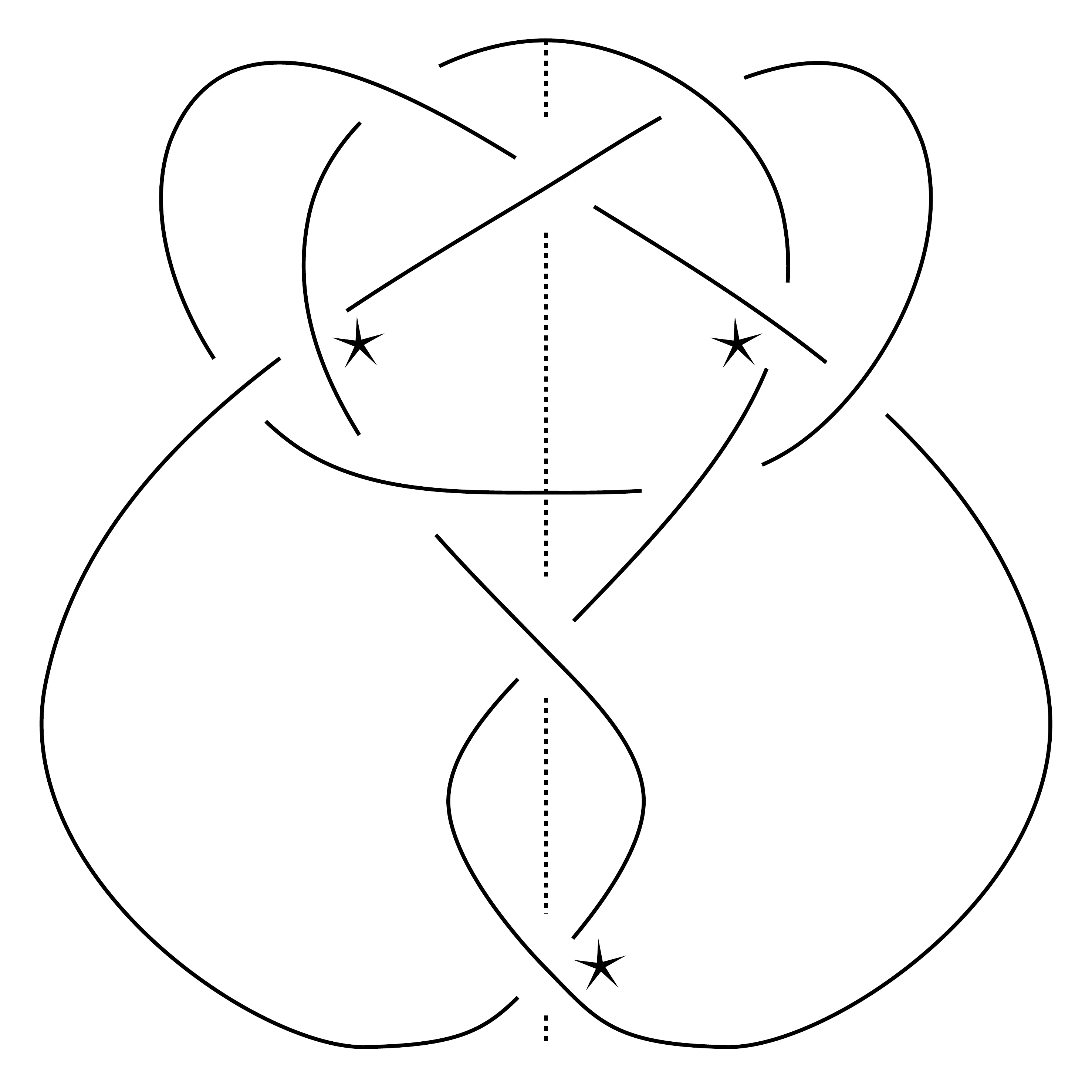}}}&\multirow{12}{*}{\scalebox{.2}{\includegraphics{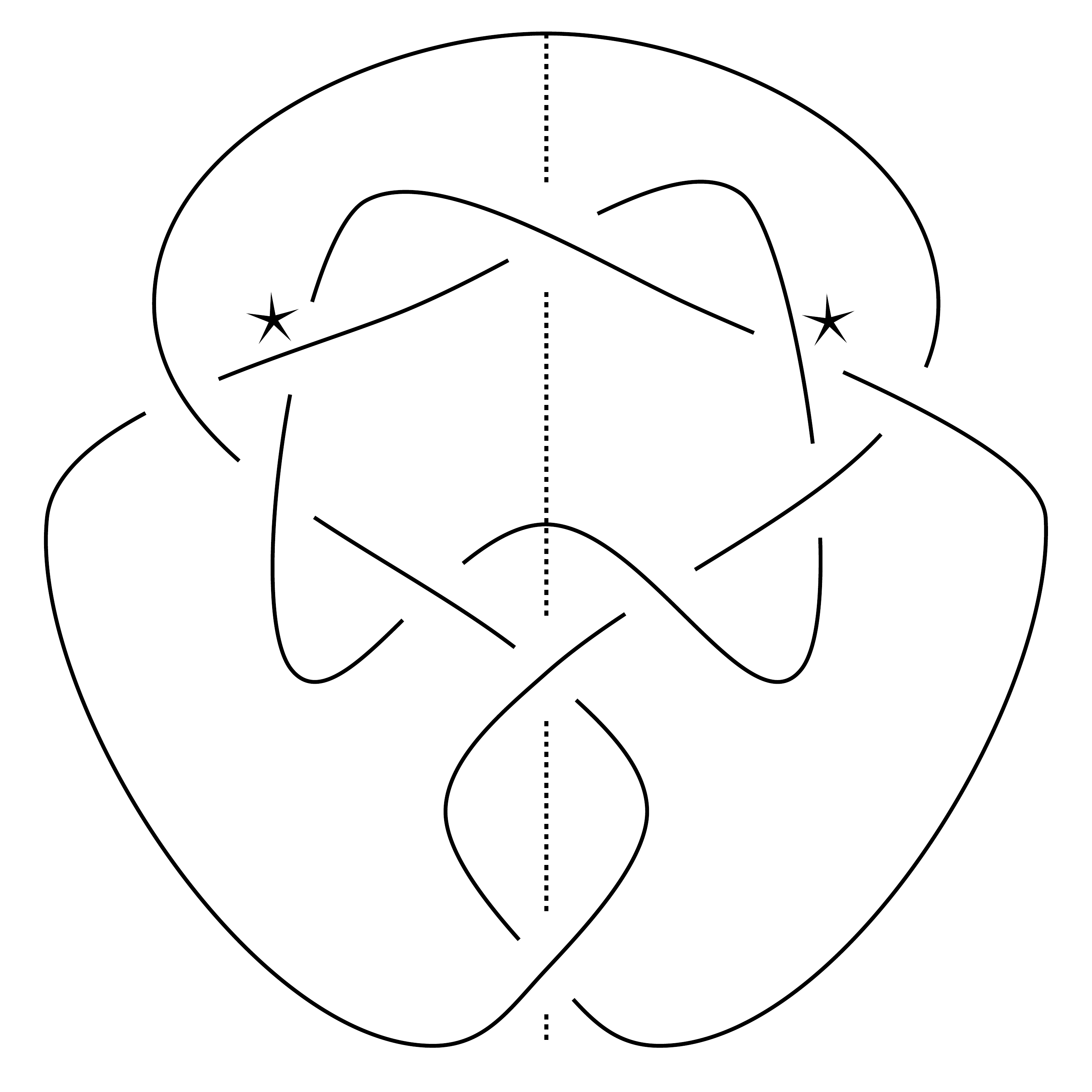}}}&\\*
 $\sigma : -2$&&&$\sigma : -2$\\*\cline{1-1}\cline{4-4}
 &&&\\*
 $g_4:\ 1$&&&$g_4:\ 1$ \\*\cline{1-1}\cline{4-4}
 &&& \\*
 $\widetilde{g}_4 \geq 2$&&&$\widetilde{g}_4 \geq 2$ \\*\cline{1-1}\cline{4-4}
 &&& \\*
 $\widetilde{g}_4\leq 3$&&&$\widetilde{g}_4 \leq 3$ \\*\cline{1-1}\cline{4-4}
 &&& \\*
 $\star \to$&&&$\star \to 5_2$ \\*
 unknot&&& \\*
 &&& \\ \hline\hline

\multicolumn{2}{c||}{$11_{366}:$ transvergent, strongly invertible}&\multicolumn{2}{c}{$11_{402}:$ transvergent, strongly invertible}\\*\cline{1-4}
 &\multirow{12}{*}{\scalebox{.2}{\includegraphics{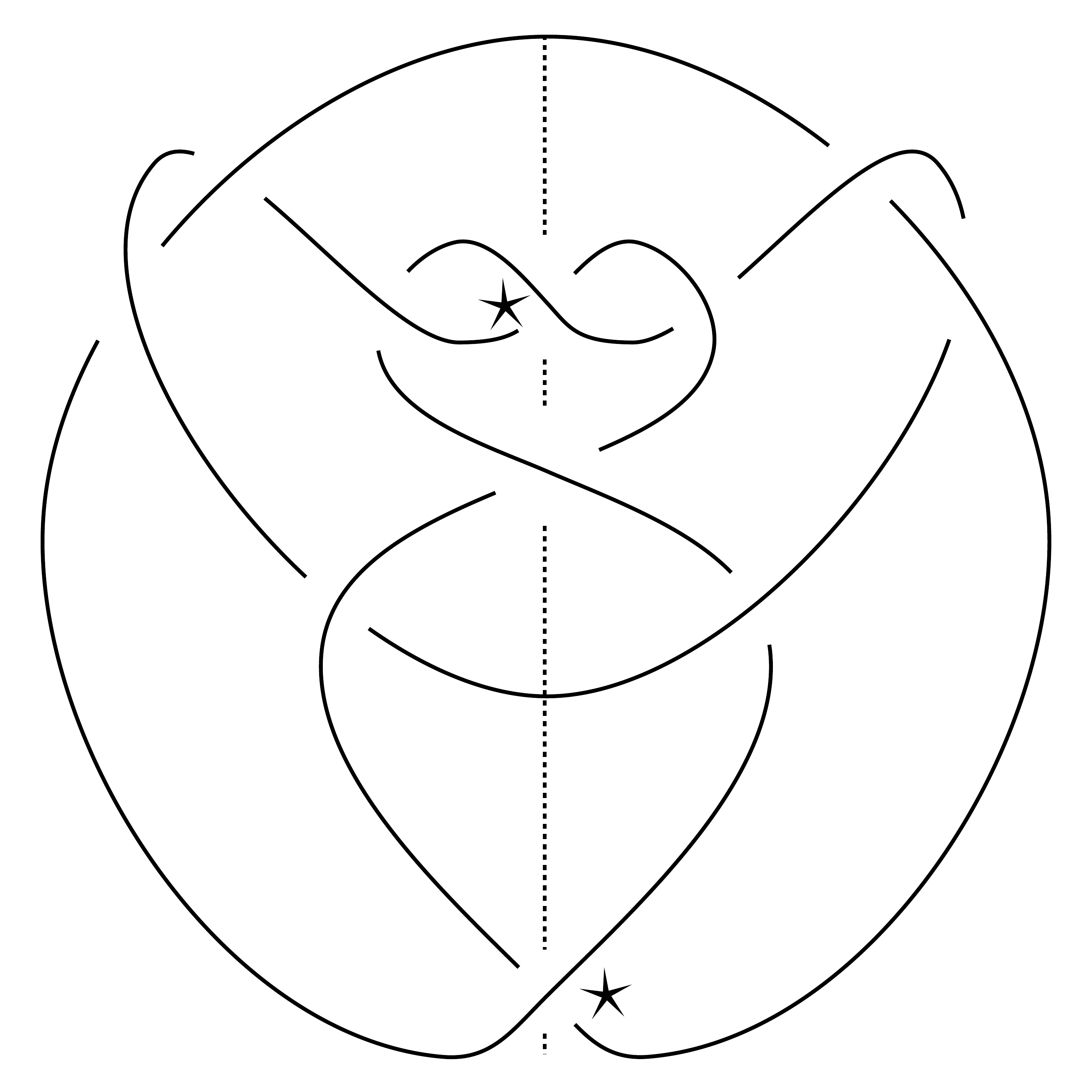}}}&\multirow{12}{*}{\scalebox{.2}{\includegraphics{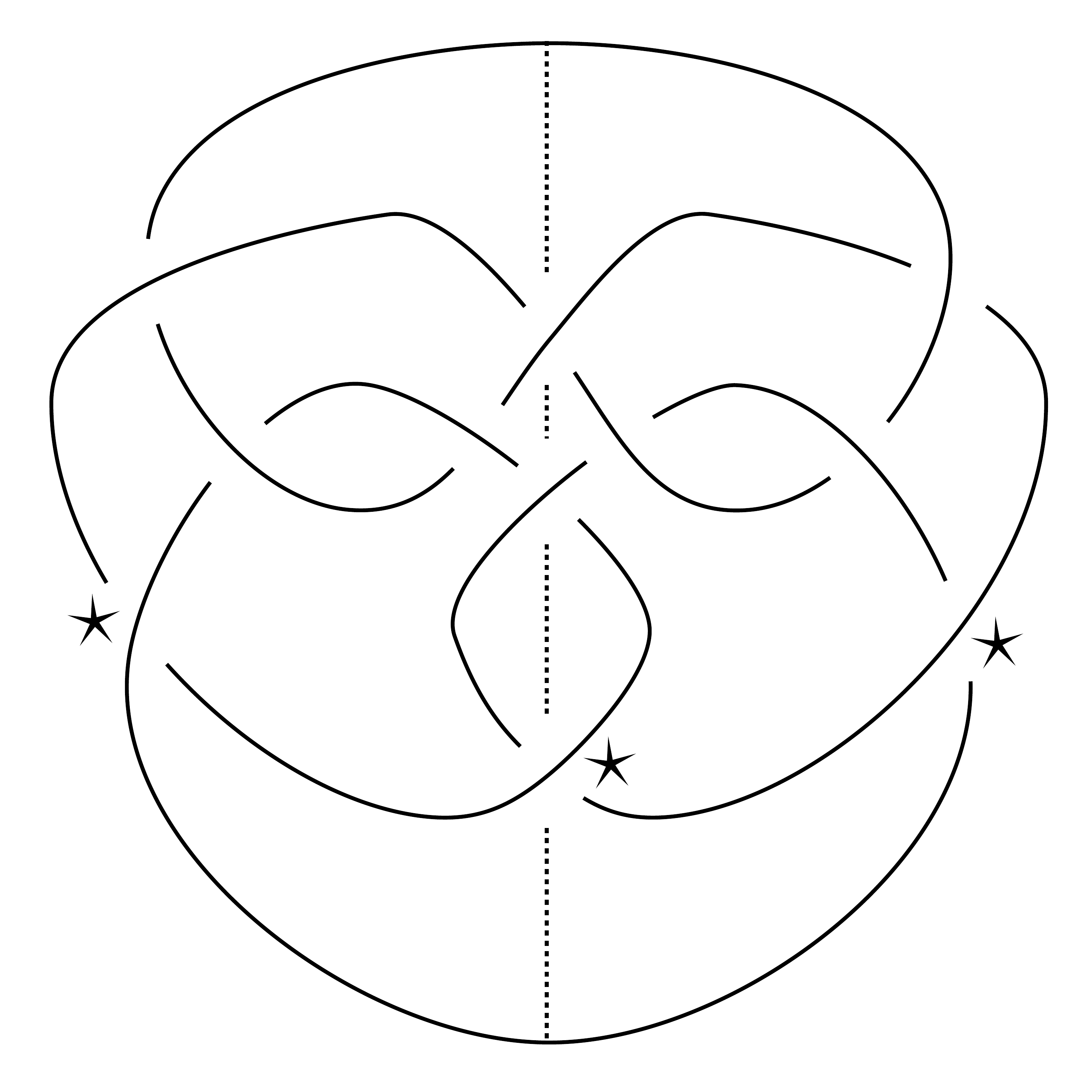}}}&\\*
 $\sigma : -2$&&&$\sigma : -2$\\*\cline{1-1}\cline{4-4}
 &&&\\*
 $g_4:\ 1$&&&$g_4:\ 1$ \\*\cline{1-1}\cline{4-4}
 &&& \\*
 $\widetilde{g}_4 \geq 2$&&&$\widetilde{g}_4 \geq 2$ \\*\cline{1-1}\cline{4-4}
 &&& \\*
 $\widetilde{g}_4\leq 3$&&&$\widetilde{g}_4 \leq 3$ \\*\cline{1-1}\cline{4-4}
 &&& \\*
 $\star \to 5_2$&&&$\star \to$ \\*
 &&&unknot \\*
 &&& \\ \hline\hline

\multicolumn{2}{c||}{$11_{443}:$ transvergent, strongly invertible}&\multicolumn{2}{c}{$11_{446}:$ transvergent, strongly invertible}\\*\cline{1-4}
 &\multirow{12}{*}{\scalebox{.2}{\includegraphics{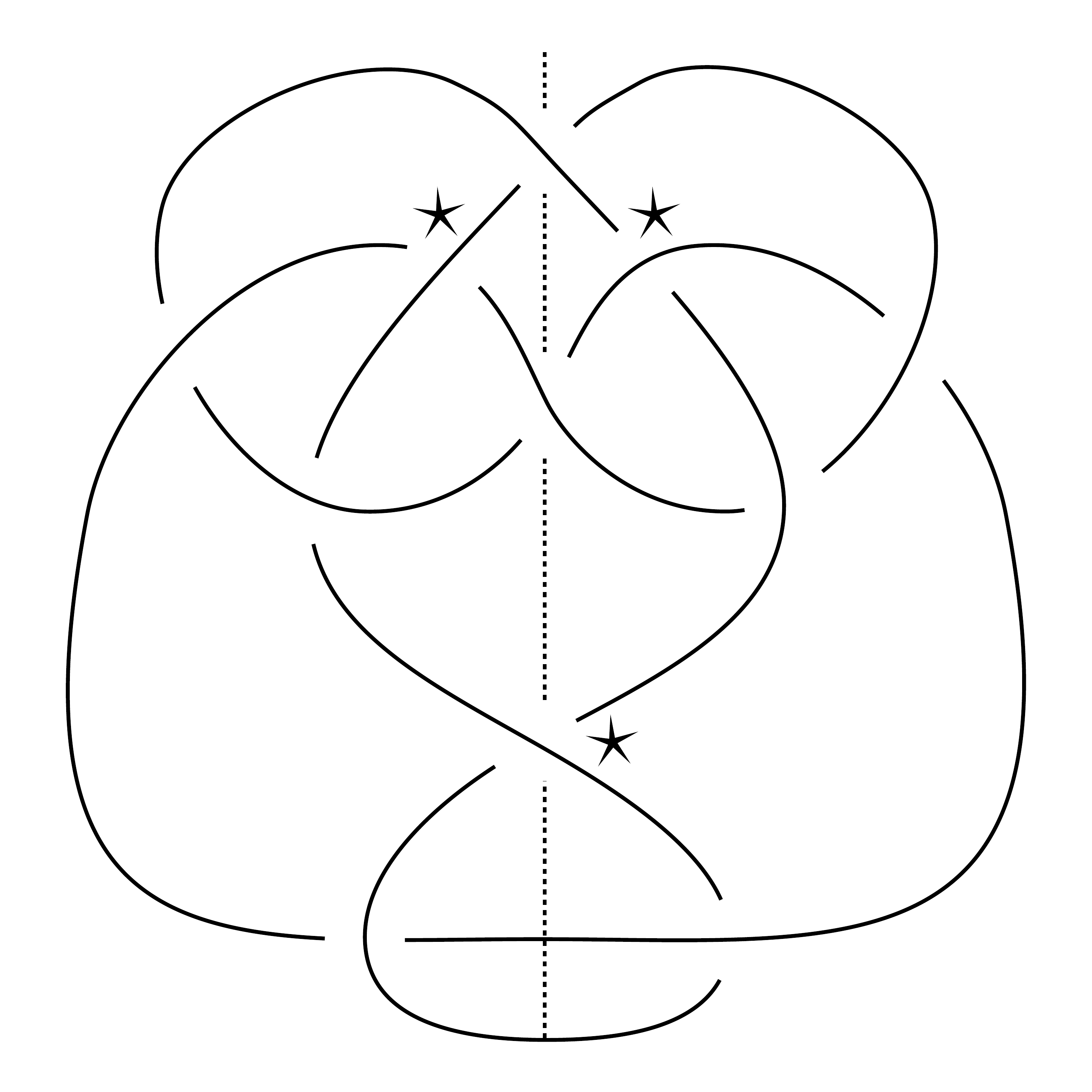}}}&\multirow{12}{*}{\scalebox{.2}{\includegraphics{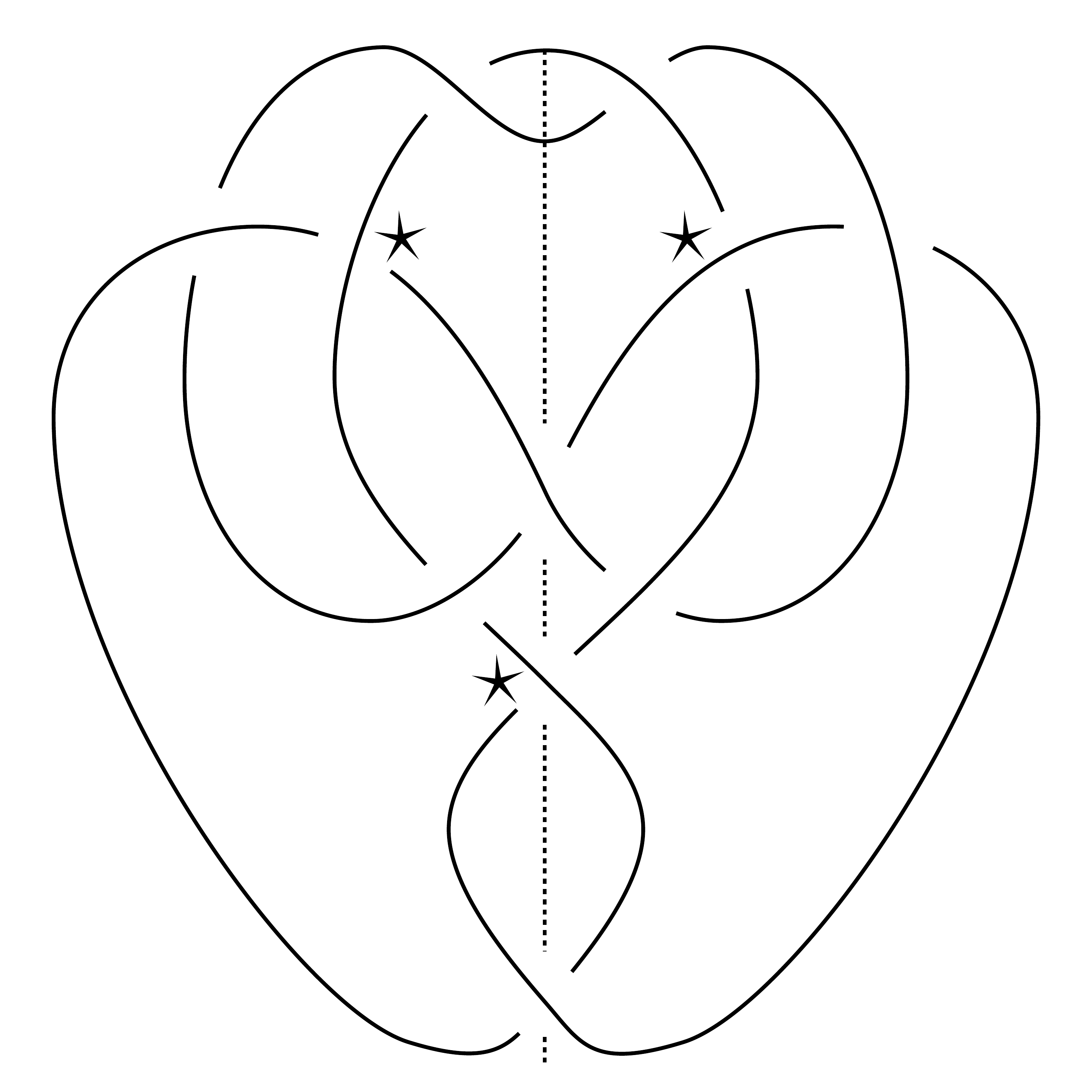}}}&\\*
 $\sigma : -2$&&&$\sigma : -2$\\*\cline{1-1}\cline{4-4}
 &&&\\*
 $g_4:\ 1$&&&$g_4:\ 1$ \\*\cline{1-1}\cline{4-4}
 &&& \\*
 $\widetilde{g}_4 \geq 2$&&&$\widetilde{g}_4 \geq 2$ \\*\cline{1-1}\cline{4-4}
 &&& \\*
 $\widetilde{g}_4\leq 3$&&&$\widetilde{g}_4 \leq 3$ \\*\cline{1-1}\cline{4-4}
 &&& \\*
 $\star \to$&&&$\star \to$ \\*
 unknot&&&unknot \\*
 &&& \\ \hline\hline

\multicolumn{2}{c||}{$11_{461}:$ intravergent, strongly invertible}&\multicolumn{2}{c}{$11_{471}:$ intravergent, strongly invertible}\\*\cline{1-4}
&\multirow{12}{*}{\scalebox{.2}{\includegraphics{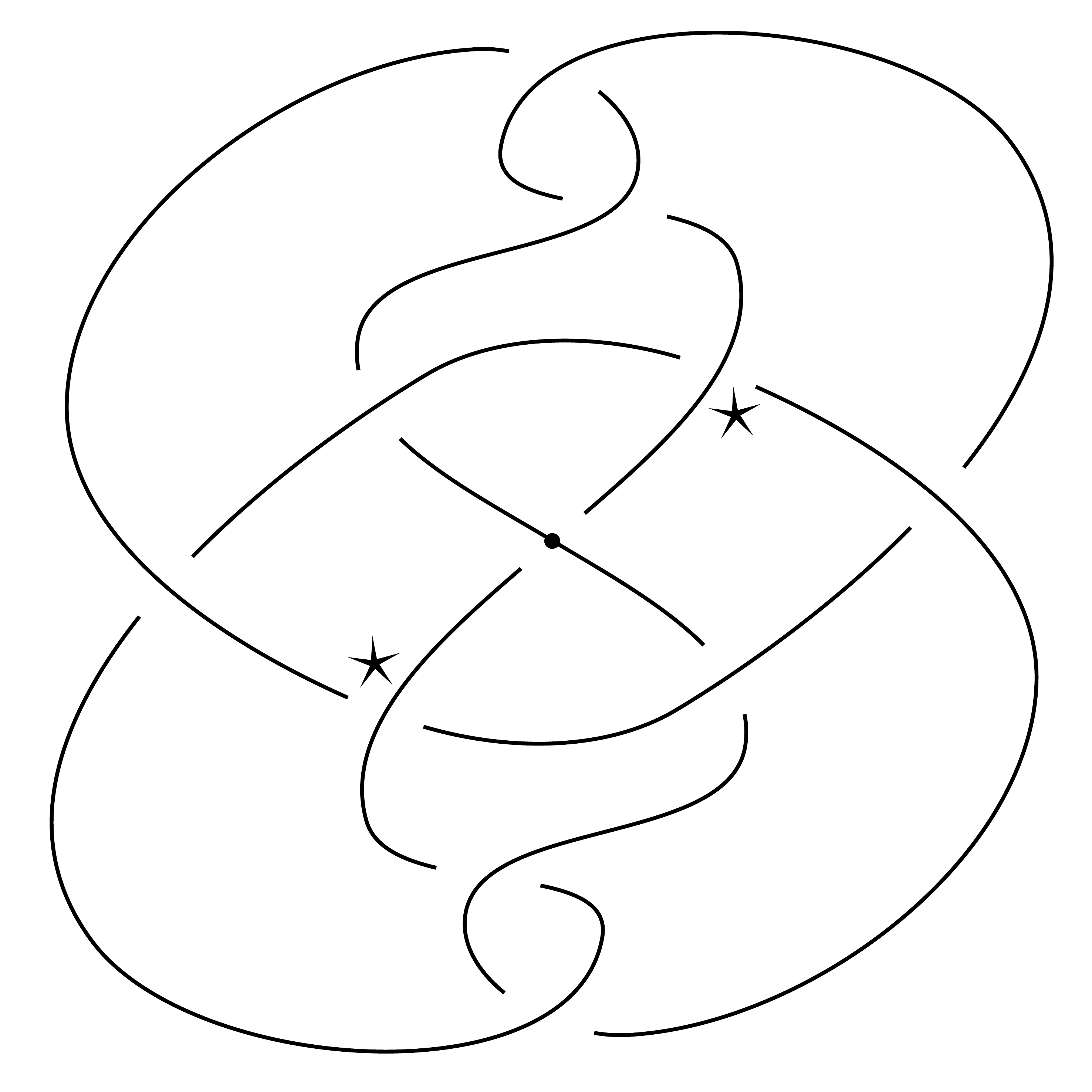}}}&\multirow{12}{*}{\scalebox{.2}{\includegraphics{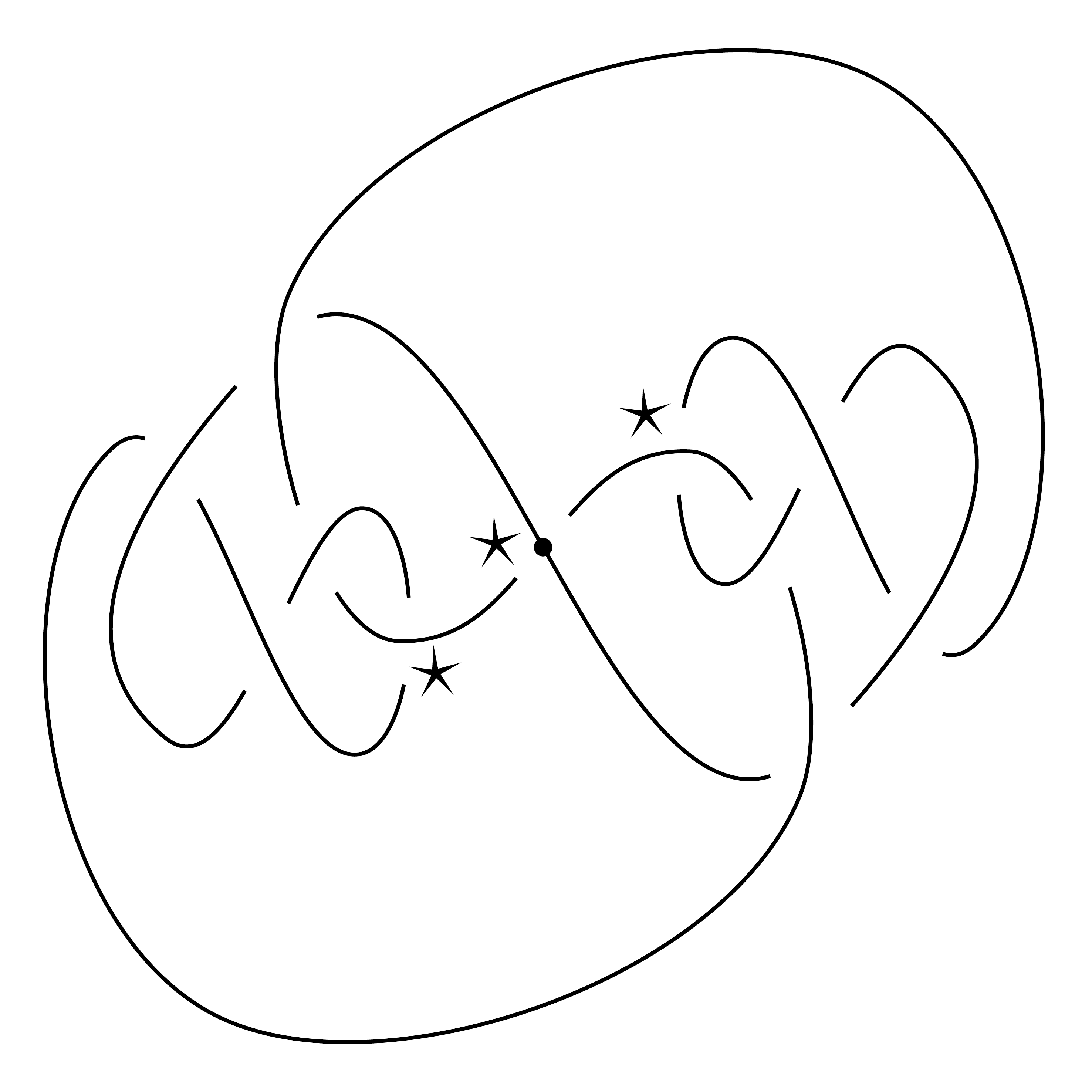}}}&\\*
 $\sigma : -2$&&&$\sigma : -2$\\*\cline{1-1}\cline{4-4}
 &&&\\*
 $g_4:\ 1$&&&$g_4:\ 1$ \\*\cline{1-1}\cline{4-4}
 &&& \\*
 $\widetilde{g}_4 \geq 2$&&&$\widetilde{g}_4 \geq 2$ \\*\cline{1-1}\cline{4-4}
 &&& \\*
 $\widetilde{g}_4\leq 3$&&&$\widetilde{g}_4 \leq 3$ \\*\cline{1-1}\cline{4-4}
 &&& \\*
 $\star \to 4_1$&&&$\star \to$ \\*
 &&&unknot \\*
 &&& \\\hline\hline
 \end{longtable}
\newpage
\begin{CJK}{UTF8}{min}
\bibliography{bibliography}
\bibliographystyle{alpha}
\end{CJK}
\end{document}